\documentclass[10pt]{amsart}
\usepackage[utf8]{inputenc}
\usepackage{amsfonts}
\usepackage{hyperref}

\hypersetup{
colorlinks = false,
hidelinks,
pdftitle={Analysis},
pdfsubject={Mathematics, Differential Geometry, Lie Groupoids, Symplectic Geometry},
pdfauthor={Clément Cren, Erfan Rezaei},
pdfkeywords={Differential Geometry, Lie Groupoids, Symplectic Geometry, arXiv, Preprint}
}
\usepackage{amsmath}
\usepackage{xcolor}
\usepackage{amsthm}
\usepackage{pdflscape}
\usepackage{pgfplots}
\usepackage{mathrsfs}
\usepackage{amssymb}
\usepackage{quiver}
\usepackage{faktor}
\usepackage[all]{xy}
\usepackage{ stmaryrd }
\usepackage{tikz}
\usetikzlibrary{
    positioning,
    decorations.pathreplacing,
    arrows.meta,
    arrows
}

\newtheorem{thm}{Theorem}[section]
\newtheorem{cor}[thm]{Corollary}
\newtheorem{prop}[thm]{Proposition}
\newtheorem{lem}[thm]{Lemma}

\theoremstyle{definition}
\newtheorem{defn}[thm]{Definition}

\newtheorem{exmp}[thm]{Example}

\theoremstyle{remark}
\newtheorem{rem}[thm]{Remark}

\newcommand{\Cont}{\operatorname{Cont}}
\newcommand{\Sz}{\operatorname{Sz}}
\newcommand{\FSz}{\mathcal{F}\mathrm{Sz}}

\newcommand{\ls}{\langle }
\newcommand{\rs}{\rangle}
\newcommand{\KK}{\mathrm{KK}}
\newcommand{\K}{\mathrm{K}}
\newcommand{\Op}{\mathrm{Op}}
\newcommand{\ev}{\mathrm{ev}}

\makeatletter
\let\c@equation\c@thm
\makeatother
\numberwithin{equation}{section}

\title[Families of Toeplitz operators]{Families of Toeplitz operators, family index and deformation quantization}
\author{ Clément Cren, Erfan Rezaei}
\pgfplotsset{compat=1.18}

\begin{document}

\begin{abstract}
Given a contact fibration, we construct smooth families of formal Szegö projections on the fibers. This allows us to define smooth families of Toeplitz operators. We apply these operators to construct a deformation quantization of prequantizable symplectic fibrations, recovering a result of Kravchenko in an analytic way. This also leads to a comparison between various forms of prequantizations and their geometric implications. We also derive a family index for these families of Toeplitz operators. To this end, we generalize an index formula of Baum and van Erp to families.
\end{abstract}

\maketitle

\tableofcontents
\newpage

\section{Introduction}
Classically, Toeplitz operators are given by compressions of multiplication operators  \(T_f:=S\mathcal{M}_fS:L^2(\mathbb{S}^1)\to L^2(\mathbb{S}^1) \), where the symbol \(f\) is a continuous function on \(\mathbb{S}^1\) and the operator \(S\) denotes the orthogonal projection onto the Hardy space \(H^2(\mathbb{S}^1)\). 
Boutet de Monvel generalized the notion of Szegö projection to strictly pseudoconvex CR manifolds in \cite{BoutetdeMonvelIndex} and to contact manifolds in \cite{BoutetdeMonvelGuillemin1981} with Guillemin.
These operators, previously defined using Fourier integral operators of Hermite type (or with complex phases \textit{à la} Melin-Sjöstrand \cite{MelinSjostrand1974}), were redefined by Epstein and Melrose using the Heisenberg calculus \cite{EpsteinMelrose1998,EpsteinMelroseUnpublished}, we will use this approach in the present paper.

Boutet de Monvel derived an index theorem for these generalized Toeplitz operators in \cite{BoutetdeMonvelIndex}, which can be formulated as follows:

\begin{thm}[Boutet de Monvel \cite{BoutetdeMonvelIndex}] \label{Theorem :Boutet de Monvel Theorem for contact manifolds}Let \(E\to M\) be a complex vector bundle over a closed contact manifold and \(f\in C^\infty(M,\mathrm{hom}(E))\), which is everywhere invertible. Let \(S\) denote a Szegö projection, acting on sections of \(E\). The Toeplitz operator \[T_f=S\mathcal{M}_fS\] is Fredholm with index \[\mathrm{ind}(T_f)=\mathrm{ind}(D_{E_f})=\ls \mathrm{Ch}(E_f)\smile \mathrm{Td}(M),[M] \rs\]
\end{thm}

Toeplitz operators were also used by Guillemin \cite{Guillemin1995} to construct a deformation quantization (i.e. a star-product) of symplectic manifolds through analytic means. He proved the following result:

\begin{thm}[Guillemin \cite{Guillemin1995}]\label{Theorem: Guillemin Theorem deformation quantization}
    Let \(M\) be a compact prequantizable symplectic manifold, \(X \to M\) the corresponding circle bundle, seen as a contact manifold. Let \(S\) be a Szegö projector on \(X\) which is equivariant for the \(\mathbb{S}^1\)-action and denote by \(\mathcal{T}^{\bullet}(X)\) the corresponding algebra of Toeplitz operators. There is an isomorphism of graded vector spaces:
    \[\faktor{\mathcal{T}^0(X)^{\mathbb{S}^1}}{\mathcal{T}^{-\infty}(X)^{\mathbb{S}^1}} \cong C^{\infty}(M)[[\hbar]].\]
    The product of Toeplitz operators on the left hand side induces a star-product on \(M\).
\end{thm}

In this paper, we wish to generalize both of these result for families of contact manifold. The proper setting for that is the one of contact fibrations in the sense of Lerman \cite{Lerman_2004}. We recall this concept in Section \ref{Section: Fibrations}. For our goal of deformation quantization, we also introduce a concept of prequantization for symplectic fibrations and relate them to contact fibration in the spirit of Boothby-Wang \cite{BoothbyWang1958}. 
This notion of prequantization is related to the underlying Poisson structure on the total space of a symplectic fibration. We thus compare in Section \ref{Section: Prequantization comparisons} our notion of prequantization to other existing notions of prequantizations of Poisson manifolds. Namely we compare it to the prequantization by a Jacobi manifold and the prequantization of associated groupoids (e.g. the existence of a contact Lie groupoid).
We then introduce the Heisenberg calculus for contact manifolds and contact fibrations in Section \ref{Section: Heisenberg Calculus}, following \cite{EpsteinMelroseUnpublished} and the more recent \cite{vanErpYuncken2019}. We then prove the existence of continuous families of formal Szegö projections in Section \ref{Section: Szego and Toeplitz families} and study the corresponding families of Toeplitz operators. We have to resort to formal Szegö projections (that is, projectors only up to smoothing operators) as it turns out that continuous families of Szegö projections may fail to exist. In Section \ref{Section: Star products} we construct a star product on prequantizable symplectic fibrations using Toeplitz operators. Finally in Section \ref{Section: Family index}, we tackle the index problem for families of Rockland (i.e. elliptic in the Heisenberg calculus) operators on a contact fibration. The index of families of Toeplitz operators is then a corollary of the general formula. Our proof is an adaptation of the proof of Baum and van Erp \cite{BaumvanErp2014} in the family setting. We also compute the Chern character of the index bundles.

\subsubsection*{Acknowledgements} The second author wants to thank Magnus Goffeng for fruitful discussions on contact families and K-homology. The first author is thankful to Jean-Marie Lescure for explanations on Szegö projections and families of operators on a fibration. Both authors are indebted to Ada Masters for finding an error in the previous version of this paper, which claimed that continuous families of Szegö projections always existed.

\section{Contact and symplectic fibrations} \label{Section: Fibrations}
In this section we recall the notion of contact fibrations of Lerman \cite{Lerman_2004} to fix the notation and  to set the stage and introduce a notion of prequantization for symplectic fibrations.


\begin{defn}
    A contact fibration is a fiber bundle \( X\overset{\pi}{\to} B\) with a co-oriented distribution \(\mathcal{H}\subset TX\) of \(\mathrm{cork}(\mathcal{H})=1\), such that for each fiber \(\pi^{-1}(b)=X_b\) the intersection  \(\mathcal{H}\cap T(X_b)=:H_b\) is a contact distribution on \(X_b\). We will denote contact fiber bundles as \((X\overset{\pi}{\to} B,\mathcal{H})\).
\end{defn}

\begin{rem}\label{Remark: Contact fibration cocycles}
    Equivalently one can define a contact fiber bundle as a fiber bundle \(M\hookrightarrow X\to B\), such that in every local trivialization \(\phi_i:\pi^{-1}(U)\to U\times M\) the fibers are given by a contact manifold \((M,H)\) and the cocycles \(\phi_{j,b}\circ \phi_{i,b}^{-1}: M \to M, b \in U_i\cap U_j\) are orientation preserving contactomorphisms. If we denote by \(\Cont(M,H)_+\) the group of co-orientation preserving contactomorphisms, we can, using these cocycles, associate to \(X\) a principal \(\Cont(M,H)_+\)-bundle. Conversely, to every principal \(\Cont(M,H)_+\)-bundle over \(B\) corresponds a contact fibration \(X \to B\).
\end{rem}

\begin{lem}[\cite{Lerman_2004}, Lemma 2.2]
    Let \((X\overset{\pi}{\to}B,\mathcal{H})\) be a contact fibration, then there exists a  \(1\)-form \(\Theta\in \Omega^1(X)\) such that \(\ker(\Theta)\cap\mathcal{V}={H}_\mathcal{V}\) i.e. \(\Theta|_{X_b}=:\theta_b\) restricts fiberwise to a contact form.
\end{lem}

In this particular example both the fibers and the total space admit contact structures, which are compatible. In (\cite{Lerman_2004}, Proposition 3.1) a condition is given, when the total space of a contact fibration is a contact manifold itself.

\begin{exmp}(Quaternionic Hopf Fibration)
    \((\mathrm{Sp}_1\to \mathbb{S}^{4n+3}\to \mathbb{P}^n(\mathbb{H}),\mathcal{H}_\xi)\). Endow \(\mathbb{S}^{4n+3}\subset\mathbb{H}^{n+1}\), with the round metric obtained from \(\mathbb{H}^{n+1}\). In this way \(\mathbb{S}^{4n+3}\) admits a natural \(3\)-Sasaki structure. Let \(\mathcal{V}:=\ker(\mathrm{d}\pi)\subset T\mathbb{S}^{4n+3}\) denote the vertical tangent bundle and define \(\mathcal{H}_{\mathrm{3-Sas}}:=\mathcal{V}^\perp\) to be the orthogonal complement, with respect to the round metric. Since \(\mathcal{V}\) is \(\mathrm{Sp}_1\) invariant, so is \(\mathcal{H}_{\text{3-Sas}}\). Hence we find a unique connection one form \(\omega\in \Omega^1(\mathbb{S}^{4n+1},\mathfrak{sp}_1)\), such that \[\mathcal{H}_\mathrm{3-Sas}=\ker\omega,\quad \omega(X^\vee)=X, \quad r^*_g\omega=\mathrm{Ad}_{g^{-1}}\circ\omega,
    \]
    for all fundamental vector fields \(X^\vee\) and \(g\in \mathrm{Sp}_1\). Any Sasaki structure in \(\mathbb{S}^{4n+3}\) can be realized as \(\eta_\xi:=\xi\circ \omega\in \Omega^1(\mathbb{S}^{4n+3})\) for \(\xi \in \mathfrak{sp}^*_1\setminus \{0\}\). We denote \(\mathcal{H}_\xi:=\ker{\eta_\xi}\). In particular this defines a contact structure in \(\mathbb{S}^{4n+3}\).

    We now want to construct a contact form on \(\mathbb{S}^3\cong \mathrm{Sp}_1\) from the Maurer-Cartan form \(\mu_{\mathrm{Sp}_1}\), which will be equivalent to the standard contact structure of \(\mathbb{S}^3\). We compose \(\eta_\xi:=\xi\circ \mu_{\mathrm{Sp}_1}\in \Omega^1(\mathrm{Sp}_1)\) and get a one form.
    From the Maurer-Cartan equation we get 
    \begin{align*}
        \mathrm{d}\eta_\xi(X,Y)&=\mathrm{d}(\xi\circ \mu_{\mathrm{Sp}_1})(X,Y)\\
                               &=\xi\circ \mathrm{d} \mu_{\mathrm{Sp}_1}(X,Y) \\
                               &=-\frac{1}{2}\xi([\mu_{\mathrm{Sp}_1},\mu_{\mathrm{Sp}_1}](X,Y))\\
                               &=-\xi([X,Y])\\
                               &=\mathrm{ad}^*_X\xi(Y)
    \end{align*}
    From the orbit decomposition of \(\mathfrak{su}^*_2\) we conclude that if \(\mathrm{d}\eta_\xi(X,Y)=0\) for all \(Y\), then \(X=0\). Hence, the \(2\)-form is nondegenerate.
    Under the (non-canonical) identification \[\Phi_p:G\to P_b \quad g\mapsto p\cdot g \] the connection one form \(\omega|_{P_b}\) pulls back to the Maurer-Cartan form \[\Phi^*_p(\omega|_{P_b})=\mu_{\mathrm{Sp}_1}\] Therefore the distribution \(\mathcal{H}_\xi\cap T(X_b)\) restricts precisely to the contact distribution \(H=\ker(\xi\circ \mu_{\mathrm{Sp}_1})\) under the identification \(\Phi_p\).
    \end{exmp}

\begin{exmp}(Prequantum \(\mathbb{S}^1\)-bundles over principally polarized Abelian varieties) As smooth manifolds principally polarized Abelian varieties are complex tori \(T_\tau\cong \faktor{\mathbb{C}^g}{\Lambda_\tau}, \tau \in M_g(\mathbb{C})\), but not every complex torus is a principally polarized Abelian variety. A complex torus is an Abelian variety, if its period matrix satisfies the Riemann bilinear relation. The Kodaira embedding theorem gives an equivalent condition, of existence of an ample line bundle. It turns out this ample line bundle is a prequantum line bundle, to which the associated principal \(\mathbb{S}^1\)-bundle inherits a natural contact structure. We denote the Siegel half space by \[\mathfrak{H}_g:=\{\tau\in M_{g}(\mathbb{C})|\tau=\tau^t, \mathfrak{Im}\tau>0\}.\]
There is a universal family of principally polarized abelian varieties. Consider \(\mathfrak{H}_g\times\mathbb{C}^g\) with the free, proper and holomorphic action \[\mathbb{Z}^{2g}\curvearrowright \mathfrak{H}_g\times\mathbb{C}^g \quad (m,n)\cdot(\tau,z):=(\tau,z+m+\tau n)\] We denote the quotient by \(\mathcal{X}_g:=(\mathfrak{H}_g\times\mathbb{C}^g)/\mathbb{Z}^{2g}\) and refer to it as the universal family. The quotient map is given by \[\pi: \mathcal{X}_g\to \mathfrak{H}_g \quad [t,z]\mapsto \tau\]
The fiber over \(\tau\) is \(\pi^{-1}(\tau)\cong \faktor{\mathbb{C}^g}{\mathbb{Z}^g+\tau\mathbb{Z}^g}=T_\tau\). Each fiber can be endowed with a Kähler form \(\omega_\tau:=\mathfrak{Im}H_\tau\), where \(H_\tau\) is the Hermitian form with respect to the almost complex structure induced from \(\tau\). The constructed bundle is a holomorphic family of symplectic manifolds. Moreover, since the Kähler classes represent principal polarizations, these fibers are prequantizable. We want to construct a smooth family of contact manifolds in a similar fashion. Since \(\mathfrak{H}_g\) is contractible, \(\mathcal{X}_g\cong\mathfrak{H}_g\times \mathbb{T}^{2g}\) is smoothly trivial as a fiber bundle. The Künneth theorem then gives \(\mathrm{H}^*(\mathcal{X}_g;\mathbb{Z})\cong \mathrm{H}^*(\mathbb{T}^{2g};\mathbb{Z})\), so every principal \(\mathbb{S}^1\)-bundle on \(\mathcal{X}_g\) is determined by a class \(\mathrm{H}^*(\mathbb{T}^{2g};\mathbb{Z})\). We pick the class of the standard symplectic form \(\omega\) on \(\mathbb{T}^{2g}\) and pull it back to \(\Omega:=\mathrm{pr}_2^*\omega\) on \(\mathcal{X}_g\). This class determines a principal \(\mathbb{S}^1\)-bundle on \(\mathcal{X}_g\). To describe this bundle more explicitly consider the trivial holomorphic line bundle \[\mathfrak{H}_g\times \mathbb{C}^g\times \mathbb{C}\to\mathfrak{H}_g\times \mathbb{C}^g \]
and define an action \[\mathbb{Z}^{2g}\curvearrowright \mathfrak{H}_g\times\mathbb{C}^g \times \mathbb{C}\quad (m,n)\cdot(\tau,z,\zeta):=(\tau,z+m+\tau n,\alpha_{(m,n)}\zeta)\]where \(\alpha_{(m,n)}(\tau,z):=\exp(\pi i (n^t\tau n+2n^t z))\) denotes the factor of automorphy. We define \[\mathcal{L}_g:=\faktor{\mathfrak{H}_g\times\mathbb{C}^g \times \mathbb{C}}{\mathbb{Z}^{2g}}\] with projection map \[p:\mathcal{L}_g\to \mathfrak{H}_g \quad [\tau,z,\zeta] \mapsto \tau\]
Each fiber over \(\tau\) is \(p^{-1}(\tau)=\mathcal{L}_g|_\tau\) is the total space of the prequantum line bundle over \(T_\tau\). There is a principal \(\mathbb{S}^1\)-bundle \(\mathcal{P}_g\to \mathfrak{H}_g\), such that every fiber \(P_\tau\) is  the total space of prequantum \(\mathbb{S}^1\)-bundle over \(T_\tau\). The Chern connection of \(\mathcal{L}_g\) defines an \(\mathbb{S}^1\)-invariant form \(\Theta\in \Omega^1(\mathcal{P}_g)\), such that \[\iota_R\Theta=1 \quad \mathcal{L}_R\Theta=0 \quad \mathrm{d}\Theta=p^*\Omega\]where \(R\in \mathfrak{X}(\mathcal{P}_g)\) denotes the generator of the \(\mathbb{S}^1\)-action.  Define \[\mathcal{H}:=\ker(\Theta)\subset T\mathcal{P}_g\] By construction \(\mathcal{H}\cap TP_\tau\) restricts to a contact distribution on each fiber, which coincides with the canonical contact distribution of the prequantum \(\mathbb{S}^1\)-bundle of the fiber.
\end{exmp}

The latter example can be abstracted to a symplectic fibration \(S\to B\) of prequantizable symplectic manifolds and its associated prequantum \(\mathbb{S}^1\)-bundle \(X\to B\).

\begin{defn}
 A symplectic fibration is a fiber bundle \(S\overset{p}{\to}B\) where each fiber is diffeomorphic to a symplectic manifold \((M,\omega)\) and the transition functions take values in the group of symplectomorphisms, i.e. \(\phi_{ij}:U_{ij}\to \mathrm{Sympl}(M,\omega)\).
\end{defn}

On a symplectic fibration, one can find \(2\)-forms \(\Omega \in \Omega^2(S)\) such that the restriction to each fiber is the symplectic form of the fiber. Such a form is called a symplectic connection if in addition
\[\forall v,w\in \mathcal{V :=}\ker(\mathrm{d}p), \iota_{v\wedge w} \mathrm{d}\Omega = 0. \]
According to \cite[Theorem 1.2.4]{GuilleminLermanSternberg}, if we take \(\mathcal{H}^{\Omega} = \mathcal{V}^{\perp_{\Omega}}\), it is an Ehresmann connection, and the previous condition is equivalent to the parallel transport it induces to be a symplectomorphism between the fibers.

\begin{defn}
        On a closed manifold \(X\) we call a vector field \(R\) regular, if the foliation generated by its flow is a fibration.
        Let \((X\to B,\mathcal{H})\) be a contact fibration. A defining \(1\)-form \(\Theta\in \Omega^1(X)\) is called regular, if the associated vertical Reeb vector field is regular. A pair \((X\to B,\Theta)\) is called a regular contact fibration.
\end{defn}

\begin{rem}
    Given a regular vector field \(R\) on a closed manifold \(X\), then all of its orbits are closed circles.
\end{rem}

Given a closed prequantizable symplectic manifold, i.e. a symplectic manifold \((M,\omega)\) such that \([\omega]\in \mathrm{im}(\mathrm{H}^2(M;\mathbb{Z})\to H_{dR}^2(M))\) there exists a canonical principal \(\mathbb{S}^1\)-bundle \(\mathbb{S}^1\to P\to M\) with a connection one form \(\theta\), whose curvature is \(\omega\). This \(\theta\) is a regular contact form on \(P\) and the Reeb vector field \(R_\theta\) generates the \(\mathbb{S}^1\) action. Given on the other hand such a closed regular contact manifold \((P,\theta)\) it is the total space of a principal \(\mathbb{S}^1\)- bundle over \(M\) and the quotient \(\mathbb{S}^1\backslash P\) is a closed symplectic manifold. This relation is often referred to as the Boothby-Wang construction \cite{BoothbyWang1958}. The next two propositions are a family version of the this relation.

\begin{prop}
\label{Proposition: Prequantization Symplectic fibration}
    Let \(S\overset{\pi_S}{\to} B\) be a symplectic fibration with a closed symplectic connection \(2\)-form \(\Omega\) and such that \([\frac{\Omega}{2\pi}]=e\in \mathrm{H}^2(S;\mathbb{Z})\) is an integral cohomology class. Then there exists a principal \(\mathbb{S}^1\)-bundle \(X\overset{p}{\to} S\) and a connection one form \(\Theta\), such that the total space is a regular contact fibration \((X\overset{\pi_X}{\to}B,\Theta)\).
\end{prop}

\begin{proof}
    The space \(\mathrm{H}^2(S;\mathbb{Z})\) classifies the principal \(\mathbb{S}^1\)-bundles over \(S\), hence there is a principal \(\mathbb{S}^1\)-bundle \(\mathbb{S}^1\to X\to S\), with \(c_1(X)=e\). Let \(\tilde\Theta\in \Omega^1(X)\) be a connection one form. Then \(\mathrm{d}\tilde{\Theta}=p^*\tilde\Omega\) for some \(\tilde{\Omega}\in \Omega^2(S)\) in the same cohomology class as \(\Omega\), hence there exists \(\lambda \in \Omega^1(S)\) such that \(\Omega-\tilde{\Omega}=\mathrm{d}\lambda\). We define a new connection \(1\)-form \[\Theta=\tilde{\Theta}+p^*\lambda\]
    The curvature of \(\Theta\) is given by 
    \[\mathrm{d}\Theta+\frac{1}{2}[\Theta,\Theta]=\mathrm{d}\Theta=\mathrm{d}\tilde{\Theta}+p^*\mathrm{d}\lambda=p^*\Omega\]
    We need to show that the total space \(X\) is indeed a contact fiber bundle over \(B\). We define the projection \(\pi_X=\pi_S\circ p: X\to B\). For each \(b\in B\) the fibers \(\pi^{-1}_X(\{b\})=p^{-1}(S_b)=:X_b\) is a principal-\(\mathbb{S}^1\)-bundle over \(S_b\) with bundle map \(p_b:=p|_{X_b}\). We denote the restriction of \(\Theta\) to \(X_b\) by \(\theta_b:=\Theta|_{X_b}\).The fibers are related by 
\[\begin{tikzcd}
	{X_b} & X \\
	{S_b} & S
	\arrow["{j_b}", hook, from=1-1, to=1-2]
	\arrow["{p_b}"', from=1-1, to=2-1]
	\arrow["p", from=1-2, to=2-2]
	\arrow["{i_b}", hook, from=2-1, to=2-2]
\end{tikzcd}\]
Therefore, we have \[\mathrm{d}\theta_b=j_b^*\mathrm{d}\Theta=j^*_b(p^*\Omega)=p_b^*(i_b^*\Omega)=p^*_b\omega_b\] This is precisely the fiberwise prequantization condition on \(S_b\). Moreover \(\dim X_b=\dim S_b+1\) and \(\theta_b\wedge(d\theta_b)^{\wedge n}=d\theta_b\wedge(p_b^*\omega_b)^{\wedge n}\) defines a volume form on the fiber \(X_b\), as \(p^*_b\omega_b\) is nowhere vanishing on the fiberwise horizontal distribution.
\end{proof}

\begin{rem}
    It is not always the case, that the one can find a closed \(2\)-form \(\Omega\) on the symplectic fibration \(S\to B\), which restricts to the fiberwise symplectic structures. The Hopf surface \(S=\mathbb{S}^3\times \mathbb{S}^1\) viewed as a principal \(\mathbb{T}^2\)-bundle \[\mathbb{T}^2\to \mathbb{S}^3\times \mathbb{S}^1\to \mathbb{S}^2\] can never admit such a form, as \(\mathrm{H}^2(\mathbb{S}^3\times \mathbb{S}^1)=\{0\}\). It is however a symplectic fibration for the constant standard symplectic structure on the fibers. Nevertheless, there always exists some \(2\)-form on \(S\), which is obtained by pulling back the fiberwise two forms and gluing them coherently as follows.
    Let \(\{U_i\}_{i\in I}\) be a locally trivializing cover of \(B\) and \(\{\rho_i\}_{i\in I}\) a partition of unity subordinate to it. Over such an open set we have:
    \[S|_{U}\overset{\phi_U}{\to}U\times M\overset{\mathrm{pr}_2}{\to M}\]
    Since \((M,\omega)\) is the symplectic model fiber and cocycles take values in  \(\mathrm{Sympl}(M,\omega)\), the pullback \(\Omega_U:=(\mathrm{pr}_2\circ\phi_U)^*\omega\in \Omega^2(S|_{U})\) is a well defined \(2\)-form which restricts fiberwise to a symplectic form. We now define \[\Omega:=\sum_{i\in I}\pi^*\rho_i \Omega_{U_i}.\]
    This form restricts on every fiber to \(\omega\) and is a symplectic connection.
\end{rem}

Conversely, assume that \(S\) admits a principal \(\mathbb{S}^1\)-bundle \(X \to S\) such that \(X\) is the total space of a contact fibration \((X\overset{\pi}{\to}B,\Theta)\), and such that the curvature of \(\Theta\) is \(\Omega\). It is then an easy consequence of Chern-Weil theory that the class \(\left[\frac{\Omega}{2\pi} \right] \in \mathrm{H}^2(S;\mathbb{Z})\). Therefore the two conditions of Proposition \ref{Proposition: Prequantization Symplectic fibration} are in fact equivalent.

\begin{defn}
    A symplectic fibration \(S\to B\) satisfying the conditions of Proposition \ref{Proposition: Prequantization Symplectic fibration} is called prequantizable. 
\end{defn}

\begin{rem}
    If \(S \to B\) is a prequantizable symplectic fibration for the symplectic form \(\Omega\), then since \(\iota_b^*\Omega = \omega_b\) then the fibers \((S_b,\omega_b)\) are prequantizable as symplectic manifolds.
\end{rem}
In \cite{GotayWeinstein} the authors briefly discuss the notion of prequantization of sympletic fibration in a remark.

\begin{prop}\label{Proposition: Prequantization Converse}
    Let \((X\to B,\Theta)\) be a regular contact fibration, such that all orbits of the Reeb flow have the same period \(2\pi\). Then \(\Theta\) is a connection \(1\)-form on the total space of \(S\) of a prequantizable symplectic fibration \((S\to B,\Omega)\), with curvature \(\mathrm{d}\Theta=p^*\Omega\).
\end{prop}

\begin{proof}
    Since all of the orbits have the same period and the vertical Reeb vector field \(R\) is nowhere vanishing, there is a free and proper \(\mathbb{S}^1\)-action on \(X\) induced by the Reeb-flow. Hence, \(X\) is the total space of a principal \(\mathbb{S}^1\)-bundle. The connection \(1\)-forms \(\Theta\) of principal \(\mathbb{S}^1\)-bundles are characterized by two properties: 
    \begin{itemize}
        \item invariance: \(\mathcal{L}_R\Theta=0\)
        \item normalization \(\iota_R\Theta=1\)
    \end{itemize}
    By Cartan's homotopy formula, this implies 
    \[0=\mathcal{L}_R\Theta=\mathrm{d}\iota_R\Theta+\iota_R\mathrm{d}\Theta, \text{ so }\iota_R\mathrm{d}\Theta=0, \text{ and } \mathcal{L}_R\mathrm{d}\Theta=0.\]
    Therefore \(\mathrm{d}\Theta\) induces a well defined \(2\)-form \(\Omega\) on the quotient \(S\cong \faktor{X}{\mathbb{S}^1}\), i.e. \(\mathrm{d}\Theta=p^*\Omega\). Since the Reeb flow is vertical, the projection \(\pi_X\) factorizes through \(p:X\to S\).
\[\begin{tikzcd}
	X && S \\
	& B
	\arrow["p", from=1-1, to=1-3]
	\arrow["{\pi_X}"', from=1-1, to=2-2]
	\arrow["{\pi_S}", from=1-3, to=2-2]
\end{tikzcd}\]

Therefore \(\pi_S\colon S\to B\) is a fibration and \(X_b\to S_b\) is a principal \(\mathbb{S}^1\)-bundle for each \(b\in B\). What is left to show, is that \(\Omega\) defines a fiberwise symplectic form. Restricting to a fiber, we have: 
\[p^*(\mathrm{d}\Omega|_{S_b})=\mathrm{d}(p^*\Omega|_{X_b})=\mathrm{d}\mathrm{d}\Theta|_{X_b}=0.\]
Since the differential \(p_*\) is surjective, the pullback \(p^*\) in injective. Hence, \(\mathrm{d}\Omega|_{S_b}\) is closed.
\end{proof}

\section{Comparison to other prequantizations}\label{Section: Prequantization comparisons}

There are several notions of prequantization of the total space \(S\) of symplectic fibrations, we want to compare in this section. The first is obviously the one introduced in the previous section with Proposition \ref{Proposition: Prequantization Symplectic fibration}. The other notions are tied to the Poisson structure of the total space of a symplectic fibration, which is given in Lemma \ref{Lemma: Symplectic fibration is Poisson}. In this section we introduce several notions of quantization and compare them with each other. In the end, we will use the previous one from Section \ref{Section: Szego and Toeplitz families} onward, as it ensures the existence of a Szegö projector. This section shows that this is however the strongest notion of quantization.

 \begin{lem}\label{Lemma: Symplectic fibration is Poisson}
    Let \(S\overset{p}\to B\) be a symplectic fibration with connected fibers. There is a natural Poisson structure \(\Pi\) on the total space \(S\) such that the symplectic leaves of \((S,\Pi)\) are the fibers of the fibration.
\end{lem}

\begin{proof}
    As \(S\overset{p}\to B\) is a symplectic fibration, there is a two form \(\omega\in \Gamma(S,\Lambda^2\mathcal{V}^*)\), such that for each \(b\in B\), the identification of \(TS_b\) and \(\mathcal{V}|_{S_b}\) identifies \(\omega|_{S_b}\) with the symplectic form \(\omega_b\). For any \(b\in B\) the symplectic form defines an isomorphism \(\omega_b^\flat:\mathcal{V}_b\to \mathcal{V}_b^*\). We denote its inverse by \(\pi_b^\sharp:=(\omega_b^\flat)^{-1}\) to obtain a nondegenerate Poisson structure \(\pi_b\) on the fiber \(S_b\). Let \((\cdot)^\mathcal{V}\) denote the restriction to vertical tangent space. We define a smooth bi-vector field \(\Pi\in \mathfrak{X}^2(S)\) by the formula: \[\Pi(\alpha\wedge\beta)|_x:=\pi_{p(x)}(\alpha^\mathcal{V}\wedge\beta^\mathcal{V})|_x \quad \mathrm{for} \quad \alpha\wedge \beta\in \bigwedge^2T_xS\] The associated bracket is given by \(\{f,g\}:=\Pi(\mathrm{d}f,\mathrm{d}g)\). To show that this defines indeed a Poisson structure, we need to show that \(\{\cdot,\cdot\}\) satisfies the Jacobi identity, or equivalently that \([ \Pi,\Pi]=0\in \mathfrak{X}^3(S)\). Since the equations are tensorial, it suffices to check that pointwise, with \(b = p(x)\): \[\{f,\{g,h\}\}(x)=\Pi(\mathrm{d}f,\Pi(\mathrm{d}g,\mathrm{d}h))(x)=\pi_b(\mathrm{d}f^\mathcal{V},\pi_b(\mathrm{d}g^\mathcal{V},\mathrm{d}h^\mathcal{V})^\mathcal{V})(x).\]
    The latter term is the restriction of the non-degenerate Poisson structure on the fiber \(S_b\) and therefore the Jacobiator vanishes.
    The symplectic foliation \(\mathcal{F}_\Pi\) of \(S\) is given by integral submanifolds of the characteristic distribution \(\mathcal{D}_\Pi|_x=\mathrm{im}(\Pi^\sharp_x:T^*_xS\to T_xS)\). By the construction of \(\Pi\) we immediately get the inclusion \(\mathcal{D}_\Pi|_x\subseteq\mathcal{V}_x\). To see the equality, we choose \(v\in \mathcal{V}_x\) and \(\iota_v\omega_b|_x \in \mathcal{V}_x^*\). Since the restriction \(T^*_xS\to \mathcal{V}^*_x\) is surjective, there exists \(\eta\in T^*_xS\) such that \(\eta|_{\mathcal{V}_x}=\iota_v\omega_b|_x\). By the definition of \(\Pi\) we get \[\Pi^\sharp(\eta)|_x=\pi_b^\sharp(\eta|_{\mathcal{V}_x})=\pi^\sharp_b(\iota_v\omega_b)=v\]

    Hence every \(v\in \mathcal{V}_x\) is in the image of \(\Pi^\sharp\) and \(\mathcal{D}_\Pi=\mathcal{V}\) as vector bundles. In particular every Hamiltonian vector field is tangent to the fibers. Since \(S\overset{p}\to B\) is a fibration each fiber \(S_b\) is an embedded submanifold and \(T_xS_b=\ker(\mathrm{d}p|_x)=\mathcal{V}_x\). Since the fibration has connected fibers, the leaves of \(\Pi\) are exactly the fibers of \(p\colon S\to B\).
\end{proof}

\begin{rem}
    In the case where the fibers of the fibration might be disconnected, then every connected component of a fiber gives a Poisson leaf. To avoid this, and some problems concerning basic differential forms (see below)
\end{rem}

\subsection{Prequantization by a Jacobi manifold}

Similarly as for symplectic fibrations, if \((X\overset{\pi}{\to}B, \Theta)\) is a contact fibration, then it is naturally a Jacobi manifold as shown in the next lemma. Jacobi structures are the analogue for contact manifolds of Poisson structures for symplectic manifolds \cite{Kirillov1976,Lichnerowicz1978Jacobi}. In this section, we show that a prequantization of a symplectic fibration by a Jacobi manifold in the sense of \cite{Chinea1996} is automatically a contact fibration. Our prequantization condition is thus equivalent to the one of \cite{Chinea1996}.

\begin{lem}\label{Lemma: Contact fibration is Jacobi}
    Let \(X\overset{\pi}\to B\) be a contact fibration with connected fibers. There is a natural Jacobi structure \(\lbrace\cdot,\cdot\rbrace\) on the total space \(X\) such that the characteristic foliation of \((X,\lbrace\cdot,\cdot\rbrace)\) has leaves which are the fibers of the fibration, in particular they are all contact.
\end{lem}

\begin{proof}
    Recall that the bracket \(\lbrace\cdot,\cdot\rbrace\) of a Jacobi manifold is completely determined by a contravariant 2-tensor \(\Lambda\) and a vector field \(R\) (with some compatibility properties). For the space \(X\), recall that we have a 1-form \(\Theta\) which restricts to each fiber to a contact form. In particular \(\mathrm{d}\Theta\) induces a symplectic structure on the vertical contact distributions. As in Lemma \ref{Lemma: Symplectic fibration is Poisson}, we may use this fiberwise symplectic structure to create a contravariant 2-tensor \(\Lambda\). The vector field \(R\) is the vertical Reeb field. To ensure these objects define a Jacobi structure we must check the relations:
    \[[\Lambda,\Lambda] = 2R\wedge \Lambda,\quad [\Lambda,R] = 0,\]
    where the bracket denotes the Schouten-Nijenhuis bracket. Both these relations can be checked on each fiber \(X_b\) where they reduce to the natural Jacobi structure associated to a contact manifold. The rest of the proof is similar to the previous Lemma \ref{Lemma: Symplectic fibration is Poisson}.
\end{proof}

Another particular case of Jacobi structure is given by Poisson structures. In that case the 2-tensor is the Poisson tensor and the vector field is equal to zero. Consequently, we may also view the total space of a symplectic fibration as a Jacobi manifold.

\begin{thm}\label{Theorem: Contact to Symplectic is Jacobi Map}
    Let \(S\overset{p}\to B\) be a prequantizable symplectic fibration and denote by \(X\overset{\pi}{\to}B\) the associated  contact fibration as in Proposition \ref{Proposition: Prequantization Symplectic fibration}. The map \(X \to S\) is a Jacobi map.
\end{thm}

\begin{proof}
    From their construction in Lemma \ref{Lemma: Symplectic fibration is Poisson} and \ref{Lemma: Contact fibration is Jacobi}, both brackets can be restricted fiberwise, and so can the map \(X \to S\) (from Proposition \ref{Proposition: Prequantization Symplectic fibration}). The Theorem then reduces to the fact that the maps \(X_b \to S_b\) are Jacobi maps which is well known, see e.g. \cite[Example 2.5]{Marle1991}.
\end{proof}

There is a prequantization procedure for Poisson manifolds, due to Chinea, Marrero, and de Le{\'o}n \cite{Chinea1996}, where the prequantization space admits the structure of a Jacobi manifold. We want to compare the prequantization of symplectic fibrations, with the one of the underlying Poisson manifold. 

Recall from the previous section that a symplectic fibration \(S\to B\) is prequantizable, if there exists a closed \(2\)-form \(\Omega\), such that \([\frac{\Omega}{2\pi}]\in \mathrm{H}^2(S,\mathbb{Z})\) and that pulls back to \(\omega_b \in \Omega^2(S_b)\) for each \(b\in B\).

Combining Propositions \ref{Proposition: Prequantization Symplectic fibration} and \ref{Proposition: Prequantization Converse}, we can summarize the statements.

\begin{prop}\label{Proposition: Equivalent Definitions Fibration QUantization}
    Let \((S\to B,\Omega)\) be a prequantizable symplectic fibration, then there exists a contact fibration \((X\to B,\Theta)\), such that \(p:X\to S\) is a principal \(\mathbb{S}^1\) bundle and \(\Theta\) is a connection \(1\)-form with curvature \(\mathrm{d}\Theta=p^*\Omega\). 
    Conversely, given a regular contact fibration \((X\to B,\Theta)\), then \(\mathrm{d}\Theta\) descends a closed \(2\)-form \(\Omega\) on the quotient \(\faktor{X}{\mathbb{S}^1}=S\)   and \((S\to B,\Omega)\) is a prequantizable symplectic fibration.
\end{prop}

An analogous definition and correspondence between Poisson manifolds and their prequantization spaces is given in \cite{Chinea1996}.

\begin{defn}
    A Poisson manifold \((M,\pi)\) is prequantizable, if there exists a closed, integral two form \(\Omega\in \Omega^2(M)\), such that \(\pi^\sharp[\Omega]=[\pi]\), where \(\pi^\sharp:\mathrm{H}^p_{dR}(M)\to \mathrm{H}^p_{LP}(M,\pi)\) denotes the map in cohomology induced by the bivector \(\pi.\) 
\end{defn}

\begin{rem}
    The spaces \(\mathrm{H}^p_{LP}(M,\pi)\) denote the Lichnerowicz-Poisson cohomology spaces. They are obtained from multivector-fields, with differential \(\mathrm{d}_{\pi} = -[\pi,\cdot]\). The map \(\pi^{\sharp} \colon \Omega^{\bullet}(M) \to \mathfrak{X}^{\bullet}(M)\) is a homomorphism of cochain complexes between the de Rham complex and the Lichnerowicz-Poisson complex, and thus induces maps in cohomology. The prequantization condition then translates to the existence of a closed 2-form \(\Omega \in \Omega^2(M)\) with \([\Omega]\in \mathrm{H}^2(M;\mathbb{Z})\), and a vector field \(Y \in \mathfrak{X}(M)\) such that
    \[\pi -\mathcal{L}_Y\pi = \pi^{\sharp}\Omega.\]
\end{rem}

\begin{prop}(\cite{Chinea1996}, Theorems 3.1 and 3.2)\label{Propotision: Equivalent Conditions Poisson QUantization}
    Let \((M,\pi)\) be a prequantizable Poisson manifold, then there exists a Jacobi manifold \((X,\Lambda,R)\) and a \(1\)-form \(\Theta\) satisfying \[\iota_R\Theta=1\quad \iota_R\mathrm{d}\Theta=0,\]
    such that \(p:X\to M\) is a principal \(\mathbb{S}^1\) bundle with principal connection \(\Theta\) and curvature \(\mathrm{d}\Theta=p^*\Omega\), where \(\Omega\) is the closed two form lifting \(\pi\) in cohomology.

    Conversely, given a regular Jacobi manifold \((X,\Lambda,R)\) with a given \(1\)-form \(\Theta\) satisfying \[\iota_R\Theta=1\quad \iota_R\mathrm{d}\Theta=0,\]
    the quotient \(M:=\faktor{X}{\mathbb{S}^1}\) is a prequantizable Poisson manifold, with the induced Poisson structure from \((\Lambda,R)\) by quotient.
\end{prop}

\begin{prop}\label{Proposition: Quantization of Fibration equals Quantization of Poisson structure. }
    The prequantization of a symplectic fibration \((S\to B,\Omega)\) to a contact fibration \((X\to B,\Theta)\) is equivalent to the prequantization of \((S,\Pi)\) to a Jacobi manifold \((X,\Lambda,R)\).
\end{prop}
\begin{proof}
    Suppose that \((S\to B,\Omega)\) is prequantizable by a contact fibration \((X\to B,\Theta)\), whose total space \(X\) is a principal \(\mathbb{S}^1\)-bundle over \(S\). By \ref{Lemma: Contact fibration is Jacobi} the total space of \(X\) is indeed a Jacobi manifold. Moreover, the compatible \(1\)-form satisfies \(\iota_R\Theta=1\) and \(\iota_R\mathrm{d}\Theta=0\). By proposition \ref{Propotision: Equivalent Conditions Poisson QUantization}, the quotient \(S\) is a Poisson manifold. Since \(\Lambda=\tilde\Pi+R\wedge Y\) is just the horizontal lift of \(\Pi\) with respect to \(\Theta\),  the induced Poisson structure is just \(\Pi\). 

    Conversely assume \((S,\Pi)\) is prequantizable as a Poisson manifold, then there is a regular Jacobi manifold \((X,\Lambda,R)\) and a \(1\)-form \(\Theta\), such that \(\iota_R\Theta=1\) and \(\iota_R\mathrm{d}\Theta=0\). Since the leaves of \((S,\Pi)\) are symplectic manifolds, the leaves of \((X,\Lambda,R)\) are contact manifolds. Moreover, since \(\pi_S:S\to B\) is a fibration, where the typical fiber is an even dimensional manifold \(M\), \(p\circ \pi_S:X\to B\) is a fibration, where the typical fiber is an \(\mathbb{S}^1\)-bundle over \(M\). By Proposition \ref{Proposition: Equivalent Definitions Fibration QUantization}, in order to show this, it suffices to show that \((X\to B,\Theta)\) is a contact fibration. This follows immediately, since the conditions \(\iota_R\Theta=1\) and \(\iota_R\mathrm{d}\Theta=0\) imply that \(i_b^*\Theta\wedge(i_b^*\mathrm{d}\Theta)^n\) is a volume form on \(X_b\) for all \(b\in B\). Hence, both prequantization conditions are equivalent.
\end{proof}



\subsection{Prequantization of associated groupoids}

For a smooth submersion \(S\to B\), the vertical tangent bundle \(\mathcal{V}\to S \) together with the inclusion and the bracket inherited from the vector fields defines a Lie algebroid. On a symplectic fibration, the symplectic \(2\)-forms on the fibers provide a Lie algebroid \(2\)-cocycle \(\omega_\mathcal{V}\in \Gamma^\infty(\bigwedge^2\mathcal{V}^*)\). The algebroid \(\mathcal{V}\) becomes a symplectic algebroid. This algebroid is always integrated by the fibered pairs groupoid \(S\times_BS\), on which we will have a special emphasis.
Associated to every Poisson manifold is the Poisson algebroid \((T^*S,[\cdot,\cdot]_\Pi ,\Pi^\sharp)\), where the bracket is defined by 
\[[\alpha,\beta]_\Pi:=\mathcal{L}_{\Pi^\sharp\alpha}\beta-\mathcal{L}_{\Pi^\sharp \beta}\alpha -\mathrm{d}(\Pi(\alpha,\beta)), \quad \alpha,\beta \in \Omega^1(S).\]
In a recent preprint \cite{Cren2026} the first author, Lescure and Mohsen studied symplectic algebroids and their central extensions. 
Specializing to our setting of symplectic fibrations, the Lie algebroid \(2\)-cocycle defines a central extension of Lie algebroids:

\[\begin{tikzcd}
	0 & {S\times\mathbb{R}} & {\mathcal{V}_\omega} & {\mathcal{V}} & 0
	\arrow[from=1-1, to=1-2]
	\arrow[hook, from=1-2, to=1-3]
	\arrow[two heads, from=1-3, to=1-4]
	\arrow[from=1-4, to=1-5]
\end{tikzcd}.\]

As a bundle \(\mathcal{V}_\omega\cong \mathcal{V}\oplus\underline{\mathbb{R}}_S\), and the Lie bracket is given by 
\[[(X,f),(Y,g)]:=([X,Y],\omega_\mathcal{V}(X,Y)+X(g) - Y(f)), \quad X,Y\in \Gamma(\mathcal{V}), f,g\in C^{\infty}(S).\] 
On \(\mathcal{V}_\omega\), we can define the \(1\)-form \(\theta\in \Omega^1(\mathcal{V}_\omega)\) by \(\theta(X,t)=-t\). This one form replaces the role of a contact form.
There is yet another notion of prequantization, due to Crainic.

\begin{defn}
    Let \(\mathcal{G}\) be a Lie groupoid, endowed with a closed multiplicative \(2\)-form \(\omega_{\mathcal{G}}\). A prequantization for \((\mathcal{G,\omega_\mathcal{G}})\) is a central extension of Lie groupoids:
\[\begin{tikzcd}
	0 & {\underline{\mathbb{S}^1}_M} & {\widetilde{\mathcal{G}}} & {\mathcal{G}} & 0
	\arrow[from=1-1, to=1-2]
	\arrow[hook, from=1-2, to=1-3]
	\arrow["p", two heads, from=1-3, to=1-4]
	\arrow[from=1-4, to=1-5]
\end{tikzcd}\]
and a multiplicative \(1\)-form \(\theta_\mathcal{G}\), such that \(\mathrm{d}\theta_{\mathcal{G}}=p^*\omega_{\mathcal{G}}\). Any such pair \((\widetilde{\mathcal{G}},\theta_\mathcal{G})\) is referred to as a groupoid prequantization for \((\mathcal{G,\omega_\mathcal{G}})\).
\end{defn}

We will write \(\mathcal{G}(\mathcal{A})\) for the source connected, simply-connected Lie groupoid integrating a Lie algebroid \(\mathcal{A}\) (if it exists).

The algebroid \(\mathcal{V}\) is integrable by the groupoid \(S\times_BS\) in particular, \(\mathcal{G}(\mathcal{V})\) exists. 

The algebroid \(T^*S\) turns out to always be integrable. We will write \(\omega_{\mathcal{G}(T^*S)} \in \Omega^2(\mathcal{G}(T^*S))\) for its canonical symplectic form. We will prove below that choosing a symplectic connection form \(\Omega\) induces an isomorphism of groupoids
\[\mathcal{G}(T^*S) \cong \mathcal{G}(\mathcal{V}) \times_S p^*(T^*B),\]
where \(T^*B \to B\) is a bundle of abelian groups. Under this isomorphism we have
\[\omega_{\mathcal{G}(T^*S)} = \omega_{\Omega}\oplus p^*\mathrm{d}\lambda\]
with \(\lambda \in \Omega^1(T^*B)\) being the Liouville form, and \(\omega_{\Omega}\) is a multiplicative \(2\)-form on \(\mathcal{G}(\mathcal{V})\). This proves that \((\mathcal{G}(T^*S),\omega_{\mathcal{G}(T^*S)})\)  is prequantizable for if and only if \((\mathcal{G}(\mathcal{V}),\omega_{\Omega})\) is prequantizable. We will show that these equivalent conditions are implied by our notion of prequantization by a contact fibration. The isomorphism of groupoids above is not very explicit. It is induces by an isomorphism of Lie algebroids, which is itself more explicit (see Proposition \ref{Proposition: Algebroid splitting} below). Because of this, the condition of \(\omega_{\Omega}\) being integral doesn't appear any more computable than the same question for \(\omega_{\mathcal{G}(T^*S)}\). A reasonable guess for an expression of \(\omega_{\Omega}\) would be \(s^*\Omega - r^*\Omega\). The problem is that this latter form is not closed in general, while \(\omega_{\Omega}\) is automatically closed by construction. Despite this issue, the two forms do not differ by much, having them coincide is equivalent to \(\omega_{\Omega}\) being the pullback of a form on the pair groupoid \(S\times_B S\). We derive in Subsection \ref{Subsection: fibered pairs} geometric conditions on the symplectic fibration to ensure this phenomenon.

We first focus on the aforementioned splitting, at the algebroid level.

\begin{prop}\label{Proposition: Algebroid splitting}
    Let \(S\to B\) be a symplectic fibration and \((S,\Pi)\) the canonical Poisson structure on the total space, and \((T^*S,[\cdot,\cdot]_\Pi ,\Pi^\sharp)\) the associated Lie algebroid. Any choice of symplectic connection induces a splitting of \(T^*S\) as a direct sum of algebroids:
    \[T^*S \cong \mathcal{V} \oplus p^*T^*B\]
    where \(\mathcal{V}\) is the same Lie algebroid as before and \(p^*T^*B\) is a bundle of abelian Lie algebras.
\end{prop}

\begin{proof}
There is a natural short exact sequence of vector bundles
\[0 \to p^*T^*B\to T^*S\overset{\Pi^\sharp}{\to}\mathcal{V}\to 0.\]
We want to show, that this sequence is a indeed a sequence of Lie algebroids and that the sequence splits as such. To do so, we first show that \(p^*T^*B\subset T^*S\) is a central ideal.
Let \(\eta,\nu\in \Omega^1(B)\), then 
\[\Pi^\sharp(p^*\eta)=\omega_{\mathcal{V}}^{\flat}(p^*\eta|_\mathcal{V})=0,\]
hence, \[[\pi^*\eta,\pi^*\nu]_\Pi=0.\]
Let \(\alpha\in \Omega^1(S)\), then 
\begin{align*}
    [\pi^*\eta,\alpha]_\Pi&=\mathcal{L}_{\Pi^\sharp \pi^*\eta}\alpha-\mathcal{L}_{\Pi^\sharp\alpha}\pi^*\eta-\mathrm{d}(\Pi(\pi^*\eta,\alpha))\\
    &=-\mathcal{L}_{\Pi^\sharp\alpha}\pi^*\eta\\
    &=0,
\end{align*}
which vanishes, as \(\Pi^\sharp\alpha\in \Gamma^\infty(\mathcal{V})\). Therefore \(\pi^*T^*B\) is a central ideal in \(T^*S\).
We fix \(\Omega \in \Omega^2(S)\) as symplectic connection form. The map \(\sigma\colon \mathcal{V}\to T^*S\quad v\mapsto \iota_v\Omega\) defines a bundle map. The composition with the anchor is the identity, as 
\[(\Pi^\sharp\circ\sigma)(v)=\Pi^\sharp(\iota_v\Omega)=\omega_{\mathcal{V}}^{\flat}(\iota_v\Omega|_\mathcal{V})=\omega_{\mathcal{V}}^\flat(\omega_{\mathcal{V}}^\sharp(v))=v,\]
for all \(v\in \mathcal{V}\). Hence, \(\sigma\) induces a splitting of vector bundles.
\[\begin{tikzcd}
	0 & {p^*T^*B} & {T^*S} & {\mathcal{V}} & 0
	\arrow[from=1-1, to=1-2]
	\arrow[from=1-2, to=1-3]
	\arrow["{\Pi^\sharp}", from=1-3, to=1-4]
	\arrow["\sigma"', shift right=2, curve={height=12pt}, from=1-4, to=1-3]
	\arrow[from=1-4, to=1-5]
\end{tikzcd}\]
The map \(\sigma\) is indeed a morphism of Lie algebroids, as one computes the compatibility with the brackets.
\begin{align*}
    [\sigma(X),\sigma(Y)]_{\Pi}&=\mathcal{L}_{\Pi^\sharp\sigma(X)}\sigma(Y)-\mathcal{L}_{\Pi^\sharp\sigma(Y)}\sigma(X)-\mathrm{d}(\Pi(\sigma(X),\sigma(Y)))\\
    &=\mathcal{L}_X\iota_Y\Omega-\mathcal{L}_Y\iota_X\Omega-\mathrm{d}(\Pi(\iota_X\Omega,\iota_Y\Omega))
\end{align*}

The latter term reduces to \[\Pi(\iota_X\Omega,\iota_Y\Omega)=\langle \Pi^\sharp(\iota_X\Omega),\iota_Y\Omega \rangle=\langle X,\iota_Y\Omega\rangle=\iota_X\iota_Y\Omega=\Omega(Y,X).\]
Using the Cartan formulas for the de Rham complex, in particular \[\mathcal{L}_X=\mathrm{d}\iota_X+\iota_X\mathrm{d}, \quad \iota_{[X,Y]}=\mathcal{L}_X\iota_Y-\iota_Y\mathcal{L}_X,\]
we obtain for any \(2\)-form \(\Omega\in \Omega^2(S)\) the following equality.

\begin{align*}
    \iota_Y\iota_X\mathrm{d}\Omega&=\iota_Y(\mathcal{L}_X\Omega-\mathrm{d}\iota_X\Omega)\\
    &=\iota_Y\mathcal{L}_X\Omega-\iota_Y \mathrm{d}(\iota_X\Omega)\\
    &=\mathcal{L}_X\iota_Y\Omega-\iota_{[X,Y]}\Omega-(\mathcal{L}_Y\iota_X\Omega-\mathrm{d}(\iota_Y\iota_X\Omega))\\
    &=\mathcal{L}_X\iota_Y\Omega-\mathcal{L}_Y\iota_X\Omega+\mathrm{d}(\Omega(X,Y))-\iota_{[X,Y]}\Omega\\
\end{align*}

Since any symplectic connection of a symplectic fibration \(S\to B\) is horizontally closed, i.e. \(\iota_Y\iota_X\mathrm{d}\Omega=0\) for \(X,Y\in \Gamma^\infty(\mathcal{V})\), we obtain \[[\sigma(X),\sigma(Y)]_\Pi=\sigma([X,Y]).\]
We now define a trivial extension of algebroids
\[\begin{tikzcd}
	0 & {p^*T^*B} & {\mathcal{A}} & {\mathcal{V}} & 0
	\arrow[from=1-1, to=1-2]
	\arrow[from=1-2, to=1-3]
	\arrow[from=1-3, to=1-4]
	\arrow[from=1-4, to=1-5]
\end{tikzcd}\]
where \(\mathcal{A}=p^*T^*B\oplus\mathcal{V}\) is the vector bundle underlying the algebroid with bracket \[[(p^*\eta,X),(p^*\nu,Y)]_\mathcal{A}:=(0,[X,Y])\] and anchor map \[\rho_\mathcal{A}(p^*\eta,X)=\sigma(X).\]
Let \(\Phi\colon\mathcal{A}\to T^*S \quad (p^*\eta,X)\mapsto p^*\eta+\sigma(X)\), it is an isomorphism of vector bundles. Its inverse is given by \(\Phi^{-1}\colon T^*S\to \mathcal{A}\quad \alpha\mapsto (\alpha-\iota_{\Pi^\sharp\alpha}\Omega,\Pi^\sharp(\alpha))\).
The map \(\Phi\) is indeed a morphism of algebroids, as
\begin{align*}
    [\Phi(p^*\eta,X),\Phi^*(p^*\nu,Y)]_\Pi&=[p^*\eta+\sigma(X),p^*\nu+\sigma(Y)]_\Pi\\
    &=[\sigma(X),\sigma(Y)]_\Pi\\
    &=\sigma([X,Y])\\
    &=\Phi(0,[X,Y])\\
    &=\Phi([(p^*\eta,X),(p^*\nu,Y)]_{\mathcal{A}}).
\end{align*}
Therefore the algebroid \(T^*S\) is isomorphic to the direct sum of algebroids \(\mathcal{A} = p^*T^*B\oplus \mathcal{V}\).
\end{proof}

\begin{cor}
    The algebroid \(T^*S\) is integrable.
\end{cor}
\begin{proof}
Since \(T^*S\) is isomorphic to \(\mathcal{A}\), we only have to show that \(\mathcal{A}\) is integrable. This is then a consequence of the fact that \(\mathcal{V}\) is integrated by \(S\times_BS\) and \(p^*T^*B\) by \(p^*T^*B\) (now seen as a bundle of abelian groups). We show by hand that their fibered product does integrate \(\mathcal{A}\). We define the groupoid \[\mathcal{G}:=(S\times_B S)\times_BT^*B\rightrightarrows S\] with structure maps:
\[s(y,x,\xi)=x \quad r(y,x,\xi)=y\quad m((z,y,\eta),(y,x,\xi))=(z,x,\eta+\xi)\]
\[e(x)=(x,x,0_{p(x)}) \quad i(y,x,\xi)=(x,y,-\xi)\]
and its tangent bundle groupoid: \[T\mathcal{G}=T(S\times_BS)\times_{TB}TT^*B.\]
As a sets, we can identify: 
\begin{align*}
    \mathcal{G}&=\{(y,x,\xi)\in S\times S \times T^*B:p(y)=p(x)=\pi_{T^*B}(\xi)\},\\
    T\mathcal{G}&=\{(Y,X,\Xi)\in TS\times TS \times TT^*B:Dp(Y)=Dp(X)=D\pi_{T^*B}(\Xi)\},
\end{align*}
with \(\pi_{T^*B} \colon T^*B\to B\) being the projection map. The groupoid \(\mathcal{G}\) integrates \(\mathcal{A}\), i.e. its algebroid \(\mathcal{A}(\mathcal{G}):=\mathrm{Lie}(\mathcal{G})\) is isomorphic to \(\mathcal{A}\). Recall, that the Lie algebroid of a groupoid is defined by \(\mathcal{A}(\mathcal{G}):=\mathrm{ker}(Ds)|_{e(S)}\) and its anchor map by \(\rho:=Dr|_{e(S)}\colon\mathcal{A}(\mathcal{G}):=\mathrm{ker}(Ds)|_{e(S)}\to S\). Let \(x\in S\) and consider \[T_{e(x)}\mathcal{G}\cong\{(Y,X,\Xi)\in T_xS\oplus T_xS\oplus T_{0_{p(x)}}T^*B\}.\] The tangent of \(T^*B\) at the zero section splits, so we can identify elements of \(T_{0_{b}}T^*B\) as \(\Xi=(v,\eta)\in T_bB\oplus T_b^*B\). The differential of the source map acts as \(Ds|_{e(x)}(Y,X,v,\eta)=X\), hence \[\mathrm{ker}(Ds|_{e(x)})=\{(Y,0,v,\eta)\mid D_xp(Y)=0=v\}\cong \mathcal{V}_x\oplus T^*_bB, \quad b=p(x).\]
The anchor map is given by the differential of range map \[Dr|_{e(S)}=\rho:\mathcal{A}(\mathcal{G})\to TS \quad (Y,\eta)\mapsto Y.\] Since \(Dp(Y)=0\), then the image of \(\rho\) is indeed \(\mathcal{V}\). The inherited bracket from the right invariant vector fields is precisely \([(Y,\eta),(X,\xi)]=([Y,X],0)\), hence \[\mathcal{A}(\mathcal{G})\cong \mathcal{A}\]
and \(\mathcal{G}\) integrates \(\mathcal{A}\).
\end{proof}

\begin{rem}
    We can rewrite the groupoid \(\mathcal{G}\) as the fibered product over \(S\) of two groupoids over \(S\):
    \[\mathcal{G} = (S\times_BS) \times_S p^*T^*B.\]
\end{rem}

\begin{cor}
    Choosing a symplectic connection form \(\Omega \in \Omega^2(S)\) induces an isomorphism of Lie groupoids:
    \[\mathcal{G}(T^*S) \cong \mathcal{G}(\mathcal{V}) \times_Sp^*T^*B.\]

    Under that isomorphism, the form \(\omega_{\mathcal{G}(T^*S)}\) is equal to \(\omega_{\Omega} + p^*\mathrm{d}\lambda\) where \(\omega_{\Omega}\in \Omega^2(\mathcal{G}(\mathcal{V}))\) is a multiplicative \(2\)-form depending on \(\Omega\), and \(\lambda \in \Omega^1(T^*B)\) is the Liouville form.
\end{cor}
\begin{proof}
    By Proposition \ref{Proposition: Algebroid splitting}, choosing a symplectic connection form \(\Omega\) induces an isomorphism of Lie algebroids
    \[T^*S \cong \mathcal{A}.\]
    This isomorphism lifts to an isomorphism between the source connected, source simply connected groupoids integrating both algebroids. The one integrating \(T^*S\) is \(\mathcal{G}(T^*S)\) by definition. The groupoid \(\mathcal{G}(\mathcal{V}) \times_S p^*T^*S\) integrates \(\mathcal{A}\) and is source connected, source simply connected. It is thus the source connected, source simply connected groupoid integrating \(\mathcal{A}\), and is then isomorphic to \(\mathcal{G}(T^*S)\).

    Now for the symplectic forms, we need to rely on the equivalent of multiplicative forms at the level of the algebroid. These are the infinitesimally multiplicative forms of \cite{BursztynCabrera}. They are the infinitesimal counterpart of multiplicative forms. Given a multiplicative form \(\sigma \in \Omega^{k}_{mult}(\mathcal{G})\) on a groupoid \(\mathcal{G}\rightrightarrows S\), we can restrict it to the units to obtain a map:
    \[\mu_{\sigma}\colon  \mathcal{A}(\mathcal{G}) \to \Lambda^{k-1}T^*S.\]
    We can do the same with its differential because it is also multiplicative: 
    \[\nu_{\sigma} := \mu_{\mathrm{d}\sigma}\colon \mathcal{A}(\mathcal{G})\to \Lambda^kT^*S.\]
    The couple \((\mu_{\sigma},\nu_{\sigma})\) completely characterizes the form \(\sigma\). These two forms satisfy some compatibility conditions.

    The infinitesimally multiplicative form induced on \(T^*S\) by \(\omega_{\mathcal{G}(T^*S)}\) is \((\mathrm{Id}_{T^*S},0)\). We have fixed a symplectic connection \(\Omega \in \Omega^2(S)\), this gives the isomorphism \(\Phi \colon \mathcal{A} \to T^*S\). We get a corresponding infinitesimally multiplicative form on \(\mathcal{V}\) given by \(X \mapsto (\iota_X\Omega\in T^*S,0)\). This infinitesimally multiplicative form corresponds to a multiplicative form \(\omega_{\Omega}\in \Omega^2(\mathcal{G}(\mathcal{V})\). This form is, by construction, the pullback of \(\omega_{\mathcal{G}(T^*S)}\) under the inclusion 
    \[\mathcal{G}(\mathcal{V}) \hookrightarrow \mathcal{G}(\mathcal{V})\times_Sp^*T^*B \cong \mathcal{G}(T^*S).\]
    Similarly the infinitesimally multiplicative form on \(p^*T^*B\) is the canonical embedding \(p^*T^*B \to T^*S\) (dual to the projection \(TS \to p^*TB\)), and the zero map for the second component. This is the same one as the one induced from the pullback of the Liouville form. Therefore, by injectivity of the map sending a multiplicative form to its infinitesimally multiplicative counterpart, we get that, under the isomorphism induced by \(\Omega\),
    \[\omega_{\mathcal{G}(T^*S)} = \omega_{\Omega} + p^*\mathrm{d}\lambda.\]
\end{proof}

\begin{cor}\label{Corollary: Equivalence prequantization symplectic and vertical}
    Let \(\Omega\) be a symplectic connection for the symplectic fibration \(S \to B\). The groupoid \(\left(\mathcal{G}(T^*S),\omega_{\mathcal{G}(T^*S)}\right)\) is prequantizable if and only if \(\left(\mathcal{G}(\mathcal{V}),\omega_{\Omega}\right)\) is prequantizable.
\end{cor}

\begin{proof}
    Assume that \(\left(\mathcal{G}(\mathcal{V}),\omega_{\Omega}\right)\) is prequantizable, i.e. there exists a Lie groupoid and multiplicative \(1\)-form \(\left(\widetilde{\mathcal{G}}({\mathcal{V}}),\theta_{\Omega}\right)\), such that:
\[\begin{tikzcd}
	{S\times \mathbb{S}^1} & {\widetilde{\mathcal{G}}(\mathcal{V})} & {\mathcal{G}(\mathcal{V})}
	\arrow[from=1-1, to=1-2]
	\arrow["{\pi_{\mathcal{G}(\mathcal{V})}}", two heads, from=1-2, to=1-3]
\end{tikzcd}\]
is a central, groupoid, \(\mathbb{S}^1\)-extension, and \[\mathrm{d}\theta_{\Omega}=\pi^*_{\mathcal{G}(\mathcal{V})}\omega_{\Omega}.\]
The pullback of a principal \(\mathbb{S}^1\)-bundle along as smooth map is again a principal \(\mathbb{S}^1\) bundle. We denote by \(q \colon \mathcal{G}(T^*S) \to \mathcal{G}(\mathcal{V})\) the groupoid homomorphism obtained by projection onto the factor \(\mathcal{G}(\mathcal{V})\) under the isomorphism \(\mathcal{G}(T^*S) \cong \mathcal{G}(\mathcal{V})\times_Sp^*T^*B\) induced by \(\Omega\). 
We define \(\widetilde{\mathcal{G}}(T^*S)\) as the pullback
\[\widetilde{\mathcal{G}}(T^*S):=q^*\widetilde{\mathcal{G}}(\mathcal{V})=\{(\gamma,\delta)\in \mathcal{G}(T^*S)\times\widetilde{\mathcal{G}}(\mathcal{V}) \mid q(\gamma)=\pi_{\mathcal{G}(\mathcal{V})}(\delta)\}.\]
The map 
\[\pi_{\mathcal{G}(T^*S)}\colon\widetilde{\mathcal{G}}(T^*S)\to\mathcal{G}(T^*S)\] 
 obtained by projection onto the first factor is the projection map of an \(\mathbb{S}^1\)-bundle over \(\mathcal{G}(T^*S)\). 
We can define a groupoid multiplication on \(\widetilde{\mathcal{G}}(T^*S)\) by 
\[(\gamma,\delta)(\gamma',\delta'):=(\gamma\gamma',\delta\delta'),\]
which is well defined, since \(q\) is a groupoid morphism, and \[q(\gamma\gamma')=q(\gamma)q(\gamma')=\pi_{\mathcal{G}(\mathcal{V})}(\delta)\pi_{\mathcal{G}(\mathcal{V})}(\delta')=\pi_{\mathcal{G}(\mathcal{V})}(\delta\delta').\]
The source, range, inverse and unit maps are defined analogously.
Therefore, \(\widetilde{\mathcal{G}}(T^*S)\) defines a central \(\mathbb{S}^1\) groupoid extension of \(\mathcal{G}(T^*S)\). We denote the projection onto the second factor by \(\tilde{q}\) and obtain the following commutative diagram of groupoids:
\[\begin{tikzcd}
	{S\times \mathbb{S}^1} & {\widetilde{\mathcal{G}}(T^*S)} & {\mathcal{G}(T^*S)} \\
	{S\times \mathbb{S}^1} & {\widetilde{\mathcal{G}}(\mathcal{V})} & {\mathcal{G}(\mathcal{V})}
	\arrow[from=1-1, to=1-2]
	\arrow["\cong", from=1-1, to=2-1]
	\arrow["{\pi_{\mathcal{G}(T^*S)}}", two heads, from=1-2, to=1-3]
	\arrow["{\tilde{q}}", from=1-2, to=2-2]
	\arrow["{q}", from=1-3, to=2-3]
	\arrow[from=2-1, to=2-2]
	\arrow["{\pi_{\mathcal{G}(\mathcal{V})}}", two heads, from=2-2, to=2-3]
\end{tikzcd}\]
We define the one form: \[\theta_{\mathcal{G}(T^*S)}:=\tilde{q}^*\theta_{\Omega}+\pi^*_{\mathcal{G}(T^*S)}p^*\lambda\in \Omega^1(\widetilde{\mathcal{G}}(T^*S)).\]
This form is multiplicative, as both of its summands are. Moreover,

\begin{align*}
    \mathrm{d}\theta_{\mathcal{G}(T^*S)}&=\mathrm{d}(\tilde{q}^*\theta_{\Omega}+\pi^*_{\mathcal{G}(T^*S)}p^*\lambda)\\
    &=\tilde{q}^*\mathrm{d}\theta_{\Omega}+\pi^*_{\mathcal{G}(T^*S)}p^*\mathrm{d}\lambda\\
    &= \pi_{\mathcal{G}(T^*S)}^*(\omega_{\Omega} +p^*\mathrm{d}\lambda) \\
    &=\pi^*_{\mathcal{G}(T^*S)}\omega_{\mathcal{G}(T^*S)}.
\end{align*}
Therefore \((\mathcal{G},\omega_\mathcal{G})\) is prequantizable.

To prove the converse, we now assume that \((\mathcal{G}(T^*S),\omega_{\mathcal{G}(T^*S)})\) is prequantizable. Thus, there is a central extension 
\[\begin{tikzcd}
	{S\times \mathbb{S}^1} & {\widetilde{\mathcal{G}}(T^*S)} & {\mathcal{G}(T^*S)}
	\arrow[from=1-1, to=1-2]
	\arrow["{\pi_{\mathcal{G}(T^*S)}}", two heads, from=1-2, to=1-3]
\end{tikzcd}\]
and a multiplicative \(1\)-form \(\theta_{\mathcal{G}(T^*S)}\in \Omega^1(\widetilde{\mathcal{G}}(T^*S))\), such that its differential is \(\mathrm{d}\theta_{\mathcal{G}(T^*S)}=\pi_{\mathcal{G}(T^*S)}^*\omega_{\mathcal{G}(T^*S)}\).

Let \[j\colon \mathcal{G}(\mathcal{V})\to \mathcal{G}(T^*S) \quad \gamma\mapsto (\gamma,0_{r(\gamma)}) \]
defined through the identification \(\mathcal{G}(T^*S)\cong \mathcal{G}(\mathcal{V})\times_Sp^*T^*B\) induced by \(\Omega\).
This map is a well defined groupoid morphism. We denote the pullback along \(j\) by 
\[\widetilde{\mathcal{G}}(\mathcal{V}):=j^*\widetilde{\mathcal{G}}(T^*S)=\{(\gamma,\delta)\in \mathcal{G}(\mathcal{V})\times \mathcal{\widetilde{G}}(T^*S)\mid j(\gamma)=\pi_{\mathcal{G}(T^*S)}(\delta)\},\]
and get a map
\[\pi_{\mathcal{G}(\mathcal{V})} \colon \widetilde{\mathcal{G}}(\mathcal{V}) \to \mathcal{G}(\mathcal{V})\]
obtained by projection onto the first factor. We denote by \(\widetilde{j} \colon \widetilde{\mathcal{G}}(\mathcal{V}) \to \widetilde{\mathcal{G}}(T^*S)\) the map induced by the projection onto the second factor. We obtain a commutative diagram of Lie groupoids:

\[\begin{tikzcd}
	{S\times \mathbb{S}^1} & {\widetilde{\mathcal{G}}(T^*S)} & {\mathcal{G}(T^*S)} \\
	{S\times \mathbb{S}^1} & {\widetilde{\mathcal{G}}(\mathcal{V})} & {\mathcal{G}(\mathcal{V})}
	\arrow[from=1-1, to=1-2]
	\arrow["{\pi_{\mathcal{G}(T^*S)}}", two heads, from=1-2, to=1-3]
	\arrow["\cong"', from=2-1, to=1-1]
	\arrow[from=2-1, to=2-2]
	\arrow["{\tilde{j}}"', from=2-2, to=1-2]
	\arrow["{\pi_{\mathcal{G}(\mathcal{V})}}", two heads, from=2-2, to=2-3]
	\arrow["{j}"', from=2-3, to=1-3]
\end{tikzcd}\]

By the same reasoning as above, \(\widetilde{\mathcal{G}}(\mathcal{V})\) is an \(\mathbb{S}^1\) central groupoid extension. We define the multiplicative one form \[\theta_{\Omega}:=\tilde{j}^*\theta_{\mathcal{G}(T^*S)}.\]
We observe that \(j\) is a section for \(q\), i.e. \(p\circ j=\mathrm{id}_{\mathcal{G}(\mathcal{V})}\) and that the composition of \(j\) with the projection onto \(p^*T^*B\) maps \(\mathcal{G}(\mathcal{V})\) to the zero section of \(p^*T^*B\).

\begin{align*}
    \mathrm{d}\theta_{\Omega}&=\mathrm{d}\tilde{j}^*\theta_{\mathcal{G}(T^*S)}\\
    &=\tilde{j}^*\mathrm{d}\theta_{\mathcal{G}(T^*S)}\\
    &=\tilde{j}^*\pi^*_{\mathcal{G}(T^*S)}\omega_{\mathcal{G}(T^*S)}\\
    &=\pi^*_{\mathcal{G}(\mathcal{V})}j^*\omega_{\mathcal{G}(T^*S)}\\
    &=\pi^*_{\mathcal{G}(\mathcal{V})}(j^*q^*\omega_{\Omega}+j^*p^*\mathrm{d}\lambda)\\
    &=\pi^*_{\mathcal{G}(\mathcal{V})}\omega_{\Omega},
\end{align*}
where we used in the last step, that the zero section \(S\hookrightarrow p^*T^*B\) is a Lagrangian submanifold for \(p^*\mathrm{d}\lambda\), thus the pullback on \(\mathcal{G}(\mathcal{V})\) vanishes.
This proves the prequantizability of \((\mathcal{G}(\mathcal{V}),\omega_{\Omega})\).

\end{proof}

We assume for now that we can find a symplectic connection form \(\Omega\) such that \(\omega_{\Omega} = s^*\Omega - r^*\Omega\). We will explain the reason of this assumption and when it is the case in the next subsection. Notice that in that case this form is the pullback of a form on \(S\times_B S\) defined in the same way, which we will also denote by \(\omega_{\Omega}\). We will also assume that it is automatically the case when \(\Omega\) is closed (we will find more precise conditions later on).

\begin{lem}
    Assume we have a symplectic connection form \(\Omega\) such that \(\omega_{\Omega} = s^*\Omega - r^*\Omega\). Then \(\left((S\times_BS)\times_S p^*T^*B, \omega_{\Omega} + p^*\mathrm{d}\lambda\right)\) is a symplectic groupoid.
\end{lem}
\begin{proof}
    The natural groupoid homomorphism \(\mathcal{G}(\mathcal{V})\times_S p^*T^*B \to (S\times_BS)\times_S p^*T^*B\) is a local diffeomorphism as both groupoids integrate the same algebroid. Since \(\mathcal{G}(\mathcal{V})\times_S p^*T^*B\) is a symplectic groupoid, then so is \((S\times_BS)\times_S p^*T^*B\).
\end{proof}

\begin{lem}
    Assume we have a symplectic connection form \(\Omega\) such that \(\omega_{\Omega} = s^*\Omega - r^*\Omega\). Then \((S\times_B S, \omega_{\Omega})\) is prequantizable if and only if the symplectic groupoid \(\left((S\times_BS)\times_S p^*T^*B, \omega_{\Omega} + p^*\mathrm{d}\lambda\right)\) is prequantizable.
\end{lem}

\begin{proof}
    The proof is identical to Corollary \ref{Corollary: Equivalence prequantization symplectic and vertical} by pulling back the respective \(\mathbb{S}^1\)-extensions.
\end{proof}

\begin{lem}
    Assume we have a symplectic connection form \(\Omega\) such that \(\omega_{\Omega} = s^*\Omega - r^*\Omega\). If \((S\times_B S, \omega_{\Omega})\) is prequantizable then \(\left(\mathcal{G}(\mathcal{V}),\omega_{\Omega}\right)\) is prequantizable.
\end{lem}

\begin{proof}
    If \((S\times_B S, \omega_{\Omega})\) is prequantizable then by pulling back the \(\mathbb{S}^1\)-bundle to \(\mathcal{G}(\mathcal{V})\) we see that \(\left(\mathcal{G}(\mathcal{V}),\omega_{\Omega}\right)\) is also prequantizable. 
\end{proof}

Notice that the converse of the previous statement does not hold in general. Even though the cohomology class \(\left[\frac{1}{2\pi}\omega_{\Omega}\right]\in \mathrm{H}^2(\mathcal{G}(\mathcal{V});\mathbb{Z})\) does pull back to the cohomology class of \(\frac{1}{2\pi}\omega_{\Omega}\) on \(S\times_B S\), the latter may not be integral anymore.

\begin{prop}
    Let \((S\to B,\Omega)\) be a prequantizable symplectic fibration, then \((S\times_BS,\omega_{\Omega})\) is prequantizable and therefore \(\left((S\times_BS)\times_S p^*T^*B, \omega_{\Omega} + p^*\mathrm{d}\lambda\right)\) is prequantizable. Consequently, \((\mathcal{G}(\mathcal{V}),\omega_{\Omega})\) and \(\left(\mathcal{G}(T^*S),\omega_{\mathcal{G}(T^*S)}\right)\) are also prequantizable.
\end{prop}

\begin{proof}
    Suppose \((S\to B,\Omega)\) is prequantizable, then in particular \(\mathrm{d}\Omega=0\), which implies that \(\omega_{\Omega} = s^*\Omega - r^*\Omega\). The projection \(p\colon X\to S\) induces a projection \((p,p)\colon X\times_BX\to S\times_BS\). This defines a central \(\mathbb{T}^2\) groupoid extension, with a circle action on each component: 
    \[(x,x')\cdot(\lambda,\lambda')=(x\cdot\lambda,x'\cdot \lambda'), \quad (x,x')\in X\times_B X, \lambda,\lambda'\in\mathbb{S}^1,\]
\[\begin{tikzcd}
	{S\times \mathbb{T}^2} & {X\times_BX} & {S\times_B S.}
	\arrow[hook, from=1-1, to=1-2]
	\arrow["{(p,p)}", two heads, from=1-2, to=1-3]
\end{tikzcd}\]
In order to obtain a central \(\mathbb{S}^1\)-extension, we consider the Atiyah groupoid of \(\mathbb{S}^1\hookrightarrow X\to S\) over \(B\) \[\mathrm{At}_B(X)=\faktor{X\times_B X}{\mathbb{S}^1},\]
    where the quotient is taken over the diagonal action.  This provides a central \(\mathbb{S}^1\)-extension, with the action \([x,x']\cdot\lambda=[x\cdot\lambda,x']=[x,x'\cdot \lambda^{-1}]\). We denote by \(\pi \colon \mathrm{At}_B(X) \to S\times_BS\) the projection given by \(\pi([x,x'])=(p(x),p(x'))\) and by \(q\colon X\times_BX\to \mathrm{At}_B(X)\) the quotient map. We get the commutative diagram
\[\begin{tikzcd}
	{S\times \mathbb{T}^2} & {X\times_BX} & \\
	{S\times \mathbb{S}^1} & {\mathrm{At}_B(X)} & {S\times_B S.}
	\arrow[hook, from=1-1, to=1-2]
	\arrow[from=1-1, to=2-1]
	\arrow["q"', from=1-2, to=2-2]
	\arrow["{(p,p)}", two heads, from=1-2, to=2-3]
	\arrow[hook, from=2-1, to=2-2]
	\arrow["{\pi}", two heads, from=2-2, to=2-3]
\end{tikzcd}\]
In order to obtain a prequantization, we need to construct a multiplicative connection form \(\theta_{\Omega}\in \Omega^1(\mathcal{G}_0)\), such that its curvature satisfies \(\mathrm{d}\theta_{\Omega}=\pi^*\omega_{\Omega}\). Let \(\Theta \in \Omega^1(X)\) be the connection form for \(X \to B\) as prequantization of the symplectic fibration \((S\to B,\Omega)\). We define on \(X\times_BX\) the multiplicative form \(\tilde{\theta}:=\mathrm{s^*\Theta-r^*\Theta}\). This form descends to \(\mathrm{At}_B(X)\), since it is basic for the diagonal \(\mathbb{S}^1\)-action. Indeed, let \(R\in \mathfrak{X}(X)\) be the vertical Reeb vector field, i.e. the generator of the \(\mathbb{S}^1\)-action. Then \((R,R)\in \mathfrak{X}(X\times_BX)\) is the generator of the diagonal \(\mathbb{S}^1\)-action, and \[\iota_{(R,R)}\tilde{\theta}=\Theta(R)-\Theta(R)=0.\]
Hence, there is a multiplicative form \(\theta_{\Omega}\in \Omega^1(\mathrm{At}_B(X))\) such that \(q^*\theta_{\Omega}=\tilde{\theta}\). We now compute the differential:

\begin{align*}
    q^*\mathrm{d}\theta_{\Omega}&=\mathrm{d}q^*\theta_{\Omega}\\
    &=\mathrm{d}\tilde{\theta}\\
    &=\mathrm{d}(s^*\Theta-r^*\Theta)\\
    &=s^*\mathrm{d}\Theta-r^*\mathrm{d}\Theta\\
    &=s^*p^*\Omega-r^*p^*\Omega\\
    &=(p,p)^*(s^*\Omega-r^*\Omega)\\
    &=(\pi\circ q)^*\omega_{\Omega}\\
    &=q^*\pi^*\omega_{\Omega}
\end{align*}
Since \(q\) is a submersion, \(q^*\) is injective and we obtain \[\mathrm{d}\theta_{\Omega}=\pi^*\omega_{\Omega}.\]

\end{proof}

We summarize the conclusions of that subsection with the following diagram:

\begin{center}
\begin{tikzpicture}[scale=1, every node/.style={scale=0.8}, align=center]
    \node[draw] (a) at (0,0) {\((S\to B, \Omega)\)\\  prequantizable};
    \node[draw] (b) at (5,0) {\((S\times_B S, s^*\Omega-r^*\Omega)\) \\ prequantizable};
    \node[draw] (c) at (10,0) {\((\mathcal{G}(\mathcal{V}),\omega_{\Omega})\) \\ prequantizable};
    \node[draw] (d) at (0,-2) {\((S,\Pi)\) \\ prequantizable};
    \node[draw] (e) at (5,-2) {\(\left((S\times_B S)\times_S p^*T^*B,\omega_{\Omega}+p^*\mathrm{d}\lambda\right)\) \\ prequantizable};
    \node[draw] (f) at (10,-2) {\(\left(\mathcal{G}(T^*S),\omega_{\mathcal{G}(T^*S)}\right)\) \\ prequantizable};
    \draw[implies-implies,double equal sign distance] (a) -- (d) ;
    \draw[implies-implies,double equal sign distance] (b) -- (e) ;
    \draw[implies-implies,double equal sign distance] (c) -- (f) ;
    \draw[-implies,double equal sign distance] (a) -- (b) ;
    \draw[-implies,double equal sign distance] (b) -- (c);
\end{tikzpicture}
\end{center}

\subsection{Prequantization of the fibered pair groupoid}\label{Subsection: fibered pairs}

Let \(\Omega\) be a symplectic connection on \(S\to B\). We now study when the form \(\omega_{\Omega}\) can be pulled back from the fibered pair groupoid \(S\times_BS\). We have a natural multiplicative form on \(S\times_B S\) given by \(s^*\Omega-r^*\Omega\). Moreover, every multiplicative form on \(S\times_BS\) is obtained in a similar way from a form on \(S\) (see Appending \ref{Appendix}). The problem is that this form is not necessarily closed. We first justify that the form on \(S\) we want has to be a symplectic connection.

\begin{lem}
    Let \(S\to B\) be a symplectic fibration and \(\Omega \in \Omega^2(S)\) be any extension of \(\omega_{\mathcal{V}}\). If \(s^*\Omega-r^*\Omega \in \Omega^2(S\times_BS)\) is closed then \(\Omega\) is a symplectic connection. Moreover, in that case, then on \(\mathcal{G}(\mathcal{V})\) we have \(\omega_{\Omega} = s^*\Omega - r^*\Omega\).
\end{lem}

\begin{proof}
      We prove the result by looking at the corresponding infinitesimally multiplicative forms. Assume that \(s^*\Omega - r^*\Omega\) is closed. Then the corresponding infinitesimally multiplicative form only contains a map \(\mu \colon \mathcal{V}\to T^*S\), the other one vanishes. This map is \(\mu = \Omega^{\sharp}|_{\mathcal{V}}\), where \(\Omega^{\sharp}\colon TS \to T^*S\) is the induced musical map.
    
    The map \(\mu\) has satisfies the following conditions:
    \begin{align*}
        \forall v_1,v_2\in \Gamma(\mathcal{V}), \quad &\iota_{v_1}\mu(v_2) = - \iota_{v_2}\mu(v_1) \\
        \forall v_1,v_2\in \Gamma(\mathcal{V}), \quad &\mu([v_1,v_2]) )= \mathcal{L}_{v_1}\mu(v_2) - \iota_{v_2}\mathrm{d}(\mu(v_1)).
    \end{align*}
    The first condition is automatic for \(\Omega^{\sharp}\), regardless of it being a symplectic connection form. We will show that the second condition is equivalent to \(\Omega\) being a symplectic connection form. The condition to show for that is
    \[\forall v_1,v_2\in \Gamma(\mathcal{V}), \iota_{v_1\wedge v_2}\mathrm{d}\Omega = 0.\]
    We thus take \(v_1,v_2\in \Gamma(\mathcal{V})\) and \(X\in \mathfrak{X}(S)\). We have
    \begin{align*}
        \mathrm{d}\Omega(v_1,v_2,X) = &\quad \mathcal{L}_{v_1}\Omega(v_2,X) + \Omega([v_1,X],v_2) \\
        & -\mathcal{L}_{v_2}\Omega(v_1,X) + \mathcal{L}_X\Omega(v_1,v_2) -\Omega([v_2,X], v_1) \\
        & -\Omega([v_1,v_2],u).
    \end{align*}
    Using similar computations as in the proof of Proposition \ref{Proposition: Algebroid splitting}, we obtain the relations:
    \begin{align*}
        \mathcal{L}_{v_1}\Omega^{\sharp}(v_2)(X) &= \mathcal{L}_{v_1}\Omega(v_2,X) + \Omega([v_1,X],v_2), \\
        \iota_{v_2}\mathrm{d}(\mu(v_1))(X) &= \mathcal{L}_{v_2}\Omega(v_1,X) - \mathcal{L}_X\Omega(v_1,v_2) +\Omega([v_2,X], v_1).       
    \end{align*}
    Therefore since \(\Omega^{\sharp} = \mu\), we get that
    \[\mathrm{d}\Omega(v_1,v_2,X) = \Omega([v_1,v_2],X) - \Omega([v_1,v_2],X) = 0.\]
    The form \(\Omega\) is thus a symplectic connection.

    Now, knowing that \(\Omega\) is a symplectic connection, we get a multiplicative form \(\omega_{\Omega}\in \Omega^2(\mathcal{G}(\mathcal{V}))\). Its corresponding infinitesimally multiplicative form is \((\Omega^{\sharp}|_{\mathcal{V}},0)\) which is the same as the one induced by \(s^*\Omega-r^*\Omega\) (assuming the latter to be closed). It is shown in \cite{BursztynCabrera} that mapping a multiplicative form to its infinitesimally multiplicative counterpart is injective. Therefore \(\omega_{\mathcal{G}(\mathcal{V})} = s^*\Omega - r^*\Omega\) on \(\mathcal{G}(\mathcal{V})\).
\end{proof}

If we fix an Ehresmann connection \(\mathcal{H}\) for the fibration, it induces a bigrading of the de Rham complex \[\Omega^k(S)=\bigoplus_{p+q=k}\Omega^{p,q}(S),\quad \Omega^{p,q}(S):=\Gamma^\infty(S;\bigwedge^p\mathcal{V}^*\otimes\bigwedge^q\mathcal{H}^*).\]
The de Rham differential then decomposes as \(\mathrm{d}=\mathrm{d}^{1,0}+\mathrm{d}^{0,1}+\mathrm{d}^{-1,2}\). The operators \(\mathrm{d}^{1,0},\mathrm{d}^{0,1}\) correspond to differentiation in the vertical and horizontal directions respectively. The operator \(\mathrm{d}^{-1,2}\) is the contraction by the map \(\Lambda^2 \mathcal{H} \to \mathcal{V}\) obtained by composition of the Lie bracket of sections and projection onto the vertical subspace (it is indeed tensorial).

\begin{rem}
    We have a map \(\Omega^{k}(S) \to \Gamma^{\infty}(S;\bigwedge^k\mathcal{V}^*)\) given by restriction. If we pick a connection, this map identifies \(\Gamma^{\infty}(S;\bigwedge^k\mathcal{V}^*)\) and \(\Omega^{k,0}(S)\). The map on \(\Gamma^{\infty}(S;\bigwedge^k\mathcal{V}^*)\) obtained from \(d^{1,0}\) is the de Rham differential computed along the leaves of the fibration. Therefore this map does not depend on any choice of connection.
\end{rem}

The equation \(\mathrm{d}^2=0\) then induces the following relations between its components:
\begin{enumerate}
    \item \((\mathrm{d}^{1,0})^2=0\)
    \item \(\mathrm{d}^{0,1}\mathrm{d}^{1,0}+\mathrm{d}^{1,0}\mathrm{d}^{0,1}=0\)
    \item \((\mathrm{d}^{0,1})^2+\mathrm{d}^{1,0}\mathrm{d}^{-1,2}+\mathrm{d}^{-1,2}\mathrm{d}^{1,0}=0\)
    \item \(\mathrm{d}^{0,1}\mathrm{d}^{-1,2}+\mathrm{d}^{-1,2}\mathrm{d}^{0,1}=0\)
    \item \((\mathrm{d^{-1,2}})^2=0\)
\end{enumerate}

\begin{lem}
    Let \(S \to B\) be a fibration with connected fibers, \(\mathcal{H} \subset TS\) an Ehresmann connection. A closed \(k\)-form is basic, i.e. the pullback of a \(k\)-form from \(B\), if and only if it is of type \((0,k)\).
\end{lem}
\begin{proof}
    Basic \(k\)-forms are always of type \((0,k)\). Conversely, let \(\omega \in \Omega^{0,k}(S)\) be a closed form. We have \(\mathrm{d}\omega = \mathrm{d}^{1,0}\omega + \mathrm{d}^{0,1}\omega\). Both of the components have different degree so the closedness of \(\omega\) gives \(\mathrm{d}^{1,0}\omega = \mathrm{d}^{0,1}\omega = 0\). In particular, since \(\mathrm{d}^{1,0}\omega = 0\) and \(\omega\) has degree \((0,k)\), then it means that \(\omega\) is locally constant on the fibers of the fibration. The fibers being connected, \(\omega\) is constant on the fibers, hence the pullback of a \(k\)-form on \(B\).
\end{proof}

For \(3\)-forms, we obtain the decomposition: \[\Omega^3(S)=\Omega^{3,0}(S)\oplus\Omega^{2,1}(S)\oplus\Omega^{1,2}(S)\oplus \Omega^{0,3}(S).\] The condition on symplectic connection forms for \(\Omega\) translates to \(\mathrm{d}\Omega\in \Omega^{1,2}(S)\oplus \Omega^{0,3}(S)\).
For a fixed symplectic connection \(\Omega\in \Omega^2(S)\), we have decomposition given by its corresponding Ehresmann connection: 
\[\Omega=\begin{bmatrix}
        \omega_\mathcal{V}&0\\
        0& \omega'
\end{bmatrix} \quad \omega_\mathcal{V}\in \Omega^{2,0}(S) \quad \omega' \in \Omega^{0,2}(S).\]

\begin{lem}\label{Lemma: fibered pairs closedness}
    Let \(S\to B\) be a symplectic fibration with a symplectic connection form \(\Omega\). Let \(S\times_BS\) be the fibered pair groupoid and \(\omega_{0}:=s^*\Omega-r^*\Omega\) the corresponding multiplicative two form. The following are equivalent:
    \begin{enumerate}
        \item \(\mathrm{d}\omega_{0}=0\).
        \item \(s^*\mathrm{d}\Omega=r^*\mathrm{d}\Omega\).
        \item \(\mathrm{d}\Omega=\pi^*\eta\) for a closed \(3\)-form \(\eta\in \Omega^3(B)\).
    \end{enumerate}
    Moreover if this is the case then \(\Pi^{\sharp}\Omega\) is closed for the Lichnerowicz-Poisson differential \(\mathrm{d}_{\Pi}\). 
\end{lem}

\begin{proof}
    To see the equivalence of \((1)\) and \((2)\) we simply apply the de Rham differential to \(\omega_{0}\):
    \[\mathrm{d}\omega_{0}=\mathrm{d}(s^*\Omega-r^*\Omega)=s^*\mathrm{d}\Omega-r^*\mathrm{d}\Omega.\]
    The equivalence of \((2)\) and \((3)\) follows from the general fact, that a form \(\alpha\in \Omega^k(S)\) satisfies that \(\mathrm{pr}_1^*\alpha=\mathrm{pr}_2^*\alpha\in \Omega^k(\mathcal{G}_0)\) if and only if \(\alpha=\pi^*\beta\) for some \(\beta\in \Omega^k(B)\). To not break the flow of the paper, we postpone this computation to Appendix \ref{Appendix}. The form \(\eta\) is in addition closed, since \[0=\mathrm{d}^2\Omega=\mathrm{d}\pi^*\eta=\pi^*\mathrm{d}\eta,\]
    and \(\pi^*\) is injective, as \(\pi\) is a submersion.

    Finally for the last claim, notice that if \(\mathrm{d}\Omega=\pi^*\eta\) then \(\Pi^{\sharp}\mathrm{d}\Omega = 0\). Since \(\Pi^{\sharp}\) is a chain map, this rewrites as \(\mathrm{d}_{\Pi}\Pi^{\sharp}\Omega = 0\).
\end{proof}

We want to understand under which conditions, we can choose a symplectic connection \(2\)-form \(\Omega\in \Omega^2(S)\) that satisfies the conditions of the previous lemma. The theorem giving this condition will use some facts about vertical cohomology. Since \((\mathrm{d}^{1,0})^2 = 0\), we obtain complexes \(\Omega^{\bullet,k}(S)\) for every \(k\geq 0\). The spaces in these complexes can be considered as total sections of sheaves defined analogously. The map \(\mathrm{d}^{1,0}\) being local, we then obtain complexes of sheaves. It is proved in \cite[Théorème 3.1]{Vaisman} that 
\[\xymatrix{0 \ar[r] & p^*\Omega^k(B) \ar[r] & (\Omega^{\bullet,k}(S), \mathrm{d}^{1,0})}\]
is a resolution of acyclic sheaves (we make a small abuse of notations between the sheaf itself and its space of global sections). Consequently:
\[\mathrm{H}^k((\Omega^{\bullet,k}(S), \mathrm{d}^{1,0})) = \mathrm{H}^k(S;p^*\Omega^k(B)).\]
Notice that, while the first space depends on the choice of connection, the second one does not. Many cohomology classes that are defined through vertical cohomology might depend on the choice of splitting. In order to solve that problem, notice that the de Rham differential on \(B\) induces maps between the sheaves \(\mathrm{d}\colon p^*\Omega^k(B) \to p^*\Omega^{k+1}(B)\). Moreover, the map \(\mathrm{d}^{0,1}\) induces a chain map between the respective the resolutions of the sheaves (with alternating signs, using the fact that \(\mathrm{d}^{1,0}\) and \(\mathrm{d}^{0,1}\) anticommute). We then obtain maps:
\[p^*\mathrm{d} \colon \mathrm{H}^{\ell}(S;p^*\Omega^k(B)) \to \mathrm{H}^{\ell}(S;p^*\Omega^{k+1}(B)), \quad k,\ell\geq 0.\]
Finally, using the relation
\[(\mathrm{d}^{0,1})^2+\mathrm{d}^{1,0}\mathrm{d}^{-1,2}+\mathrm{d}^{-1,2}\mathrm{d}^{1,0}=0,\]
we obtain for each \(\ell \geq 0\) a chain complex:
\[\left(\mathrm{H}^{\ell}(S,p^*\Omega^{\bullet}(B)),p^*\mathrm{d}\right).\]
We denote by
\[\mathrm{H}^k(\mathrm{H}^{\ell}(S;p^*\Omega^{\bullet}(B)))\]
its degree \(k\) cohomology group.
\begin{thm}\label{Theorem: Obstruction classes}
    Let \(S\to B\) be a symplectic fibration, \(\mathcal{H}\) a symplectic connection. There exists a symplectic connection form \(\Omega\in \Omega^2(S)\), such that \(\mathrm{d}\Omega\) is basic and \(\mathcal{
    H}^\Omega=\mathcal{
    H}\) if and only if \([\mathrm{d}^{-1,2}\omega_\mathcal{V}]=0\in \mathrm{H}^1(S;p^*\Omega^2(B))\). 
    
    Moreover, there exists a symplectic connection form with basic differential if and only if \([[\mathrm{d}^{-1,2}\omega_\mathcal{V}]]=0\in \mathrm{H}^2(\mathrm{H}^1(S;p^*\Omega^{\bullet}(B)))\).
\end{thm}

\begin{proof}
    Suppose \(\Omega\in \Omega^2(S)\) is a symplectic connection \(2\)-form. It extends \(\omega_\mathcal{V}\), i.e. \(\iota_{v_1\wedge v_2}\Omega=\iota_{v_1\wedge v_2}\Omega'\). In the decomposition provided by the Ehresmann connection \(\mathcal{H}\), we can write \(\Omega=\omega_{\mathcal{V}} +\alpha+\beta\), corresponding to the bigrading:

\[\omega_\mathcal{V}\in \Omega^{2,0}(S),\quad \alpha\in \Omega^{1,1}(S), \quad \beta\in \Omega^{0,2}(S).\]

Computing the decomposition of \(\mathrm{d}\Omega\) into its bigraded types, we obtain:
\begin{align*}
     (\mathrm{d}\Omega)^{3,0} &=\mathrm{d}^{1,0}\omega_{\mathcal{V}} \\
     (\mathrm{d}\Omega)^{2,1} &=\mathrm{d}^{0,1}\omega_{\mathcal{V}}+\mathrm{d}^{1,0}\alpha\\
     (\mathrm{d}\Omega)^{1,2}&=\mathrm{d}^{0,1}\alpha+\mathrm{d}^{1,0}\beta+\mathrm{d}^{-1,2}\omega_{\mathcal{V}}\\
     (\mathrm{d}\Omega)^{0,3}&=\mathrm{d}^{-1,2}\alpha+\mathrm{d}^{0,1}\beta.
\end{align*}
 
Since \(\Omega\) is a symplectic connection form, for every pair \(v_1,v_2\in \mathcal{V}_x, x \in S\) we have \(\iota_{v_1\wedge v_2}\mathrm{d}\Omega = 0\). This implies that \(\mathrm{d}\Omega\) only has parts in degree \((1,2)\) and \((0,3)\). By the above, this means that 
\(\mathrm{d}^{1,0}\alpha=-\mathrm{d}^{0,1}\omega_{\mathcal{V}}\) and that \(\mathrm{d}^{1,0}\omega_{\mathcal{V}}=0\).

Since \(\mathrm{d}\Omega\) is closed, and the fibers of \(S\to B\) are connected, it is basic if and only if it is  of type \((0,3)\). This is thus equivalent to the equation:
\begin{equation}\label{Equation:VerticalExactness}
    \mathrm{d}^{1,0}\beta=-(\mathrm{d}^{0,1}\alpha+\mathrm{d}^{-1,2}\omega_{\mathcal{V}}).
\end{equation}

Let us first assume that we want to find a connection form with the additional assumption that \(\mathcal{H}^{\Omega}=\mathcal{H}\). In the decomposition of \(\Omega\) into bigraded parts, this means \(\alpha =0\). Moreover, for every \(\beta \in \Omega^{0,2}(S)\), the form \(\omega_{\mathcal{V}}+\beta\) is a symplectic connection \(2\)-form, with underlying Ehresmann connection \(\mathcal{H}\). Finding one that satisfies \(\mathrm{d}\Omega\) basic is equivalent, by Equation \ref{Equation:VerticalExactness} with \(\alpha=0\), to finding \(\beta \in \Omega^{0,2}(S)\) with
\[\mathrm{d}^{-1,2}\omega_{\mathcal{V}} = \mathrm{d}^{1,0}\beta,\]
i.e. that \(\mathrm{d}^{-1,2}\omega_{\mathcal{V}}\) is vertically exact (for the splitting induced by \(\mathcal{H}\)). Now notice in general that 
\[\mathrm{d}^{1,0}\mathrm{d}^{-1,2}\omega_{\mathcal{V}} = -\left(\mathrm{d}^{0,1}\right)^2\omega_{\mathcal{V}} - \mathrm{d}^{-1,2}\mathrm{d}^{1,0}\omega_{\mathcal{V}} = 0,\]
as \(\mathrm{d}^{0,1}\omega_{\mathcal{V}} = \mathrm{d}^{1,0}\omega_{\mathcal{V}} = 0\).

We thus get that \(\mathrm{d}^{-1,2}\omega_{\mathcal{V}}\) is always vertically closed, hence defines a class
\[\left[\mathrm{d}^{-1,2}\omega_{\mathcal{V}}\right] \in \mathrm{H}^1(S;p^*\Omega^2(B)).\] 
This class vanishes if and only if there is a connection \(2\)-form \(\Omega\) with underlying Ehresmann connection \(\mathcal{H}\) for which \(\mathrm{d}\Omega\) is basic. This proves the first part of the Theorem.

The second part is similar. Finding a symplectic connection \(2\)-form \(\Omega\) with \(\mathrm{d}\Omega\) basic is now equivalent to find \(\alpha\in \Omega^{1,1}(S)\) and \(\beta\in\Omega^{0,2}(S)\) satisfying the equations:
\begin{align*}
    \mathrm{d}^{1,0}\alpha&=-\mathrm{d}^{0,1}\omega_{\mathcal{V}} \\
    \mathrm{d}^{1,0}\beta&=-(\mathrm{d}^{0,1}\alpha+\mathrm{d}^{-1,2}\omega_{\mathcal{V}}).
\end{align*}
Since \(\mathcal{H}\) is a symplectic connection, we can still take \(\alpha =0\) in the above equations, in particular \(\mathrm{d}^{0,1}\omega_{\mathcal{V}}=0\).
The same reasoning above still works so we still have \(\mathrm{d}^{-1,2}\omega_{\mathcal{V}}\) vertically closed. Since \(\mathrm{d}^{1,0}\alpha=-\mathrm{d}^{0,1}\omega_{\mathcal{V}}=0\), \(\alpha\) is also vertically closed. It defines a class \([\alpha]\in \mathrm{H}^1(S,p^*\Omega^1(B))\). Finally, by Equation \ref{Equation:VerticalExactness}, there is a symplectic connection form \(\Omega\) with \(\mathrm{d}\Omega\) basic if and only if we can find \(\alpha \in \Omega^{1,1}(S)\) with
\[\left[\mathrm{d}^{-1,2}\omega_{\mathcal{V}}\right] =  p^*\mathrm{d}[\alpha].\]
In general \(\mathrm{d}^{0,1}\mathrm{d}^{-1,2}\omega_{\mathcal{V}} = -\mathrm{d}^{-1,2}\mathrm{d}^{0,1}\omega_{\mathcal{V}} = 0\) so we get a class
\[\left[\left[\mathrm{d}^{-1,2}\omega_{\mathcal{V}}\right]\right] \in \mathrm{H}^2(\mathrm{H}^1(S;p^*\Omega^{\bullet}(B))).\]
By the above, this class is the obstruction to the existence of a symplectic connection \(2\)-form \(\Omega\) with \(\mathrm{d}\Omega\) basic, which proves the theorem.

\end{proof}

The cohomology groups and differentials appear naturally in the Leray spectral sequence, associated to the fibration \((M,\omega)\to S\to B\). Following \cite[Chapter 3]{Griffiths}, we will briefly discuss this spectral sequence and relate it to the theorem above. 
The filtration \(0\subset\mathcal{V}\subset TS\) induces a dual filtration \(0\subset \mathcal{V}^\perp\subset T^*S\) and a natural filtration on \(\Omega^\bullet(S)\), given by 
\[F^l\Omega^k(S):=\left\lbrace\eta\in \Omega^k(S)\mid\iota_{X_{k-l+1}}\cdots\iota_{X_1}\eta=0,\quad X_1,..,X_{k-l+1}\in \Gamma^\infty(S;\mathcal{V})\right\rbrace.\]
The space \(F^l\Omega^k(S)\) is the subspace of \(k\)-forms which vanish on \(k-l+1\) vertical vectors. Choosing an Ehresmann connection \(\mathcal{H}\),we can identify \[F^l\Omega^k(S)=\bigoplus_{l\leq r}\Omega^{k-r,r}(S),\]
therefore under this identification, these are the forms that require at least \(l\)-horizontal entries. Since \(\Gamma^\infty(S;\mathcal{V})\) is involutive, the filtration is preserved by the de Rham differential, i.e. \(\mathrm{d}F^l\Omega^k(S)\subset F^l\Omega^{k+1}(S)\). The filtration of the de Rham complex provides the \(0\)-th page of the spectral sequence
\[E^{q,p}_0=\frac{F^q\Omega^{q+p}(S)}{F^{q+1}\Omega^{q+p}(S)}\cong \Omega^{p,q}(S),\]
where we used the connection for the latter identification. The differential \(\mathrm{d}_0\) is induced from the de Rham differential \(\mathrm{d}\) under the quotient, hence \((E^{q,p}_0,\mathrm{d}_0)\cong (\Omega^{p,q}(S),\mathrm{d}^{1,0})\). We construct a sequence of differential modules \(\{(E^{\bullet,\bullet}_i,d_i)\}_{i\in \mathbb{N}}\) with \(E_{i+1}\cong\mathrm{H}(E^{\bullet,\bullet}_i,d_i)\), where \(d_i\colon E^{q,p}_i\to E^{q+i,p+1-i}_i\). The \(E_1\)-page is obtained by passing to cohomology: \[E^{q,p}_1=\mathrm{H}^p(\Omega^{\bullet,q}(S),\mathrm{d}^{1,0})\cong \mathrm{H}^p(S;p^*\Omega^q(B)).\]
In particular, \(E^{2,1}_1=\mathrm{H}^1(S;p^*\Omega^2(B))\).
The operator \(\mathrm{d}_1\) is then either identified with 
\[\mathrm{d}^{0,1}\colon \mathrm{H}^p(\Omega^{\bullet,q}(S),\mathrm{d}^{1,0})\to \mathrm{H}^p(\Omega^{\bullet,q+1}(S),\mathrm{d}^{1,0}),\] 
or respectively with 
\[p^*\mathrm{d}\colon \mathrm{H}^p(S;p^*\Omega^q(B))\to \mathrm{H}^p(S;p^*\Omega^{q+1}(B)).\]
The second page of the spectral sequence therefore reads as:
\[E^{q,p}_2=\mathrm{H}^p(\mathrm{H}^q(\Omega^{\bullet,\bullet}(S),\mathrm{d}^{1,0}),\mathrm{d}^{0,1})\cong \mathrm{H}^p(\mathrm{H}^q(S;p^*\Omega^{\bullet}(B))).\]
In particular we have  \(E^{2,1}_2=\mathrm{H}^2(\mathrm{H}^1(S;p^*\Omega^\bullet(B)),p^*\mathrm{d})\) and the cohomology groups, in which the obstruction classes from Theorem \ref{Theorem: Obstruction classes} appear as entries in the first and second page of the Leray spectral sequence.

Another geometric interpretation of \(\mathrm{d}_0\), \(\mathrm{d}_1\) and the groups they act on can be given in terms of higher direct images, bundles of cohomology groups of the fibers and their Gauss-Manin connection. We first identify \[\Omega^{p,q}(S)\cong \Gamma^\infty(S;\bigwedge^p\mathcal{V}^*\otimes\bigwedge^q\mathcal{H}^*)\cong \Omega^q(B;\Omega^p(S/B)),\]
where \(\Omega^p(S/B):=p_*(\bigwedge^p\mathcal{V}^*)\) denotes the infinite rank Fréchet vector-bundle, where the fibers are given by \(\Omega^p(S/B)_b=\Gamma^\infty(S_b;\bigwedge^p\mathcal{V}^*|_{S_b})\). Under this identification we have \[E_0^{q,p}\cong \Omega^q(B;\Omega^p(S/B))\] and \(\mathrm{d}_0=\mathrm{d}^{1,0}\), which is the de Rham differential along the fibers. 

We denote by \(\underline{\mathcal{H}}^p_B:=R^pp_*\underline{\mathbb{R}}_S\) the \(p\)-th direct image of the local system \(\underline{\mathbb{R}}_S\) and by \(\mathcal{H}^p(S/B):=\underline{\mathcal{H}}^p_B\otimes_\mathbb{R}C^\infty_B\) the sheaf of \(C^\infty(B)\)-modules, which corresponds to the vector bundle with fibers \(\mathcal{H}^p(S/B)_b=\mathrm{H}^p(S_b)\). Using local trivializations of \(S\) over \(B\), we obtain an identification \[\mathrm{H}^p(S;p^*\Omega^q(B))\cong \Omega^q(B;\mathcal{H}^p(S/B)).\]
Indeed, over \(p^{-1}(U)\cong U\times F\), we have:\[\mathrm{H}^p(\Omega^{\bullet,q}(U\times M),\mathrm{d}^{1,0})\cong\Omega^q(U)\otimes_\mathbb{R}\mathrm{H}^p(M),\]
and patching these together provides identification. Hence,
\[E^{q,p}_1=\Omega^q(B;\mathcal{H}^p(S/B)).\]
The vector bundle \(\mathcal{H}^*(S/B)\) comes equipped with a flat connection \(\nabla^{\mathrm{GM}}\), which extends linearly to a differential \(\mathrm{d}^{\mathrm{GM}}\) on \(\Omega^{\bullet}(B;\mathcal{H}^p(S/B))\) and corresponds to \(\mathrm{d}_1\) under the previous identification.
The second page of the spectral sequence is therefore given by the following identification
\[E^{q,p}_2=\mathrm{H}^q(B;\mathcal{H}^p(S/B))\cong \mathrm{H}^q(B;\underline{\mathcal{H}}^p_B).\]

The differential \(\mathrm{d}_2\) of the spectral sequence can be can be described as follows. Let \(\omega\in E^{q,p}_2\) be represented by an element \(\eta_0\in \Omega^{p,q}(S)\) be a form representing it. The condition, that it represents a class in \(E^{q,p}_2\) are that \(\mathrm{d}^{1,0}\eta_0=0\) and that there exists \(\eta_1\in \Omega^{p-1,q+1}(S)\), such that \(\mathrm{d}^{1,0}\eta_1+\mathrm{d}^{0,1}\eta_0=0\). The differential is then represented by \[\mathrm{d}_2\omega=[\mathrm{d}^{-1,2}\eta_0+\mathrm{d}^{0,1}\eta_1]\in E_2^{q+2,p-1}.\]

We can therefore state Theorem \ref{Theorem: Obstruction classes} in terms of the de Rham cohomology with coefficients in flat vector bundles.

\begin{thm}
     Let \(S\to B\) be a symplectic fibration, \(\mathcal{H}\) a symplectic connection. There exists a compatible \(2\)-form \(\Omega\in \Omega^2(S)\), such that \(\mathrm{d}\Omega\) is basic and \(\mathcal{
    H}^\Omega=\mathcal{
    H}\) if and only if \([\mathrm{d}^{-1,2}\omega_\mathcal{V}]=0\in \Omega^1(B;\mathcal{H}^2(S/B))\), i.e. if \([\mathrm{d}^{-1,2}\omega_\mathcal{V}]\) is the zero section of \(\Lambda^1T^*B \otimes \mathcal{H}^2(S/B))\).
    
    Moreover, there exists a symplectic connection form with basic differential if and only if \([[\mathrm{d}^{-1,2}\omega_\mathcal{V}]]=0\in \mathrm{H}^1(B;\mathcal{H}^2(S/B))
\) , i.e. \([\mathrm{d}^{-1,2}\omega_\mathcal{V}]\) is exact for \(\mathrm{d}^{GM}\).
\end{thm}

We can furthermore consider the third page of the spectral sequence and its differential in order so see, which further geometric conditions the vanishing of the class under this differential provides.

The third page is by definition given by:
\[E^{q,p}_3=\frac{\ker (\mathrm{d}_2\colon\mathrm{H}^q(\mathrm{H}^p(S,p^*\Omega^\bullet(B)))\to \mathrm{H}^{q+2}(\mathrm{H}^p(S,p^*\Omega^\bullet(B))))}{\mathrm{im} (\mathrm{d}_2\colon\mathrm{H}^{q-2}(\mathrm{H}^p(S,p^*\Omega^\bullet(B)))\to \mathrm{H}^{q}(\mathrm{H}^p(S,p^*\Omega^\bullet(B))))}\]

In particular, \[E^{3,0}_3=\frac{\mathrm{H^3(\mathrm{H^0(S;p^*\Omega^\bullet(B))})}}{\mathrm{im}(\mathrm{d}_2\colon\mathrm{H}^1(\mathrm{H}^1(S;p^*\Omega^\bullet(B)))\to\mathrm{H}^3(\mathrm{H}^0(S;p^*\Omega^\bullet(B))))}.\]
If the fibers of the fibration are connected, then \(\mathrm{H}^0(S;p^*\Omega^\bullet(B))\cong \Omega^\bullet(B)\) and therefore \(\mathrm{H^3}(\mathrm{H}^0(S;p^*\Omega^\bullet(B)))\cong \mathrm{H}^3(B)\). We then obtain \(E^{3,0}_3\cong \frac{\mathrm{H}^3(B)}{\mathrm{im}(\mathrm{d}_2)}\). Given again a class \(\omega\in E^{q,p}_3\), we can pick a representative \(\eta_0\in \Omega^{p,q}(S)\). It represents the class, if \(\mathrm{d}^{1,0}\eta_0=0\) and if we find \(\eta_1\in \Omega^{p+1,q-1}(S)\) and \(\eta_2\in \Omega^{p+2,q-2}(S)\), such that \(\mathrm{d}^{0,1}\eta_0+\mathrm{d}^{1,0}\eta_1=0\) and \(\mathrm{d}^{-1,2}\eta_0+\mathrm{d^{0,1}\eta_1+\mathrm{d}^{1,0}\eta_2=0}\). The differential is then represented by \[\mathrm{d}_3[\eta_0]={[\mathrm{d}^{-1,2}\eta_1+\mathrm{d}^{0,1}\eta_2}]\in E^{p+3,q-2}_3.\]

Let us now apply the differentials to \(\omega_\mathcal{V}\), in order to see what they mean geometrically.
Firstly, \(\omega_\mathcal{V}\in \Omega^{2,0}(S)\cong E^{0,2}_0\). The zeroth differential maps \(\mathrm{d}_0:E^{0,2}_0\to E^{0,3}_0\) The equation \[\mathrm{d}_0\omega_\mathcal{V}=\mathrm{d}^{1,0}\omega_\mathcal{V}=0\]
holds, if and only if \(\omega_\mathcal{V}\) is fiberwise symplectic. This is always true, in our setting. In that case, the class of \(\omega_\mathcal{V}\) descends to \([\omega_\mathcal{V}]\in E^{0,2}_1\cong \Omega^0(B;\mathcal{H}^2(S/B))\). The first differential maps : \(\mathrm{d}_1:E^{0,2}_1\to E^{1,2}_1\). The class \([\omega_\mathcal{V}]\) is \(\mathrm{d}_1\)-closed, if and only if \[\mathrm{d}_1[\omega_\mathcal{V}]=\mathrm{d^{GM}}[\omega_\mathcal{V}]=[\mathrm{d^{0,1}\omega_\mathcal{V}}]=0,\]
i.e. if the class is flat with respect to the Gauss-Manin connection, or equivalently, if it is invariant under the parallel transport of the chosen connection. This is also true, in the setting we are in. In that case, again \([\omega_\mathcal{V}]\) descends further to cohomology and defines a class \([[\omega_\mathcal{V}]]\in E^{0,2}_2\cong \mathrm{H}^0(B;\mathcal{\underline{H}}^2_B)\). The second differential maps \(\mathrm{d}_2:E^{0,2}_2\to E^{2,1}_2\). The class \([[\omega_\mathcal{V}]]\) is \(\mathrm{d}_2\)-closed, i.e.
\[\mathrm{d}_2[[\omega_\mathcal{V}]]=[[\mathrm{d^{-1,2}\omega_\mathcal{V}}]]=0,\]
if and only if \ref{Equation:VerticalExactness} is satisfied, i.e. if there exists a symplectic connection \(\Omega\), such that \(\mathrm{d}\Omega\) is basic. If this is the case, then \([[\omega_\mathcal{V}]]\) descends to a class \([[[\omega_\mathcal{V}]]]\in E^{0,2}_3\). The differential maps \(\mathrm{d}_3:E^{0,2}_3\to E^{3,0}_3\). The class \([[[\omega_\mathcal{V}]]]\) is \(\mathrm{d}_3\) closed i.e. \[\mathrm{d}_3[[[\omega_\mathcal{V}]]]=0\]
if and only if it satisfies the equations, which amounts to the vanishing of the pullback form.

\begin{cor}
    There exists a closed extension \(\Omega\in \Omega^2(S)\) of \(\omega_\mathcal{V}\in \Omega^{2,0}(S)\) if and only if the class \(\mathrm{d}_3[\omega_\mathcal{V}]\in \frac{\mathrm{H}^3(B)}{\mathrm{im}(\mathrm{d_2})}\) vanishes.
\end{cor}

Finally, we point out that in \cite{GotayWeinstein} the authors provide more geometric interpretations of the vanishing of these classes, in terms of reduction of structure groups of the symplectic fibration.

\begin{thm}[\cite{GotayWeinstein},Theorem 4]
    Let \(S\to B\) be a symplectic fibration, \(\Omega\in \Omega^2(S)\), any extension of the vertical symplectic form.
    \begin{enumerate}
        \item The form \(\Omega\) is a symplectic connection, if and only if the parallel transport with respect to the induced connection preserves the fiberwise symplectic structure.
        \item If \(\Omega\) is horizontally closed, the curvature of the induced connection takes values in the Lie algebra of fiberwise locally Hamiltonian vector fields. The form \(\mathrm{d}\Omega\) is basic if and only if the curvature takes values in the globally defined Hamiltonian vector fields. Equivalently, the structure group of the fibration can be reduced to one which admits a moment map.
        \item The form \(\Omega\) is closed if and only if the structure group of the fibration can be reduced to one with an \(\mathrm{Ad}^*\)-invariant moment map.
    \end{enumerate}
Hence, in particular, all the prequantizable symplectic fibrations admit a reduction to the structure group with an \(\mathrm{Ad}^*\)-invariant moment map and up to the integrality those are the ones, which are prequantizable.
    
\end{thm}

\section{Heisenberg calculus for contact fibrations}\label{Section: Heisenberg Calculus}

In this section we recall the construction of Heisenberg calculus on compact contact manifolds. In the whole section, we can consider operators acting on sections of vector bundles. For simplicity however, we will describe the calculus with scalar valued operators. The details of these constructions and their properties can be found in \cite{EpsteinMelrose1998,BaumvanErp2014,vanErpYuncken2019,GorokhovskyvanErp2022} and the references therein.

\subsection{Heisenberg calculus on a contact manifold}

Consider first a contact manifold \((M,H)\). The Heisenberg calculus considers vector fields tangent to \(H\) as differential operators of order 1 but other vector fields as operators of order 2. This assigns a new order to every differential operator. A new pseudodifferential calculus \(\Psi^{\bullet}_H(M)\) is then constructed to allow parametrices to operators elliptic in that calculus. The ellipticity condition in this calculus involves unitary representations in a family of Heisenberg groups naturally associated to the contact manifold \((M,H)\).

The vector bundle \(\mathfrak{t}_HM := H \oplus \faktor{TM}{H}\) has a Lie bracket on its space of sections obtained from the one on \(\mathfrak{X}(M)\). A simple computation shows that this Lie bracket actually comes from a smooth family of Lie algebra structures on the fibers of \(\mathfrak{t}_HM\). Moreover, locally, in Darboux coordinates, this family is easily shown to be a trivial family of Heisenberg Lie algebras. The vector bundle \(\mathfrak{t}_HM\) is then endowed with a structure of locally trivial bundle of Heisenberg Lie algebras. We will denote by \(T_HM\) the bundle of Heisenberg groups integrating this bundle of Lie algebras. The bundle \(\mathfrak{t}_HM\) is called the osculating Lie algebra bundle and \(T_HM\) the osculating group bundle. Note that in more general contexts (with \(H\) not necessarily contact), the Lie algebra structure is not necessarily locally trivial. In that case \(\mathfrak{t}_HM\) and \(T_HM\) are also referred to as the osculating Lie algebroid and osculating groupoid respectively.

We have an action of the positive real numbers on \(\mathfrak{t}_HM\) by Lie algebra automorphisms on the fibers. If \(\lambda > 0\), we get an automorphism:
\[\delta_{\lambda} = \mathrm{diag}(\lambda,\lambda^2) \colon H \oplus \faktor{TM}{H} = \mathfrak{t}_HM \to \mathfrak{t}_HM.\]

In this setting, one can construct a class of pseudodifferential operators \(\Psi^k_H(M), k\in \mathbb{Z}\). Their corresponding Schwartz kernel on \(M\times M\) are properly supported, and smooth off the diagonal. Their behavior near the diagonal is similar to usual pseudodifferential operators in the sense that their Fourier transform (which makes sense near the diagonal) has a polyhomogeneous expansion for the dilations \(\delta_{\lambda}, \lambda > 0\). The precise construction and behavior of these operators will not be needed, we will mostly use their formal properties. The interested reader can see for instance \cite{EpsteinMelroseUnpublished, vanErpYuncken2019}.

\begin{rem}
	If the contact structure is co-oriented, we can pick a contact form \(\theta \in \Omega^1(M)\). We have \(H = \ker(\theta)\) and \(\mathrm{d}\theta\) restricts to a symplectic form on each fiber of \(H\). We can trivialize the bundle \(\faktor{TM}{H}\) using the Reeb field (i.e. the unique vector field \(R \in \ker(\mathrm{d}\theta)\) with \(\theta(R) = 1\)). The bundle of Lie algebras \(\mathfrak{t}_HM\) then becomes isomorphic to the bundle of Heisenberg algebras \(\mathfrak{heis}(H,\mathrm{d}\theta)\) obtained by seeing \(\mathrm{d}_x\theta\) as a cocycle on the abelian Lie algebra \(H_x\) for every \(x \in M\) (it is also trivialized after a 	choice of Darboux coordinates).
\end{rem}

Now for a differential operator \(D\) of order \(k\) in the Heisenberg calculus, we can "freeze the coefficients" at a point \(x \in M\) and obtain a left-invariant operator on the osculating group \(T_{H,x}M\), i.e. an element of the universal enveloping Lie algebra \(\mathcal{U}(\mathfrak{t}_{H,x}M)\). This element depends on a choice of local Darboux coordinates around \(x\) but the highest order part (so in \(\mathcal{U}_k(\mathfrak{t}_{H,x}M)\) where the degree is induced by the grading on the Lie algebra) doesn't. If \(\pi \in \widehat{T_{H,x}M}\) is an irreducible unitary representation, we can also see it as a representation of the universal enveloping Lie algebra by unbounded operators, with domain containing the space of smooth vectors \(\mathcal{H}_{\pi}^{\infty} \subset \mathcal{H}_{\pi}\) (this is actually a core for this operator so it is enough to consider it acting on this space). Denote the corresponding operator by:
\[\sigma^k(D,x,\pi) \colon \mathcal{H}_{\pi}^{\infty} \to \mathcal{H}_{\pi}^{\infty}.\]
We say that \(D\) satisfies the Rockland condition if it is left-invertible for every \(x \in M, \pi \in \widehat{T_{H,x}M} \setminus \{0\}\) where \(0 \in \widehat{T_{H,x}M}\) denotes the trivial representation. This condition is equivalent to the existence of a parametrix in the Heisenberg calculus. This means that we can find an operator \(P \in \Psi^{-k}_H(M)\) with \(PD - 1 \in \Psi^{-\infty}_H(M)\). 
Finally notice that the element of the universal enveloping Lie algebra obtained by freezing the coefficients of \(D\) and keeping the highest order part is homogeneous of degree \(k\) for the dilations \(\delta_{\lambda}, \lambda > 0\). Consequently, identifying \(\mathcal{H}_{\pi}\cong \mathcal{H}_{\delta^*\pi}\) (as \(\delta^*\pi = \pi \circ \delta_{\lambda}\) where \(\delta_{\lambda}\) now seen as a Lie group automorphism) we get:
\[\forall \lambda> 0, \sigma^k(D,x,\delta_{\lambda}^*\pi) = \lambda^k \sigma^k(D,x,\pi).\]

This construction of the principal symbol can be generalized:
\begin{itemize}
\item to pseudodifferential operators, in which case we do not have the intermediate step of having an element of \(\mathcal{U}_k(\mathfrak{t}_{H,x}M)\), but still have a family of unbounded operators on the spaces of non-trivial irreducible unitary representations of the osculating groups. These operators still satisfy the homogeneity condition above (but we can now have \(k \in \mathbb{Z}\) or even \(\mathbb{C}\)).
\item to operators between sections of vector bundles. If \(P\in \Psi^k_H(M;E,F)\) its principal symbol at a point \(x\in M\) and representation \(\pi \in \widehat{T_{H,x}M} \setminus \{0\}\) is then an operator:
\[\sigma^k(D,x,\pi) \colon \mathcal{H}_{\pi}^{\infty}\otimes E_x \to \mathcal{H}_{\pi}^{\infty}\otimes F_x.\]
\end{itemize}

We will denote by \(\Sigma^k_H(M)\) (or \(\Sigma^k_H(M;E,F)\) in case vector bundles are involved) the space of principal symbols of order \(k \in \mathbb{Z}\). We have the exact sequence:
\begin{equation}\label{Equation: Pseudos Exact sequence}
\xymatrix{0 \ar[r] & \Psi^{k-1}_H(M) \ar[r] & \Psi_H^k(M) \ar[r] & \Sigma^k_H(M) \ar[r] & 0.}    
\end{equation}

To understand the ellipticity condition we thus need to understand the representations of the bundle of Heisenberg group. They are of two kinds: the characters (one dimensional), parameterized by \(H^*\), and the others which are infinite dimensional. The latter are parameterized by \(H^{\perp}\setminus M\) the following way. At each point \(x\in M\), every element \(\theta\in H_x^{\perp}\setminus \{0\} \subset T^*_xM\) provides a symplectic form on \(H_x\) by:
\[(X,Y) \mapsto \theta([X,Y]).\]
The pullback bundle \(H \times_M (H^{\perp}\setminus M) \to H^{\perp}\setminus M\) is then symplectic, we fix a complex structure \(J\) on it, compatible with the symplectic structure (this can be done by taking a metric on the bundle). Now let \(H^{1,0} \to H^{\perp}\setminus M\) be the hermitian vector bundle obtained from the complex structure. We can construct the bundle of symmetric Fock spaces:
\[\mathcal{F}^+(H^{1,0}) := \ell^2-\bigoplus_{k = 0}^{+\infty}\mathrm{Sym}^k(H^{1,0}),\]
as the \(\ell^2\)-completion of the orthogonal sum of all the symmetric tensor products of \(H^{1,0}\) with the induced metric. Elements of \(\mathfrak{t}_HM\) act on \(\mathcal{F}^+(H^{1,0})\) using the so called creation and annihilation operators, see e.g. \cite{BaumvanErp2014,Cren2024Toeplitz} which integrates into a group representation. For each \((x,\theta) \in H^{\perp}\setminus M\), we obtain an irreducible unitary representation of \(T_HM\) on \(\mathcal{F}^+(H^{1,0})_{(x,\theta)}\). This representation is characterized by the fact that its form on the center \(\faktor{T_xM}{H_x} \subset T_{H,x}M\) being given by the character \(\exp(\mathrm{i} \theta)\).

In case a contact structure \(\Theta \in \Omega^1(M)\) has been chosen (if the contact structure is co-oriented), we can construct the bundle of Fock spaces directly on \(M\) by using the symplectic vector bundle \((H,\mathrm{d}\Theta)\) instead. In that case for every \(x \in M\) and \(\lambda > 0\), there is a unique irreducible unitary representation of \(T_{H,x}M\) \(\pi_{\lambda}\) characterized by \(\mathrm{d}\pi_{\lambda}(R_x) = \lambda\mathrm{Id}\), where \(R \in \mathfrak{X}(M)\) is the Reeb field. 

This description of the space of representations means that, using the homogeneity of the principal symbols, a principal symbol \(\sigma \in \Sigma^k_H(M)\) can be understood by two families of operators \(\sigma_+ = \sigma(\cdot,+1), \sigma_-= \sigma(\cdot,-1)\) on the fibers of \(\mathcal{F}^+(H^{1,0})\) and a smooth function \(\sigma_0 =\sigma(\cdot,0) \in C^{\infty}(\mathbb{S}^*H)\). It turns out that the equatorial part \(\sigma_0\) can be recovered from both \(\sigma_{\pm}\) so we may omit it. We can see \(\sigma_{\pm}\) as restrictions of the principal symbol to the upper and lower hemispheres of the space of representations \(\widehat{T_HM}\) and \(\sigma_0\) as an equatorial part, which glues both \(\sigma_{\pm}\). However not every pair \(\sigma_+,\sigma_-\) corresponds to an element of \(\Sigma_H^k(M)\). Even if both \(\sigma_{\pm}\) induce the same \(\sigma_0\) they might not glue in a smooth way, see \cite{GorokhovskyvanErp2022}.

Operators of order \(0\) on a compact manifold extend to bounded operators on \(L^2(M)\) and operators of negative order to compact operators. We denote by \(\Psi_H(M)\) the closure of \(\Psi^0_H(M)\) inside the space of bounded operators. We then get from \ref{Equation: Pseudos Exact sequence} the extension of \(C^*\)-algebras:

\begin{equation}\label{Equation: Pseudos Exact Sequence Completed}
    \xymatrix{0 \ar[r] & \mathcal{K}(L^2(M)) \ar[r] & \Psi_H(M) \ar[r] & \Sigma_H(M) \ar[r] & 0.}
\end{equation}

The \(C^*\)-algebra \(\Sigma_H(M)\) is obtained by completion of \(\Sigma_H(M)\). It can be described by a pair of families of Weyl operators on \(H^{1,0}\) and \(H^{0,1}\) whose principal symbols coincide (they correspond to \(\sigma_0\)), see \cite{GorokhovskyvanErp2022}.

\subsection{Families of Heisenberg operators on contact fibrations}

We now explain the case of families. Consider a contact fibration \(\pi \colon X \to B\) with distribution \(\mathcal{H}\), and write \(X_b = \pi^{-1}(\{b\}), b \in B\). We fix \(\Theta \in \Omega^1(M)\) such that \(\ker(\Theta)\cap \mathcal{V} = \mathcal{H}_{\mathcal{V}}\) is the bundle of contact distributions on the fibers.

The vertical vector fields \(\Gamma^\infty(\mathcal{V})\) are stable under Lie bracket, therefore \(\mathcal{V}\) is a Lie algebroid. This Lie algebroid is filtered by \(\mathcal{H}_{\mathcal{V}} \subset \mathcal{V}\). The associated graded vector bundle \(\mathcal{V}_{\mathcal{H}} = \mathcal{H}_{\mathcal{V}} \oplus \faktor{\mathcal{V}}{\mathcal{H}_{\mathcal{V}}}\) is then a bundle of nilpotent Lie algebras (the same way \(\mathfrak{t}_HM\) was). We will also denote the Lie group bundle the same way. If \(\iota_b \colon X_b \to X\) is the inclusion of a fiber, then we have the isomorphism:
\[\iota_b^*\mathcal{V}_{\mathcal{H}} \cong T_{H_b}X_b.\]

We denote by \(\Psi^k_H(X/B)\) the families of operators of order \(k\) on the fibers. We can see these families in two different ways. The first is more topological. If we denote by \((M,H)\) the typical contact manifold fiber of \(X \to B\), we can associate to \(X \to B\) a principal \(\Cont(M,H)_+\)-bundle \(P \to B\) as in Remark \ref{Remark: Contact fibration cocycles}. The group of contact transformations acts by conjugation on the spaces \(\Psi^k_H(M)\), preserving the composition and order of operators. We can then consider the associated (infinite dimensional, Fréchet) vector bundles \(P \times_{\Cont(M,H)_+}\Psi^k_H(M) \to B\). The space \(\Psi^k_H(X/B)\) is then defined as the space of smooth sections of that vector bundle.

The second way to define these operators is to see them directly as properly supported distributions on the groupoid \(X\times_B X = \sqcup_{b\in B} X_b\times X_b \rightrightarrows X \). The algebroid of this groupoid is \(\mathcal{V}\), which is filtered by \(\mathcal{H}_{\mathcal{V}}\subset \mathcal{V}\). The space \(\Psi^k_H(X/B)\) then corresponds to properly supported distributions on \(X \times_B X\), transverse to the range map, and that extends to the adiabatic deformation between \(X\times_B X\) and \(\mathcal{V}_{\mathcal{H}}\) see \cite{vanErpYuncken2019,Mohsen2021Deformation}.

In both cases, when \(X\) is compact, operators of order \(0\) act as bounded operators on the Hilbert bundle \(L^2(X/B) :=(L^2(X_b))_{b\in B}\). Moreover, operators of negative order act as compact operators of this bundle and thus define elements of \(C^*(X\times_BX) \cong \mathcal{K}(L^2(X/B)) \cong C(B,\mathcal{K}(L^2(M)))\). For the last identification, see e.g. \cite{DixmierDouady1963}. We also have the exact sequences similar to \ref{Equation: Pseudos Exact sequence}:

\begin{equation}\label{Equation: Pseudos Exact Sequence Families}
    \xymatrix{0 \ar[r] & \Psi^{k-1}_H(X/B) \ar[r] & \Psi_H^k(X/B) \ar[r] & \Sigma^k_H(X/B) \ar[r] & 0.}
\end{equation}

The space \(\Sigma^k_H(X/B)\) can either be seen as homogeneous distributions of order \(k\) on  the bundle of Heisenberg groups \(\mathcal{V}_{\mathcal{H}}\), or as sections of the associated vector bundle \(P\times_{\Cont(M,H)_+}\Sigma^k_H(M)\).
For \(k = 0\) and \(X\) compact, we can, as for \ref{Equation: Pseudos Exact Sequence Completed} with a single contact manifold, take the respective \(C^*\)-algebraic closures and get:

\begin{equation}\label{Equation: Pseudos Exact Sequence Families Completed}
    \xymatrix{0 \ar[r] & C(B,\mathcal{K}(L^2(M))) \ar[r] & \Psi_H(X/B) \ar[r] & \Sigma_H(X/B) \ar[r] & 0.}
\end{equation}

\section{Szegö projections and Toeplitz families}\label{Section: Szego and Toeplitz families}

\subsection{Families of Szegö projections}

In \cite{EpsteinMelrose1998}, Epstein and Melrose generalize the definition of Szegö projections using the Heisenberg calculus. This coincides with the prior generalization made by Guillemin and Boutet de Monvel in \cite{BoutetdeMonvelGuillemin1981} using Fourier integral operators but simplifies their understanding. Regardless of these differences, they generalize the Szegö projector onto the Hardy space of the disk (which acts on \(L^2(\mathbb{S}^1)\) by cutting out the negative Fourier modes).

Let \((M,\theta)\) be a (co-oriented) contact manifold, \(H = \ker(\theta)\). Let \(J \in \mathrm{End}(H)\) be a complex structure compatible with the symplectic form \(\omega = \mathrm{d}\theta\). This means that \(\omega(J\cdot,\cdot) > 0\) and \(\omega(J\cdot,J\cdot) = \omega(\cdot,\cdot)\). In coordinates \((\xi,\eta) \in \mathfrak{t}_HM^*\), we define:
\[s_0(\xi,\eta) = \exp(-\omega(J\xi,\xi)/\eta),\quad \eta >0\]
and extended by \(0\) when \(\eta \leq 0\). This function corresponds to a principal symbol of order \(2\), \(s_0 \in \Sigma^0_H(M)\) (notice the homogeneity). Its equatorial part, as well as its negative part, both vanish (it is called a Hermite symbol). If we fix \(x \in M\) and consider the representation \(\Sigma_H^0(M) \to \mathcal{B}(\mathcal{F}^+(H^{1,0})_x)\) given by \(\sigma \mapsto \sigma_+(x)\), the image of \(s_0\) is given by the orthogonal projection onto the one dimensional subspace \(\mathrm{Sym}^0(H^{1,0})_x\).

\begin{defn}
    A ground state projection is a principal symbol \(\sigma \in \Sigma^0_H(M)\) of the previous form for an appropriate choice of compatible complex structure \(J\).
\end{defn}

\begin{rem}
    Since the space of compatible complex structure is path connected, we can always find a path between two ground states projections within this subset of \(\Sigma^0_H(M)\). Also, notice that we need to make a coherent choice of positive representation, hence the co-orientation requirement.
\end{rem}

\begin{rem}
    In \cite{EpsteinMelrose1998}, the authors also consider "level \(k\)" and "up to level \(k\)" projections of the previous kind, by taking symbols which project onto \(\mathrm{Sym}^k(H^{1,0})\) and \(\oplus_{j \leq k}\mathrm{Sym}^j(H^{1,0})\) respectively (again, for an appropriate choice of compatible complex structure on \(H\)). The terminology comes from the fact that the spaces \(\mathrm{Sym}^j(H^{1,0})\) correspond to eigenspaces for a quantum harmonic oscillator. The number \(j\) then corresponds to the energy level of the states of that eigenspace. The space \(\mathrm{Sym}^0(H^{1,0})\) is then the space of vectors of lowest energies, the ground state or vacuum state (it is only of dimension 1 since the representation is irreducible).
\end{rem}

\begin{defn}
    A Szegö projector is an element \(S \in \Psi^0_H(M)\) which is a projection, i.e. \(S^2 = S\), and such that its principal symbol is a ground state projection. We denote the set of such projections by \(\Sz(M)\subset \Psi^0_H(M)\) (this set does depend on the contact structure).
    In the presence of a vector bundle \(E\to M\), a Szegö projection will be a projection \(S \in \Psi^0_H(M;E)\) with principal symbol of the form \(s_0\otimes\mathrm{Id}_E\) for a ground state projection \(s_0\). We denote by \(\Sz(M;E)\) the space of such Szegö projections.
\end{defn}

\begin{rem}
    Again, one can consider "level \(k\)" and "up to level \(k\)" Szegö projections by modifying the requirement on the principal symbol. We will only consider Szegö projections as in the previous definition. Although the index theorem of Section \ref{Section: Family index} also works for these operators, the star-product of Section \ref{Section: Star products} only uses the first kind of Szegö projections.
\end{rem}

For our purposes with contact fibrations, we want to relax the assumption that \(S\) has to be a projector, as families of such operators may fail to exist.

\begin{defn}
    A formal Szegö projector is an element \(S \in \Psi^0_H(M)\), whose principal symbol is a ground state projection, and such that \(S^2 - S \in \Psi^{-\infty}_H(M)\). We denote the set of such operators by \(\FSz(M)\subset \Psi^0_H(M)\), or \(\FSz(M;E)\) if the operators act on a vector bundle \(E \to M\).
\end{defn}

The Szegö projections form a subset of formal Szegö operators. The advantage of working with the latter is that the space \(\FSz(M;E)\) is contractible \cite[Lemma 33]{EpsteinMelroseUnpublished}. On the other hand, few is known about the topology of \(\Sz(M;E)\). It is known, when \(M\) is compact, that its connected components are parameterized by \(\mathbb{Z}\), two operators belong to the same connected component if and only if their relative index is zero (see \cite{EpsteinMelrose1998} for more details). To our knowledge nothing is known about the topology of these connected components. Moreover, in the non-compact case, it is very unlikely that \(\Sz(M)\) is even non-empty, as a projector might not be properly supported.

We now generalize these notions to families of operators. This can be done two different ways. Consider the action of \(\Cont(M,H)_+\) on \(M\). Since it preserves the co-orientation, it preserves the (conformal) symplectic structure on \(H\). Therefore its action on \(\Psi^0_H(M)\) by conjugation preserves the subspace \(\Sz(M)\) of Szegö projectors, and \(\FSz(M)\) of formal Szegö projections. Consequently, if \(X \to B\) is a contact fibration with fiber \((M,H)\), and \(P \to B\) denotes the corresponding \(\Cont(M,H)_+\)-principal bundle we get associated fiber bundles \(P \times_{\Cont(M,H)_+} \Sz(M) \to B\) and \(P \times_{\Cont(M,H)_+} \FSz(M) \to B\). 

\begin{defn}
    A family of Szegö projections is a global section of the fiber bundle \(P \times_{\Cont(M,H)_+} \Sz(M) \to B\). Equivalently, it is an element \(S \in \Psi^0_H(X/B)\) such that its restriction on each fiber \(S_b \in \Psi_H^0(X_b), b\in B\) is a Szegö projection. We denote by \(\Sz(X/B)\) the space of families of Szegö projections. 
    
    Likewise, we denote by \(\FSz(X/B)\) the set of families of formal Szegö operators. It is the set of global sections of \(P \times_{\Cont(M,H)_+} \Sz(M) \to B\).

    Finally, we denote by \(\Sz(X/B;E), \FSz(X/B;E)\) families of Szegö projections (formal Szegö operators respectively) acting on a vector bundle \(E \to X\), defined analogously as in the non-family case.    
\end{defn}

If we have a contact fibration \(X \to B\) with form \(\Theta\), and fix a complex structure \(J \colon \mathcal{H}_{\mathcal{V}} \to \mathcal{H}_{\mathcal{V}}\) compatible with the symplectic forms \(\mathrm{d}\Theta\), we obtain a bundle of symmetric Fock spaces \(\mathcal{F}^+(\mathcal{H}_{\mathcal{V}}^{1,0}) \to X\). This bundle restricts over each fiber \(X_b\) to the symmetric Fock space of the contact manifold \((X_b,\theta_b)\) with compatible complex structure \(J|_{X_b}\). In particular if \(S \in \FSz(X/B)\) is a family of formal Szegö operators, its principal symbol \(\sigma^0(S) \in \Sigma^0_H(X/B)\) gives for each \(x \in X\) a ground state projection on the fiber \(\mathcal{F}^+(\mathcal{H}_{\mathcal{V}}^{1,0})_x\).

\begin{thm}
    Let \(X \to B\) be a contact fibration. The space \(\FSz(X/B)\) is non-empty. More precisely, given a field of ground state projections \(s_0 \in \Sigma^0_H(X/B)\), there exists a family of formal Szegö projections \(S \in \FSz(X/B)\) such that \(\sigma^0(S) = s_0\).

    Finally, if \(G\) is a compact group (e.g. \(\mathbb{S}^1\)) acting on the fiber \((M,H)\) such that the cocycles of \(X \to B\) are \(G\)-equivariant, then there exists a family of Szegö projections which is \(G\)-invariant.
\end{thm}

\begin{proof}
    This follows easily from the fact that for the typical fiber \(M\), the space \(\FSz(M)\) is non-empty and contractible. The set of sections \(\FSz(X/B)\) is thus non-empty as well. Let us prove there exist equivariant ones. Let \(G\) be a compact group acting on \(M\) by co-orientation preserving contactomorphisms and such that the cocycle are \(G\)-invariant. Then \(G\) acts on \(X\) fiberwise and we can assume that it preserves \(\Theta\) (otherwise we can work with a new contact form which will be equivariant and conformal to the initial one). This action preserves the vertical contact distribution. The action of \(G\) on the space of metrics on \(\mathcal{V}\) has a fixed point, the corresponding almost-complex structure is then also \(G\)-equivariant. The ground state projection \(s_0 \in \Sigma^0(X/B)\) given by this complex structure is then \(G\)-invariant. The group \(G\) then acts on the space 
    \[\{P \in \Psi_H^0(X/B) \mid \sigma^0(P) = s_0\}\]
    which is affine. We can thus take an equivariant operator \(P \in \Psi_H^0(X/B)\) with principal symbol \(s_0\). We now use \(P_0\) to build a formal Szegö operator. Since \(s_0\) is a projector then \(P_0^2-P_0 \in \Psi^{-1}_H(X/B)\). We can define iteratively for \(k\geq 0\)
    \begin{align*}
        D_k &= (P_k^2 - P_k)(\mathrm{Id} - 2P_k)\\
        P_{k+1} &= P_k +D_k.
    \end{align*}
    A simple computation shows that \(D_k \in \Psi^{-2^k}_H(X/B)\). By asymptotic completeness of the calculus, there is an operator \(S \in \Psi^0_H(X/B)\) such that for every \(k\geq 0\),
    \[S - P_k \in \Psi^{-2^{k}}_H(X/B).\]
    This operator is by construction a formal Szegö operator. All the operators in this process are equivariant except maybe \(S\). The operator obtained by asymptotic completeness is uniquely defined modulo \(\Psi^{-\infty}_H(X/B)\). The particular choice of \(S\) we did will thus be \(G\)-equivariant modulo \(\Psi^{-\infty}_H(X/B)\). Looking at the action of \(G\) on the affine space \(S + \Psi^{-\infty}_H(X/B)\) we can find a fixed point.
\end{proof}

\begin{rem}
    Notice that the construction of equivariant sections also provides the proof of existence of families of formal Szegö operators, without relying on the contractibility of the space \(\FSz(M)\). The proof of this fact turns out to be in the same spirit as the proof above.
\end{rem}

We have to resort to families of formal Szegö operators as families of Szegö projections may fail to exist, as we show an obstruction in the next example given by the contact degree \cite{EpsteinMelrose1998}. Consider a compact contact manifold \(M\). Given a contactomorphism \(\varphi \in \Cont(M,H)_+\) and a Szegö projection \(S \in \Sz(M)\), the operator \(S_{\varphi} := (\varphi^*)^{-1}S\varphi^*\) is also a Szegö projection. The relative index of \(S\) and \(S_{\varphi}\) does not depend on \(S\) and is called the contact degree of \(\varphi\). This is a group homomorphism
\[\operatorname{c-deg} \colon \Cont(M,H)_+ \to \mathbb{Z}.\]

\begin{exmp}
    Let \(M\) be a compact contact manifold and \(\varphi \in \Cont(M,H)_+\) a contactomorphism with non-zero contact degree. Let 
    \[T_{\varphi} := \faktor{M\times [0;1]}{(x,0)\sim (\varphi(x),1)}\] 
    be the mapping torus. This is a compact contact fibration over the circle \(\mathbb{S}^1\). The space \(\Sz(T_{\varphi}|\mathbb{S}^1)\) is empty.
    
    Indeed, assume that a global section \(S\) exists. It corresponds to a continuous family of operators \(S_t, t\in [0;1]\) with \(S_1 = (S_0)_{\varphi}\). Then by continuity \(S_1\) and \(S_0\) are in the same connected component of \(\Sz(M)\) which forces \(\mathrm{c-deg}(\varphi) = 0\), contradicting the assumption.
\end{exmp}

Conversely, if the structure group of a compact contact fibration can be reduced to the kernel of \(\mathrm{c-deg}\), then this primary obstruction does not exist anymore. It is not clear however that even in that case one could find a continuous family of Szegö projections, as we lack information about the topology of \(\Sz(M)\).

\subsection{Families of Toeplitz operators}

\begin{defn}
    Let \(X\to B\) be a contact fibration and fix a family of formal Szegö projections \(S\in \FSz(X/B)\). The set of Toeplitz families of order \(k\in \mathbb{Z}\) is the subspace of operators 
    \[\mathcal{T}^k(X/B) = S \Psi^k_H(X/B) S + \Psi^{-\infty}_H(X/B).\] 
    If a vector bundle \(E \to X\) is involved,  we will denote the Toeplitz operators acting on the sections of \(E\) by \(\mathcal{T}^k(X/B;E)\) (in both case we omit the reference to \(S\) in the notation).
\end{defn}

\begin{rem}\label{Remark: Toeplitz families restriction}
    With the same notations, the restriction on a fiber \(X_b, b\in B\) gives the Toeplitz operators \(\mathcal{T}(X_b)\) for the formal Szegö projection \(S_b \in \Psi^0_H(X_b)\), as considered in e.g. \cite{Melrose2004}.
\end{rem}

We now fix a contact fibration \(X \to B\) and a family of formal Szegö projections \(S \in \FSz(X/B)\) with principal symbol \(s_0\). If \(P \in \mathcal{T}^k(X/B)\) we have \(P = SP = PS\) modulo regularizing operators because \(S^2 \equiv S \mod \Psi^{-\infty}_H(X/B)\). This gives \(\sigma^k(P)s_0 = s_0\sigma^k(P)\). Therefore, \(\sigma^k(P)\) can only be non-zero in the positive representation and in this representation, it commutes with the projector \(s_0\) which has one dimensional range. The symbol \(\sigma^0(P)\) thus reduces to a smooth function on \(f\in C^{\infty}(X)\) through the identification:
\[\forall x \in X, \sigma_+^k(P)(x) = s_0(x) f(x) s_0(x).\]
Another characterization of this function, for operators \(P \in \mathcal{T}^0(X/B)\), is the following. With the same notations, the function \(f\) is the unique function such that if we write \(T_f:= S\mathcal{M}_fS\), where \(M_f \in \Psi^0_H(M)\) is the operator of multiplication by \(f\) then:
\[P - T_f \in \mathcal{T}^{-1}(X/B).\]
Indeed, both operators have the same principal symbol so we just have to show that Toeplitz operators of order \(0\) that are in \(\Psi^{-1}_H(X/B)\) also have order \(-1\), as Toeplitz operators. If \(SPS\) is such an operator, with \(P\in \Psi^0_H(X/B)\), writing \(R = S^2 - S \in \Psi^{-\infty}_H(X/B)\) we get
\[SPS = S(SPS)S - RPS - SPR - RPR.\]
We have \(RPS, SPR, RPR \in \Psi^{-\infty}_H(X/B)\) so this proves that \(SPS \in \mathcal{T}^{-1}(X/B)\).
More generally, with the same reasoning and the exact sequence \ref{Equation: Pseudos Exact Sequence Families}, we get the exact sequence:

\begin{equation}\label{Equation: Exact sequence Toeplitz families}
    \xymatrix{0 \ar[r] & \mathcal{T}^{k-1}(X/B) \ar[r] & \mathcal{T}^k(X/B) \ar[r] & C^{\infty}(X) \ar[r] & 0.}
\end{equation}

\begin{thm}\label{Theorem: Asymptotic completeness}
    Let \(P_k \in \mathcal{T}^{m_k}(X/B), k \geq 0\) be a sequence of operators with the order \(m_k, k\geq 0\) decreasing to \(-\infty\). Then there exists \(P \in \mathcal{T}^{m_0}(X/B)\) such that:
    \[P \sim \sum_{k = 0}^{+\infty}P_k,\]
    i.e. 
    \[\forall k \geq 0, P - \sum_{j = 0}^k P_j \in \mathcal{T}^{m_{k+1}}(X/B).\]
\end{thm}

\begin{proof}
    Let \(S\in\FSz(X/B)\) be the underlying family of Szegö projectors. We have \(\mathcal{T}^m(X/B) \subset \Psi^m_H(X/B)\). The calculus \(\Psi^{\bullet}_H(X/B)\) is asymptotically complete (see \cite[Theorem 59]{vanErpYuncken2019}), thus we may find \(Q \in \Psi^{m_0}_H(X/B)\) such that:
    \[Q \sim \sum_{k = 0}^{+\infty}P_k. \]
    Now set \(P = SQS\). We have \(P \in \mathcal{T}^{m_0}(X/B)\) and since for every \(k \geq 0\) we have \(SP_kS \equiv P_k \mod \Psi^{-\infty}_H(X/B)\), then:
    \[P \sim  \sum_{k = 0}^{+\infty}P_k.\]
\end{proof}

\subsection{Fredholm property}

In this subsection, we work with the assumption that \(X\to B\) is a compact, contact fibration, and that we have a family of Szegö projections \(S\in \Sz(X/B)\). As mentioned before this is not always possible (even with \(X\) compact). This subsection leads to the computation of the index of families of Toeplitz operators that will be performed in Section \ref{Section: Family index}. This index computation only depends on the principal symbol so in particular can still be done when \(\Sz(X/B)\) is empty (even if it looses some of its meaning). 

Since we work with an actual projection, we modify the notion of Toeplitz operators. In this subsection only, we will work with
\[\mathcal{T}^k(X/B) = S\Psi^k_H(X/B)S.\]
Since \(S\) is a projection, the exact sequence \ref{Equation: Exact sequence Toeplitz families} still holds with this definition of Toeplitz operators.

Since \(\mathcal{T}^k(X/B) \subset \Psi^k_H(X/B)\) and \(X\) is compact, Toeplitz families of order \(0\) are bounded on the Hilbert bundle \(L^2(X/B)\) and Toeplitz families of negative order are compact. The norm closure of \(\mathcal{T}^{-1}(X/B)\) is the corner subalgebra \(S \mathcal{K}(L^2(X/B))S\). Since on each fiber \(\mathrm{im}(S_b)\) is infinite dimensional, the space \(S_b\mathcal{K}(L^2(X_b))S_b\) is isomorphic to \(\mathcal{K}(\mathrm{im}(S_b))\) and hence isomorphic to \(\mathcal{K}\). Therefore the corner subalgebra \(S \mathcal{K}(L^2(X/B))S\) is the algebra of compact operators on the Hilbert bundle \(\mathrm{im}(S) = (\mathrm{im}(S_b))_{b\in B}\to B\) (this bundle can be seen as the bundle of Hardy spaces corresponding to \(S\)). Consequently we have isomorphisms:
\[S \mathcal{K}(L^2(X/B))S \cong \mathcal{K}(\mathrm{im}(S)) \cong C(B,\mathcal{K}).\]
If \(S\) is not a projector, then \(\mathrm{im}(S)\) is not necessarily a Hilbert bundle over \(B\) anymore, and \(S \mathcal{K}(L^2(X/B))S\) is not necessarily a subalgebra. The completion of order \(-1\) Toeplitz operators with the more general definition would then give all compact operators on \(L^2(X/B)\).

Likewise, the closure of the algebra of order \(0\) operators is \(S\Psi_H(M)S\). We will simply denote it by \(\mathcal{T}(X/B)\). Taking the norm closure from exact sequence \ref{Equation: Exact sequence Toeplitz families} (or compose with \(S\) on both sides of \ref{Equation: Pseudos Exact Sequence Families Completed}) we obtain:

\begin{equation}\label{Equation: Exact sequence Toeplitz families completed}
    \xymatrix{0 \ar[r] & C(B,\mathcal{K}) \ar[r] & \mathcal{T}(X/B) \ar[r] & C(X) \ar[r] & 0.}
\end{equation}

\begin{rem}
    All the algebras above are fibered over the space \(B\). In that regard, the algebra \(C(X)\) (or \(C^{\infty}(X)\)) should be seen as the algebra of (global) sections of the bundle \((C(X_b))_{b\in B}\) through the obvious identification.
\end{rem}

\begin{thm}\label{Theorem: Fredholm Condition}
    Let \(E \to X\) be a complex hermitian vector bundle. Let \(f \in C(X,\mathrm{End}(E))\) be a continuous section of the endomorphism bundle. The operator \(T_f \in \mathcal{T}(X/B;E)\) is a continuous family of Fredholm operators if and only if for every \(x \in X\), \(f(x)\) is invertible, i.e. \(f \in C(X,\mathrm{GL}(E))\).
\end{thm}

\begin{proof}
    In the presence of a vector bundle, the exact sequence \ref{Equation: Exact sequence Toeplitz families completed} becomes:
    \[\xymatrix{0 \ar[r] & C(B,\mathcal{K}) \ar[r] & \mathcal{T}(X/B;E) \ar[r] & C(X,\mathrm{End}(E)) \ar[r] & 0.}\]
    Now \(T_f\) is a continuous family of Fredholm operators if and only if it is invertible modulo \(C(B,\mathcal{K})\). But the image of \(T_f\) under that quotient is precisely \(f\).
\end{proof}

\begin{rem}
    In the above proof, the algebra \(C(B,\mathcal{K})\) comes in a similar way as in the scalar case. First, the Hilbert bundle is replaced by the one of \(L^2\)-sections of the bundle \(E\), \(L^2(X/B;E) = (L^2(X_b,E_b))_{b\in B}\) with \(E_b = \iota_b^*E\) and \(\iota_b\colon X_b \to X\) the inclusion of the fiber. The algebra \(C(B,\mathcal{K})\) corresponds to the compact operators on the subbundle \(\mathrm{im}(S) \subset L^2(X/B;E)\).
\end{rem}

We will solve the index problem for this type of operators in Section \ref{Section: Family index} by identifying its index bundle with the index bundle of a particular Dirac operator. Before that, notice that \(1-S\) is invertible on \(\mathrm{im}(S)^{\perp}\). Therefore the family index of \(T_f\) is equal to the index of the Heisenberg-elliptic operator \(T_f + (1-S)\).

If \(S\) is not a projector anymore, the completion of \(\mathcal{T}^{-1}(X/B)\) now contains all the compact operators on \(L^2(X/B)\). A Toeplitz operator \(T_f\) is then Fredholm under the same conditions as in Theorem \ref{Theorem: Fredholm Condition}. Its family index is also the same as \(T_f + (1-S)\) so the computations from Theorem \ref{Corollary: Toeplitz index computation} still applies in that context. However this time this is not the family index of \(T_f\) acting on \(\mathrm{im}(S)\). This does make sense from the point of view of index theory as the index only depends on the principal symbol of an operator. The principal symbol does not see if \(S\) is an actual projection or not.

\section{Applications to star-products and deformation quantization}\label{Section: Star products}

In this subsection, we consider a prequantizable symplectic fibration \(S \to B\) (not necessarily compact), and its corresponding contact fibration \(X\to B\). We will see functions on \(S\) as functions on \(X\) that are invariant for the \(\mathbb{S}^1\)-action on the fibers. We will consider a family of formal Szegö projections \(P \in\FSz(X/B)\) which is \(\mathbb{S}^1\)-invariant.

\begin{thm}\label{Theorem: Commutator of Toeplitz}
    Let \(f,g \in C^{\infty}(S)\) and consider the Toeplitz families \(T_f,T_g\) obtained by contracting the multiplication operator with the family of formal Szegö projectors \(P\). Then the operator \([T_f,T_g]\) has order \(-1\), and its principal Toeplitz symbol is:
    \[\sigma^{-1}\left(\left[T_f,T_g\right]\right) = -\mathrm{i}\lbrace f,g\rbrace.\]
    where \(\lbrace \cdot,\cdot\rbrace\) denotes the Poisson bracket on \(C^{\infty}(S)\).
\end{thm}

\begin{proof}
    For \(b\in B\), denote by \(f_b,g_b \in C^{\infty}(S_b)\) the respective restrictions of \(f,g\). The algebra of principal symbols is commutative so \(\sigma^0\left(\left[T_f,T_g\right]\right) = 0\) and \(\left[T_f,T_g\right]\) has thus order \(-1\). Moreover, the operator \(\left[T_f,T_g\right]\) is equivariant. Consequently its principal symbol is also a function on \(S\). For \(x \in S\), we have:
    \[\sigma^{-1}\left(\left[T_f,T_g\right]\right)(x) = \sigma^{-1}\left(\left[T_{f_{\pi(x)}},T_{g_{\pi(x)}}\right]\right)(x).\]
    From \cite[Proposition 11.9]{BoutetdeMonvelGuillemin1981}, see also \cite[Eqn. (289), p.90]{EpsteinMelroseUnpublished}, and writing \(b = \pi(x)\), we get:
    \[\sigma^{-1}\left(\left[T_{f_{b}},T_{g_{b}}\right]\right) = -\mathrm{i}\left\lbrace f_{b},g_{b}\right\rbrace_{S_b}.\]
    We then use the fact that \(\lbrace f,g\rbrace_S(x) = \lbrace f_{b},g_{b}\rbrace_{S_b}(x)\), we get the result.
\end{proof}

Recall that for a Poisson manifold \((\mathcal{P},\{\cdot,\cdot\})\), a star-product is a product law \(\ast\) on the formal power series \(C^{\infty}(\mathcal{P})[[\hbar]]\) such that:
\begin{itemize}
    \item \((C^{\infty}(\mathcal{P})[[\hbar]],+,\ast)\) is an associative algebra
    \item \(\forall f,g \in C^{\infty}(\mathcal{P}), f\ast g = fg + \mathcal{O}(\hbar)\).
    \item \(\forall f,g \in C^{\infty}(\mathcal{P}), [f,g]_{\ast} = -\mathrm{i}\hbar \lbrace f,g\rbrace + \mathcal{O}(\hbar^2)\).
\end{itemize}

These rules should be understood as follows: we want to quantize the algebra of classical observables \(C^{\infty}(\mathcal{P})\) into the non-commutative algebra \(C^{\infty}(\mathcal{P})[[\hbar]]\) of quantum observables. When the quantum scales \(\hbar\) goes to zero, we want to recover the classical observables. The commutator of two quantum observables, which has now necessarily order at least one, is then given by the Poisson structure.

In the case where \(P\) is a symplectic manifold, a wide variety of constructions exist, see e.g. \cite{Weinstein1995Bourbaki} for an overview. Here we follow the analytic construction given by Guillemin \cite{Guillemin1995}.

\begin{rem}
    The normalization condition:
    \[\forall f,g \in C^{\infty}(\mathcal{P})[[\hbar]], [f,g]_{\ast} = -\mathrm{i}\hbar \lbrace f,g\rbrace + \mathcal{O}(\hbar^2),\]
    is a bit different from the literature, where the convention \([f,g]_{\ast} = \mathrm{i}\hbar \lbrace f,g\rbrace + \mathcal{O}(\hbar^2)\) is mostly used. These conventions only differ from a change of variables (replacing \(\hbar\) by \(-\hbar\)) and do not affect any of the results thereafter.
\end{rem}

The Poisson manifold we are interested in here is the total space \(S\) of a symplectic fibration \((S \overset{p}{\to} B,\Omega)\). We assume this fibration to be prequantizable, and denote by \((X \overset{\pi}{\to} B,\Theta)\) the corresponding contact fibration. As before, we consider \(P \in \FSz(X/B)^{\mathbb{S}^1}\), a family of formal Szegö projections, invariant under the action of the \(\mathbb{S}^1\)-action. Denote by \(R \in \Gamma^\infty(\mathcal{V})\) the Reeb vector field.

\begin{thm}
    There is an isomorphism of filtered vector spaces:
    \[\faktor{\mathcal{T}^0(X/B)^{\mathbb{S}^1}}{\mathcal{T}^{-\infty}(X/B)^{\mathbb{S}^1}} \cong C^{\infty}(S)[[\hbar]].\]
    Moreover this isomorphism restricts to every fiber into the isomorphism of \cite[Lemma 3]{Guillemin1995}:
    \[\faktor{\mathcal{T}^0(X_b)^{\mathbb{S}^1}}{\mathcal{T}^{-\infty}(X_b)^{\mathbb{S}^1}} \cong C^{\infty}(S_b)[[\hbar]],\quad  b\in B.\]
    Consequently, the product structure on Toeplitz operators induces a star-product on the Poisson manifold \(S\) which restricts to each fiber into the star-product of the symplectic manifold \(S_b, b\in B\) from \cite[Theorem 2]{Guillemin1995}.
\end{thm}

\begin{proof}
    Consider \(D_R = -\mathrm{i}\mathcal{L}_R \in \Psi^2_H(X/B)\). 
    The corresponding Toeplitz operator \(T_R = PD_RP \in \mathcal{T}^2(X/B)^{\mathbb{S}^1}\) has principal symbol \(x \mapsto 1\). 
    We may find an operator \(T \in \mathcal{T}^{1}(X/B)^{\mathbb{S}^1}\) such that \(T^2 = T_R\) modulo smoothing operators. Moreover, \(T\) is also elliptic as a Toeplitz operator so we may consider a parametrix \(T^{-1} \in \mathcal{T}^{-1}(X/B)^{\mathbb{S}^1}\) also well defined modulo smoothing. Consider now \(A \in \mathcal{T}^0(X/B)^{\mathbb{S}^1}\). By \ref{Equation: Exact sequence Toeplitz families}, we can write:
    \[A = P\mathcal{M}_fP + R,\]
    where \(f \in C^{\infty}(X)\) is the principal symbol of \(A\) and \(R \in \mathcal{T}^{-1}(X/B)\). Since \(A\) is equivariant, then so is \(f\). In particular we may consider \(f \in C^{\infty}(S)\). The operator \(R\) is also equivariant so that \(TR \in \mathcal{T}^0(X/B)^{\mathbb{S}^1}\). We may apply the same result again and get:
    \[TR = P\mathcal{M}_gP + R',\]
    which rewrites as:
    \[A = P\mathcal{M}_fP + T^{-1}P\mathcal{M}_gP + T^{-1}R',\]
    where now \(T^{-1}R' \in \mathcal{T}^{-2}(X/B)^{\mathbb{S}^1}\). We may proceed inductively and write:
    \[A =\sum_{i = 0}^{+\infty}T^{-i} P\mathcal{M}_{f_i}P \mod \mathcal{T}^{-\infty}(X/B)^{\mathbb{S}^1},\]
    with \(f_i \in C^{\infty}(X)^{\mathbb{S}^1} = C^{\infty}(S)\). These functions \(f_i, i\geq 0,\) can be obtained iteratively as the principal symbol of a (equivariant) Toeplitz operator: 
    \[\forall i\geq 0, f_i = \sigma^0\left(T^i\left(A - \sum_{j = 0}^{i-1} T^{-j}P\mathcal{M}_{f_j}P\right)\right).\] 
    Conversely, by asymptotic completeness of Toeplitz calculus from Theorem \ref{Theorem: Asymptotic completeness}, given such a family of functions \(f_i, i\geq 0,\) we may construct an operator \(A\) with the previous asymptotic expansion. Two such operators would be equal modulo smoothing operators. This proves the first part of the Theorem.

    To show that the product of Toeplitz operators induces a star product, we have to check the 3 axioms. The associativity comes from the associativity of composition of pseudodifferential operators. The second condition comes form the fact that \(\sigma^0(T_f) = f\). Finally, the last property is  Theorem \ref{Theorem: Commutator of Toeplitz}.

    Using Remark \ref{Remark: Toeplitz families restriction} we obtain the last part of the Theorem. 

\end{proof}

\section{The family index theorem}\label{Section: Family index}

The goal of this section is to derive a family index theorem for families of Rockland operators on compact contact fibrations, and apply it to families of Toeplitz operators. This is a generalization of an index theorem of Baum and van Erp \cite{BaumvanErp2014}, from which we adapt the proof. The main ingredient of their proof is to use the geometric K-homology in the sense of Baum-Douglas \cite{BaumDouglas1980}. Any elliptic operator (or Rockland in our case) on a manifold \(M\) defines a K-homology class \([P] \in K_*(M)\) (with \(* = 0\) or \(1\) depending on the presence or not of a \(\mathbb{Z}/2\)-grading). On the other hand, Baum and Douglas gave an analytic description of K-homology using the so-called geometric cycles. These are basically K-homology classes of a \(\mathrm{Spin}^c\)-Dirac operator on a \(\mathrm{Spin}^c\) manifold \(X\) with a map \(X \to M\). These cycles, subject certain relations (see below), form an abelian group \(\mathrm{K}^{geo}_0(M)\). Pushing forward the K-homology class of a cycle to \(M\) defines an assembly map:
\[\mu \colon \mathrm{K}^{geo}_0(M) \to \mathrm{K}_0(M),\]
which turns out to be an isomorphism \cite{BaumDouglas1980,BaumHigsonSchick2007,BaumvanErp2016Elliptic}. Since one of the easiest index formulas is the one for Dirac operators, Baum and Douglas state the index problem as follows: given an elliptic operator \([P]\), find the unique (up to equivalence) geometric cycle \(\xi\) representing \([P]\), i.e. satisfying \(\mu(\xi) = [P]\).

Here we work in a bivariant case. A family of elliptic operators on a fibration \(X \to B\) defines a class in Kasparov bivariant K-theory \(\KK_*(X,B):=\KK_*(C(X),C(B))\) from \cite{Kasparov1980}. To replace the geometric K-homology, we use a bivariant version due to Connes and Skandalis \cite{ConnesSkandalis1981CRAS,ConnesSkandalis1984}, see also \cite{EmersonMeyer2010}.

The proof of the index theorem of Baum and van Erp then proceeds as follows: in the classical calculus, an elliptic symbol on \(M\) defines a K-theory class of the cotangent bundle \(\K^0(T^*M)\). Any pseudodifferential operator with this symbol will yield the same K-homology class. This defines a group homomorphism:
\[\Op \colon \K^0(T^*M) \to \K_0(M).\]
This homomorphism turns out to be invertible and is the K-theoretic version of Poincaré duality. Similarly in the Heisenberg setting, a Rockland symbol defines a K-theory class in \(\K_0(C^*(T_HM))\) and we still have a Poincaré duality produced the same way:
\[\Op_H \colon \K_0(C^*(T_HM)) \to \K_0(M).\]
Finally, using a Connes-Thom argument \cite{Connes1981ThomIsom}, one can construct an isomorphism \(\Psi \colon K^0(T^*M) \to K_0(C^*(T_HM))\), so that we get the commutative triangle:
\[\begin{tikzcd}
\K_0(C^*(T_HM)) \arrow[rrd, "\Op_H"]            &  &         \\
                                                &  & \K_0(M) \\
\K^0(T^*M) \arrow[uu, "\Psi"] \arrow[rru, "\Op"'] &  &        
\end{tikzcd}\]

Baum and van Erp then proceed to define similar maps to the geometric K-homology groups. This produces an explicit geometric cycle from an elliptic or Rockland symbol, which corresponds to the K-homology class of the operator through the assembly map. Both of these constructions are done through clutching constructions. In the next subsections, we first recall the definition of the bivariant geometric K-homology, and then proceed to adapt the proof of \cite{BaumvanErp2014} to this bivariant setting.

\begin{rem}
    It is worth mentioning that all the results of this section work with submersions whose fibers are contact manifolds (that vary smoothly) without changing anything to the proofs. In this context however, we cannot ensure the existence of a Szegö operator which would be a genuine projection and thus of Toeplitz operators as defined before. Nevertheless, the index only depends on the principal symbol of an operator. We can thus still talk about the index of operators which have the same principal symbols as the ones Toeplitz operators would have.
\end{rem}

\subsection{Bivariant geometric K-homology}\label{Subsection: Geometric KK}

We recall here the construction of geometric bivariant K-homology, following Connes and Skandalis \cite{ConnesSkandalis1981CRAS,ConnesSkandalis1984}. The cycles in this theory are called correspondences. Throughout this subsection, we fix spaces \(X,Y\). In full generality, we could take \(X\) to be a locally compact Hausdorff space and \(Y\) a manifold. For what is of interest to us (ultimately, compact contact fibrations) we will consider \(X,Y\) compact manifolds.

A correspondence between \(X\) and \(Y\) is a quadruple \((M,E,f_X,g_Y)\) where \(M\) is a smooth manifold, \(E \to M\) a complex vector bundle, \(f_X\colon M \to X\) is a proper map and \(g_Y \colon M \to Y\) is a K-oriented map. The latter means that the vector bundle \(TM \oplus g_Y^*(TY)\) has a \(\mathrm{Spin^c}\)-structure.

Given a K-oriented map \(g_Y \colon M \to Y\), Connes and Skandalis define a wrong way map in K-theory: \(g_{Y!} \in \KK_0(M,Y)\). The general construction is quite subtle since one needs to decompose \(g_Y\) as a composition of an immersion and a submersion, the construction of the wrong-way map being different in both cases. The only ones that will appear in this paper will be submersions so we recall the construction for that case. If \(g_Y \colon M \to Y\) is a K-oriented submersion then the vertical bundle \(\ker(\mathrm{d}g_Y)\) has a \(\mathrm{Spin}^c\)-structure \(S\). We can consider the vertical Dirac operator 
\[\mathrm{D}_{\mathcal{V}}\colon C^{\infty}(M,S) \to C^{\infty}(M,S)\]
Since the vertical Dirac operator is elliptic on the fibers, it defines a regular unbounded operator on the Hilbert bundle \(L^2(M/Y;S) \to Y\). The algebra \(C(M)\) acts on this Hilbert module by left multiplication \(\mathcal{M}\). We then obtain a Kasparov cycle:
\[g_{Y !} = [(L^2(M/Y;S),\mathcal{M},\mathrm{D}_{\mathcal{V}}] \in \KK_*(M,Y).\]
The value of \(*\) depends on the fact that \(S\) is graded or not, which only depends on the dimension of the fibers.

\begin{rem}
    This construction can be used by replacing \(\mathrm{D}_{\mathcal{V}}\) with any pseudodifferential operator which is longitudinally elliptic or Rockland with respect to the fibration (see subsection \ref{Subsection: Choosing an operator} below for more precisions). If \(P\) is such an operator on a fibration \(M\to Y\) (possibly acting on sections of a vector bundle such as for \(\mathrm{D}_{\mathcal{V}}\)), we will simply refer to 
    \[ [P] \in KK_0(M,Y)\]
    for the class in KK-theory obtained this way. Therefore in our notations, \(g_{Y!} = [\mathrm{D}_{\mathcal{V}}]\). The class \([P]\) is the type of class that we want an index theorem for. For elliptic families this is the Atiyah-Singer index theorem for families, and we want to generalize it to Rockland families.
\end{rem}

Given a correspondence \((M,E,f_X,g_Y)\) we may define an analytic class in KK-theory \(\mu(M,E,f_X,g_Y) \in \KK_*(X,Y)\) the following way (again, \(*=0,1\) depending on dimension of the \(\mathrm{Spin}^c\)-structure). We have a class \([E] \in \KK_*(pt,M)\) so we may take its Kasparov product with \(g_{Y!}\) to obtain \([E] \otimes_Mg_{Y!} \in \KK_*(M,Y)\). We then take a push-forward by \(f_X\) to obtain:
\[\mu(M,E,f_X,g_Y) := f_{X*}([E] \otimes_M g_{Y!} ) \in \KK_*(X,Y).\]

\begin{rem}\label{Remark: KK class is vertical twisted Dirac}
    Notice that in the submersion case where the class \(g_{Y!}\) is given by a vertical Dirac operator, the meaning of the Kasparov product becomes:
    \[[E] \otimes_M g_{Y!} = [E] \otimes_M \left[\mathrm{D}_{\mathcal{V}}\right] = \left[\mathrm{D}_{\mathcal{V}}^E\right],\]
    where \(\mathrm{D}_{\mathcal{V}}^E\) is the vertical twisted Dirac operator, acting on sections of \(S \otimes E\).
\end{rem}

We may define a group structure on the set of correspondences by the disjoint union, the inverse will be given by reversing the K-orientation on \(g_Y\). It becomes the group after modding out by the following relations (which we don't state precisely since we won't use them):
\begin{itemize}
    \item (Bordism) If a cycle is the boundary of a cycle defined the same way using a manifold with boundary, its class is equivalent to zero
    \item (Addition) If we have two vector bundles \(E_1,E_2 \to M\) then \((M,E_1\oplus E_2,f_X,g_Y) \sim (M,E_1,f_X,g_Y) + (M,E_2,f_X,g_Y)\).
    \item (Vector bundle modification) For a certain sphere bundle \cite[Paragraph 10]{BaumDouglas1980} \(p \colon M' \to M\) and vector bundle \(E' \to M'\), we have \((M,E,f_X,g_Y) \sim (M',E',f_X\circ p, g_Y\circ p)\).
\end{itemize}

Under this equivalence relation, we obtain a group \(\KK^{geo}_*(X,Y)\), with \(* = 0,1\) depending on the parity of the dimension of the \(\mathrm{Spin}^c\)-structure (which is consistent with the grading of the underlying Dirac operator). Moreover, the map \(\mu\) described above induces a group homomorphism (still denoted the same way) called the assembly map:
\[\mu \colon \KK^{geo}_*(X,Y) \to \KK_*(X,Y).\]
This map is in fact an isomorphism \cite{ConnesSkandalis1981CRAS,EmersonMeyer2010}. Moreover, the Kasparov product on the analytic side has a natural analog on the geometric side given by the composition of correspondences. If \([M,E,f_X,g_Y] \in \KK^{geo}_*(X,Y), [M',E',f_Y,g_Z] \in \KK^{geo}_*(Y,Z)\) are such that \(g_Y\) and \(f_Y\) are transverse, then we may construct the correspondence:
\[(M\times_YM', \mathrm{pr}_1^*E \otimes\mathrm{pr}_2^*E',f_X \circ \mathrm{pr}_1,g_Z\circ \mathrm{pr}_2),\]
and denote its class by:
\[[M,E,f_X,g_Y] \otimes_Y[M',E',f_Y,g_Z] \in \KK^{geo}_*(X,Z).\]

This class is well defined and we may always find representatives that meet the transversality condition (see \cite[Paragraph 3]{ConnesSkandalis1984}). The assembly map is then compatible with the product structures on both sides.

\begin{rem}
    In the definition of the geometric product, the grading at the end is compatible with the addition of gradings modulo 2.
\end{rem}

\subsection{Fundamental classes of a contact fibration}\label{Subsection: Fundamental classes}
In this section we will describe the vertical fundamental classes of co-oriented compact contact fibrations \(X\to B\), corresponding to the two different \(\mathrm{Spin}^c\) orientations of the vertical tangent bundle \(\mathcal{V}\). We will denote the these classes by \([(X/B)^\pm]\). Analytically these classes are realized by the vertical Dirac operators \([\mathrm{D}^\pm_{\mathcal{V}}]\in \mathrm{KK}_1(X,B)\). Geometrically we can realize them as correspondences \([X,\mathbb{\underline{C}},\mathrm{id}_X,\pi]^\pm\in \mathrm{KK}_1^{geo}(X,B)\), where we again choose the opposite \(\mathrm{Spin}^c\)-orientations.

These two classes are actually the same up to a sign. Indeed the opposite \(\mathrm{Spin}^c\)-structure uses the opposite complex structure on \(\mathcal{H}\) and reverses the orientation on the line bundle \(\faktor{\mathcal{V}}{\mathcal{H}}\). Therefore we have:
\[[\mathrm{D}^+_{\mathcal{V}}] = (-1)^{n+1}[\mathrm{D}^-_{\mathcal{V}}],\]
where \(n\) is the complex rank of \(\mathcal{H}\), i.e. the fibers of \(X \to B\) have dimension \(2n+1\).

\begin{rem}
    This relation shows that depending on the parity of \(n\), our two \(\mathrm{Spin}^c\)-structures may coincide. We will call them opposite nonetheless.
\end{rem}

\subsection{Choosing an operator}\label{Subsection: Choosing an operator}

We fix a compact contact fibration \(X \to B\), with a co-orientation of the contact structure, and denote by \(\mathcal{V}_{\mathcal{H}} \to X\) the corresponding vertical osculating group bundle. In this subsection we construct maps:
\begin{align*}
    \Op &\colon \K^0(\mathcal{V}^*) \to \KK_0(X,B), \\
    \Op_H &\colon \K_0(C^*(\mathcal{V}_{\mathcal{H}})) \to \KK_0(X,B).
\end{align*}

These maps do the following: starting with an elliptic (or Rockland respectively) vertical symbol, we choose a family of operators having this symbol as principal symbol. These operators induce a class in \(\KK_0(X,B)\) which only depends on the class of the symbol in K-theory.

Every class in \(\K^0(\mathcal{V}^*)\) can be represented by a triplet \((\sigma,E,F)\) where \(\sigma\) is a vector bundle map:
\[\sigma \colon \pi^*E \to \pi^*F,\]
with \(\pi \colon \mathcal{V}^* \to X\) the projection map. Such a map has to be invertible outside of a compact subset (typically the zero section). For practical applications, \(\sigma\) may be homogeneous outside of that compact set (take a homogeneous symbol on \(\mathcal{V}^*\setminus X\) and extend it to \(\mathcal{V}^*\) by using a cutoff function near the zero section). Given such a K-theory cycle, we may construct a family of pseudodifferential operators:
\[P \colon C^{\infty}(X,E) \to C^{\infty}(X,F)\]
on the fibration \(X \to B\) with symbol \(\sigma\) (so that \(P\) is a family of elliptic operators on the fibers). The class \([P] \in \KK_0(X,B)\) does only depend on the class \([(\sigma,E,F)] \in K^0(\mathcal{V}^*)\).

This correspondence is well explained with the adiabatic groupoid:
\[G^{ad} = X\times_BX \times (0;1] \sqcup \mathcal{V} \times \{0\}.\]
This groupoid is constructed as a deformation to the normal cone (see e.g. \cite{DebordSkandalis2014}), and has thus a natural structure of Lie groupoid. Consider the map 
\[\ev_0 \colon C^*(G^{ad}) \to C^* (\mathcal{V}) = C_0(\mathcal{V}^*).\]
It induces the exact sequence:
\[\xymatrix{0 \ar[r] & C^*(X\times_B X) \otimes C_0((0;1]) \ar[r] & C^*(G^{ad}) \ar[r] & C_0(\mathcal{V}^*) \ar[r] & 0.}\]
The kernel is contractible and the quotient nuclear so the element 
\[\ev_0 \in \KK_0(C^*(G^{ad}),C_0(\mathcal{V}^*))\] 
is invertible. 
Recall that if \(A\) is a \(C(X)\)-algebra, there is a group homomorphism:
\[\alpha_X \colon \KK_*(\mathbb{C},A) \to \KK_*(C(X),A),\]
obtained by adding the natural representation of \(C(X)\) as central multipliers of \(A\) in any Kasparov cycle. This map is natural in \(A\). A particular case is for \(A = C(M)\) where \(M\) is a \(X\)-space, or more generally to \(A = C^*(G)\) where \(G\) is a groupoid and its units \(G^{(0)}\) are an \(X\)-space. This is the case with \(G^{ad}\) since its units are \(X\times [0;1]\).
We then obtain:
\[\Op = \alpha_X(\ev_0^{-1} \otimes_{\mathcal{V}^*} -) \otimes_{C^*(G^{ad})} \ev_1 \colon\K^0(\mathcal{V^*}) \to \KK_0(X,B),\]
where we use the Morita equivalence \(X\times_B X \sim B\) to replace \(C^*(X\times_B X) \) by \(C(B)\).

The construction of \(\Op_H\) is done in a similar fashion. This time, since the algebroid \(\mathcal{V}\) is filtered by \(\mathcal{H} \subset \mathcal{V}\), we may consider the filtered adiabatic groupoid \cite[Section 12]{vanErpYuncken2019}:
\[G^{ad}_H = X\times_B X \times (0;1] \sqcup \mathcal{V}_{\mathcal{H}} \times \{0\} \rightrightarrows X \times [0;1].\]

The bundle of Heisenberg groups \(\mathcal{V}_{\mathcal{H}}\) is an amenable groupoid so its \(C^*\)-algebra is still nuclear and the corresponding element:
\[\ev_{0,H} \in \KK_0(C^* (G^{ad}_H),C^*(\mathcal{V}_{\mathcal{H}})),\]
is still invertible. We define analogously:

\[\Op_H = \alpha_X(\ev_{0,H}^{-1} \otimes_{C^*(\mathcal{V}_{\mathcal{H}})} -) \otimes_{C^*(G^{ad}_H)} \ev_{1,H} \colon\K_0(C^*(\mathcal{V}_{\mathcal{H}})) \to \KK_0(X,B),\]

To relate these two constructions, note that the Connes-Thom isomorphism induces an isomorphism:
\[\Psi \colon \K^0(\mathcal{V}^*) \to \K_0(C^*(\mathcal{V}_{\mathcal{H}})).\]
This can be seen by applying directly the Connes-Thom isomorphim inductively (a nilpotent group can be obtained as a iteration of crossed products by \(\mathbb{R}\)). We construct it in an analogous way as \(\Op\) and \(\Op_H\) above, following \cite{Mohsen2022}. Consider the adiabatic deformation of \(\mathcal{V}_H\) seen as a groupoid over \(M\). By fixing an identification between the Lie algebroid of \(\mathcal{V}_H\) (the vector bundle \(\mathcal{H}\oplus \faktor{\mathcal{V}}{\mathcal{H}}\)) and \(\mathcal{V}\), we obtain:
\[\mathcal{V}_{\mathcal{H}}^{ad} = \mathcal{V}_{\mathcal{H}} \times (0;1] \sqcup \mathcal{V} \times \{0\} \rightrightarrows M \times [0;1].\]

We then have:
\[\Psi = -\otimes_{\mathcal{V}^*} \ev_0^{-1}\otimes_{C^*(\mathcal{V}_{\mathcal{H}}^{ad})} \ev_1 \colon \K^0(\mathcal{V}^*) \to \K_0(C^*(\mathcal{V}_{\mathcal{H}})).\]
This map coincides with the Connes-Thom map used inductively and is thus an isomorphism.

\begin{rem}
    One can also describe the inverse map as follows. Consider \(\mathcal{V}\) and its filtration. Its filtered adiabatic groupoid \(\mathcal{V}_H^{ad}\) induces a map:
    \[-\otimes_{C^*(\mathcal{V}_{\mathcal{H}})} \ev_0^{-1}\otimes_{C^*(\mathcal{V}_{{H}}^{ad})} \ev_1 \colon \K_0(C^*(\mathcal{V}_{\mathcal{H}})) \to \K^0(\mathcal{V}^*).\]
    This map turns out to be the inverse of \(\Psi\) (again, this is a consequence of the fact that it coincides with an iteration of the Connes-Thom isomorphism).
\end{rem}

\begin{lem}\label{Lemma: Kasparov cycle}
    Let \(P \in \Psi^0_{(H)}(X/B;E,F)\) be a family of self-adjoint elliptic or Rockland operators. Consider \(E \oplus F\) as a \(\mathbb{Z}/2\)-graded vector bundle. Let \(\widetilde{P} = \begin{pmatrix}0 & Q \\ P & 0\end{pmatrix}\) with \(Q \in \Psi^0(X/B;F,E)\) a parametrix for \(P\). Let \(L^2(X/B;E,F)\) be the \(\mathbb{Z}/2\)-graded Hilbert bundle over \(B\) obtained by completion of \((C^{\infty}(X_b,E\oplus F))_{b\in B}\). Let \(C(X)\) act on this Hilbert bundle by multiplication operators on the fibers, denoted by \(\mathcal{M}\). The triple \((L^2(X/B;E,F), \mathcal{M},\widetilde{P})\) is a Kasparov cycle. The same is true if \(P \in \Psi^k(X/B;E,F)\) if we then replace \(P\) by \(P(1+{P}^*{P})^{-1/2}\) instead. 
\end{lem}

\begin{proof}
    The proof are the same for elliptic and Rockland, we prove the lemma in the Rockland case. The case of higher order operators reduces to order 0. Moreover, replacing \(P\) by \(\widetilde{P}\), we may assume that \(E = F\) and \(P^2-1\) has order \(-1\). But then, we have \(P^2-1 \in \mathcal{K}(L^2(X/B;E))\), and so does \(\mathcal{M}(f)(P^2-1)\) for every \(f \in C(X)\). Finally, if \(f \in C^{\infty}(X)\), then \([P,\mathcal{M}(f)]\in \Psi^0_H(X/B;E)\). Its principal symbol vanishes:
    \[\sigma^0_H([P,\mathcal{M}(f)]) = [\sigma^0_H(P),\sigma^0_H(\mathcal{M}(f))],\]
    and \(\sigma^0_H(\mathcal{M}(f))\) is the operator by \(f(x)\mathrm{Id}_{\mathcal{H}_{\pi}}\) in every representation of \(\pi\) \(\mathcal{V}_{\mathcal{H},x}, x \in X\), so the commutator vanishes. Therefore \([P,\mathcal{M}(f)]\) has order \(-1\) and is then compact on each fiber.
\end{proof}

We now turn to the principal symbol. A principal symbol \(\sigma \in \Sigma^k_H(X/B)\) can be lifted (through Schwartz kernel methods \cite{vanErpYuncken2019}) as a family of compactly supported distributions \(p\) on the fibers of \(\mathcal{V}_{\mathcal{H}}\), which are quasi-homogeneous:
\[\forall \lambda > 0, \delta_{\lambda*} p - \lambda^k p \in C^{\infty}_c(\mathcal{V}_{\mathcal{H}}).\]
Here \(\delta_{\lambda}\) denotes the family of Heisenberg dilations on the fibers. Two choices \(p,p'\) represent the same principal symbol \(\sigma\) if and only if \(p-p'\in C^{\infty}_c(\mathcal{V}_{\mathcal{H}})\).
Denote by \(\mathcal{E}^{'k}(\mathcal{V}_{\mathcal{H}})\) the set of such distributions, so that we have:
\[\Sigma^k_H(X/B) \cong \faktor{\mathcal{E}^{'k}(\mathcal{V}_{\mathcal{H}})}{C^{\infty}_c(\mathcal{V}_{\mathcal{H}})}.\]
These distributions act by convolution on \(C^{\infty}_c(\mathcal{V}_{\mathcal{H}})\). Those of order \(0\) extend bounded multipliers of \(C^*(\mathcal{V}_{\mathcal{H}})\) and those of negative order, to elements of \(C^*(\mathcal{V}_{\mathcal{H}})\). Moreover, the Rockland condition for an element in \(\mathcal{E}^{'k}(\mathcal{V}_{\mathcal{H}})\) is equivalent to the existence of an inverse in \(\mathcal{E}^{'-k}(\mathcal{V}_{\mathcal{H}})\) modulo \(C^{\infty}_c(\mathcal{V}_{\mathcal{H}})\).

\begin{lem}
    An element \(p \in \mathcal{E}^{'k}(\mathcal{V}_{\mathcal{H}})\) that satisfies the Rockland condition induces a class in \(K_0(C^*(\mathcal{V}_{\mathcal{H}}))\) that only depends on its image in \(\Sigma^k_H(X/B)\).
\end{lem}
\begin{proof}
    The proof is identical to the previous Lemma \ref{Lemma: Kasparov cycle}. Using a parametrix, we may replace \(p\) by another distribution of order \(0\), and that satisfies \(p^2 -1 \in C^{\infty}(\mathcal{V}_{\mathcal{H}})\). Since this new \(p\) is of order 0, it defines a bounded multiplier of \(C^*(\mathcal{V}_{\mathcal{H}})\). We then get a Kasparov cycle. This cycle does not change if we replace \(p\) by \(p + f\) where \(f \in C^{\infty}_c(\mathcal{V}_{\mathcal{H}})\) as this change does not affect the parametrix. Therefore the class in \(\K_0(C^*(\mathcal{V}_{\mathcal{H}}))\) only depends on the principal symbol determined by \(p\).
\end{proof}

We now relate the two classes using the map \(\Op_H\) and show some compatibility with the Connes-Thom isomorphism.

\begin{thm}\label{Theorem: Op behaviour and commutation}
    If \(P \in \Psi^k_{(H)}(X/B;E,F), k \geq 0\) is a family of elliptic or Rockland operators, and \(\sigma_{(H)}(P)\) its principal symbol then we have the relation:
    \[\Op_{(H)}([\sigma_{(H)}(P)]) = [P] \in \KK_0(X,B).\]
    Moreover, the following diagram commutes:
    \[\begin{tikzcd}
    \K_0(C^*(\mathcal{V}_{\mathcal{H}})) \arrow[rrd, "\Op_H"]            &  &         \\
                                                &  & \KK_0(X,B) \\
    \K^0(\mathcal{V}^*) \arrow[uu, "\Psi"] \arrow[rru, "\Op"'] &  &        
    \end{tikzcd}\]
\end{thm}

\begin{proof}
    The first part works in the same way for both operators we do it in the Rockland case (and for lightening the notations, with scalar operators). Let \(\mathbb{P} \in \mathcal{E}'_r(G^{ad}_H)\) be a quasi-homogeneous distribution as in \cite{vanErpYuncken2019} with \(\mathbb{P}_1 = P\). Then \(\mathbb{P}_0\) induces a K-theory class that coincides with the one of \(\sigma_H(P)\). Then \(\mathbb{P}\) induces a class:
    \[[\mathbb{P}] \in \KK_0(\mathbb{C},C^*(G^{ad}_H)).\]
    This class is constructed in the same way as the one of \(P\) in the previous lemma, but in family over \([0;1]\). The continuity at \(0\) comes from the estimates of \cite[Corollary 45]{vanErpYuncken2019}. Now we have \(\ev_{0,H}([\mathbb{P}]) = [\mathbb{P}_0] = [\sigma_H(P)]\). On the other hand we have \(\alpha_X([\mathbb{P}]) \otimes \ev_{1,H} = [P]\) which concludes the first part of the proof.

    The commutativity of the diagram is established by looking at the square given by the groupoid \(\mathbb{G} := (G_H^{ad})^{ad} \rightrightarrows M \times [0;1]^2\). If we consider coordinates \((s,t) \in [0;1]^2\) then we have that:
    \begin{itemize}
        \item The restriction \(\mathbb{G}_{s,\bullet}\) with \(s \neq 0\) is \(G^{ad}_H\).
        \item The restriction \(\mathbb{G}_{\bullet, t}\) with \(t \neq 0\) is \(G^{ad}\).
        \item The restriction \(\mathbb{G}_{0,\bullet}\) is identified with \(\mathcal{V} \times [0;1]\).
        \item The restriction \(\mathbb{G}_{\bullet, 0}\) is \(\mathcal{V}_{\mathcal{H}}^{ad}\).
    \end{itemize}
    Given the construction of \(\Op,\Op_H,\Psi\) using the rows and columns of this groupoid, we obtain the commutativity of the diagram.
\end{proof}

\begin{rem}\label{Remark: Not Poincaré duality}
    In non family case (so when \(B\) is a point), the map
    \[\Op \colon \K^0(T^*X) \to \K_0(X),\]
    is the Poincaré duality map in K-theory and is thus an isomorphism. This is not the case anymore in the bivariant case. Indeed, in the situation where \(X = M \times B\), we have \(\mathcal{V} = TM \times B\). Connes and Skandalis prove \cite[Corollary 3.11]{ConnesSkandalis1984} that the map \(\Op\) composed with \(\KK_0(X,B) \to \KK_0(M,B)\) (push-forward by the projection \(X \to M\)) is an isomorphism. Since the latter is in general not an isomorphism, \(\Op\) won't be one either. \(\Op\) will then be injective but not surjective. The reason for this is that the group \(\K^0(\mathcal{V}^*)\) encodes information about the map \(X \to B\) while the bivariant group \(\KK_0(X,B)\) only sees the two spaces.
\end{rem}

\subsection{Decomposing the class of a family of Rockland symbols}


In the previous subsection, we introduced the K-theory class of a Rockland symbol, and linked it to the K-homology class of an operator having this principal symbol. In this subsection we describe how the principal symbol class in \(\K_0(C^*(\mathcal{V}_{\mathcal{H}}))\), can be decomposed into two parts using the representation theory of the Heisenberg group. These two parts will allow us in the next subsection to define geometric cycles using clutching constructions.
This more topological interpretation of the symbol follows \cite{BaumvanErp2014}.

Recall from Section \ref{Section: Heisenberg Calculus} that a principal symbol \(\sigma \in \Sigma^k_H(X/B)\) has 3 components:
\begin{itemize}
    \item \(\sigma_{\pm}\), both taking values in the family of infinite dimensional representations of the osculating groups (divided in a positive and negative parts by the co-orientation).
    \item the equatorial symbol \(\sigma_0\), taking values in the characters of the osculating groups.
\end{itemize}

This dichotomy in the representations defines an exact sequence:

\[\xymatrix{0  \ar[r] & I_{\mathcal{H}} \ar[r] & C^*(\mathcal{V}_{\mathcal{H}}) \ar[r] & C_0(\mathcal{H}^*) \ar[r] & 0.}\]

The algebra \(I_{\mathcal{H}}\) is isomorphic to two copies of \(C_0(X \times \mathbb{R}^*_+,\mathcal{K})\).

These algebras of compact operators can be made explicit. Take a complex structure \(J \colon \mathcal{H} \to \mathcal{H}\) compatible with the symplectic form \(\mathrm{d}\Theta\) (where \(\Theta\) is the family of contact forms on the contact fibration). We then have a canonical identification with the algebras of continuous sections:
\begin{align*} 
I_{\mathcal{H}} \cong C(\mathbb{R}^*_-)\otimes \mathcal{K}(\mathcal{F}^+(\mathcal{H}^{0,1})) \oplus C(\mathbb{R}^*_+)\otimes \mathcal{K}(\mathcal{F}^+(\mathcal{H}^{1,0})).
\end{align*}

\begin{rem}
    In the above identification, we see \(\mathcal{K}(\mathcal{F}^+(\mathcal{H}^{0,1}))\) as the sections of the bundle of compact operators over \(X\) associated to the Hilbert bundle \(\mathcal{F}^+(\mathcal{H}^{0,1})\). Therefore the algebra \(\mathcal{K}(\mathcal{F}^+(\mathcal{H}^{0,1}))\) is Morita equivalent to \(C(X)\). The same goes for \(\mathcal{K}(\mathcal{F}^+(\mathcal{H}^{1,0}))\).
\end{rem}

If \(\sigma \in \Sigma^k_H(X/B;E,F)\) where \(E,F\to X\) are complex vector bundles, we get vector bundle morphisms:
\begin{align*}
\sigma_{+} &\colon E \otimes \mathcal{F}^+(\mathcal{H}^{1,0}) \to  F \otimes \mathcal{F}^+(\mathcal{H}^{1,0}) \\
\sigma_{-} &\colon E \otimes \mathcal{F}^+(\mathcal{H}^{0,1}) \to  F \otimes \mathcal{F}^+(\mathcal{H}^{0,1}) \\
\sigma_{eq} &\colon \pi^*E \to \pi^*F,\quad \pi \colon \mathbb{S}^*\mathcal{H} \to X.
\end{align*}

The Rockland condition can be equivalently stated as the invertibility of these 3 maps. 

\begin{prop}\label{Proposition: Reduction of equatorial symbol}
    Let \(\sigma \in \Sigma^k_H(X/B)\) be a family of Rockland symbols. The map \(\sigma_0  \colon \pi^*E \to \pi^*F\) is, up to stabilization of \(E\) and \(F\), homotopic to the pullback of a vector bundle isomorphism \(\sigma'_0 \colon E \to F\).
\end{prop}

We delay the proof of that proposition for a bit. We will now denote by \(\sigma_0\) the It therefore makes sense to define the isomorphisms of vector bundles over \(X\):
\begin{align*}
    \sigma_0^{-1}\circ \sigma_+ &\colon E \otimes \mathcal{F}^+(\mathcal{H}^{1,0}) \to E \otimes \mathcal{F}^+(\mathcal{H}^{1,0}) \\
    \sigma_0^{-1}\circ \sigma_- &\colon E \otimes \mathcal{F}^+(\mathcal{H}^{0,1}) \to E \otimes \mathcal{F}^+(\mathcal{H}^{0,1}).
\end{align*}

These two automorphisms of vector bundles thus define classes:
\[[\sigma_0^{-1}\circ \sigma_{\pm} ] \in \K^1(X).\]

\begin{rem}
    Technically the bundle \(\mathcal{F}^+(\mathcal{H}^{1,0})\) has infinite rank, so won't define a class in K-theory. However, it appears as a limit of finite rank subbundles. We may therefore see the element \([\sigma_0^{-1}\circ \sigma_{-} ] \in K^1(X)\) as the stable limit of the morphisms of the bundle
    \[\bigoplus_{k = 0}^N\operatorname{Sym}^k(\mathcal{H}^{0,1}) \otimes E\]
    given by compression of \(\sigma_0^{-1}\circ \sigma_{-}\), for \(N\) big enough. The same goes for \([\sigma_0^{-1}\circ\sigma_+]\).
\end{rem}

We will explicitly relate these classes to the class of \(\sigma\) in Subsection \ref{Subsection: Clutching}. In preparation of that, and in order to prove Proposition \ref{Proposition: Reduction of equatorial symbol}, we show the following:

\begin{prop}\label{Proposition: Lift symbol class}
    The map \(K_0(I_{\mathcal{H}}) \to K_0(C^*(\mathcal{V}_{\mathcal{H}}))\) is surjective.
\end{prop}

\begin{proof}
    Both \(C^*(\mathcal{V}_{\mathcal{H}})\) and \(C_0(\mathcal{V}^*)\) have a quotient isomorphic to \(C_0(\mathcal{H}^*)\). The Connes-Thom map \(\Psi\) induces the following diagram (the dotted arrows are isomorphims in the category \(\KK\) or \(\KK^X\), the horizontal ones are actual morphisms of \(C^*\)-algebras):
    \[\xymatrix{0 \ar[r] & I_{\mathcal{H}} \ar[r] & C^*(\mathcal{V}_{\mathcal{H}}) \ar[r]  & C_0(\mathcal{H}^*) \ar[r]  & 0 \\
    0 \ar[r] & C_0(\mathcal{V}^*\setminus \mathcal{H}^*) \ar[r] \ar@{.>}[u] & C_0(\mathcal{V}^*) \ar@{.>}[u]^{\Psi} \ar[r] & C_0(\mathcal{H}^*) \ar[r] \ar@{.>}[u] & 0.}\]
    This diagram thus induces the commutative square in K-theory:
    \[\xymatrix{\K_0(I_{\mathcal{H}}) \ar[r] & \K_0(C^*(\mathcal{V}_{\mathcal{H}})) \\
                \K^0(\mathcal{V}^* \setminus \mathcal{H}^*) \ar[u] \ar[r] & \K^0(\mathcal{V}^*) \ar[u]^{\Psi}}\] 
    Since the vertical arrows are isomorphisms, the proposition is equivalent to the surjectivity of the map \(\K^0(\mathcal{V}^* \setminus \mathcal{H}^*) \to \K^0(\mathcal{V}^*) \). Now the inclusion of each connected component of \(\mathcal{V}^* \setminus \mathcal{H}^*\) to \(\mathcal{V}^*\) is a homotopy equivalence, therefore the map \(\K^0(\mathcal{V}^* \setminus \mathcal{H}^*) \to \K^0(\mathcal{V}^*)\) is surjective, and so is \(K_0(I_{\mathcal{H}}) \to K_0(C^*(\mathcal{V}_{\mathcal{H}}))\).    
\end{proof}

\begin{proof}[Proof of Proposition \ref{Proposition: Reduction of equatorial symbol}]
   Let us denote by \(q\colon C^*(\mathcal{V}_{\mathcal{H}}) \to C_0(\mathcal{H}^*)\) the quotient map. By Proposition \ref{Proposition: Lift symbol class}, the map \(K_0(I_{\mathcal{H}}) \to K_0(C^*(\mathcal{V}_{\mathcal{H}}))\) is surjective. Using the 6-terms exact sequence in \(\K\)-theory, we deduce that \(q_* \colon \K_0(C^*(\mathcal{V}_{\mathcal{H}})) \to \K_0(C_0(\mathcal{H}^*))\) is the zero map. In particular, since \([\sigma_0] = q_*([\sigma]) = 0\). This implies that, up to stabilization of \(E\) and \(F\), we can perturb \(\sigma_0\) to a map:
   \[\sigma_0' \in \Gamma(\mathcal{H}^*, \mathrm{Iso}(\pi^*E,\pi^*F)).\]
   If we define \(\sigma_{eq} := \sigma'|_{M\times \{0\}\subset\mathcal{H}^*}\), then \(\sigma_0\) is homotopic to \(\pi^*\sigma_{eq}\).
\end{proof}

\subsection{Clutching constructions}\label{Subsection: Clutching}

Let \(X\to B\) be a compact contact fibration. In this subsection we define the the maps:
\begin{align*}
    c:\mathrm{K}^0(\mathcal{V^*})\to \mathrm{KK}_0^{geo}(X,B)\\
    b:\mathrm{K}_0(C^*(\mathcal{V_H}))\to \mathrm{KK}_0^{geo}(X,B)
\end{align*}
 and show the compatibility with the Connes-Thom map. 

The maps \(c\) and \(b\) are the family versions of the (noncommutative) Poincaré duality maps defined in \cite{BaumvanErp2014}, although in the family setting this terminology is misleading, as the maps are not necessarily isomorphisms (see Remark \ref{Remark: Not Poincaré duality} above). Both maps are incarnations of the clutching construction. The proofs of the analogous statements for contact manifolds, presented in (\cite{BaumvanErp2014}, Section 4) extend almost vertabim to the family setting. For the sake of completeness, we will recall these ideas and adjust them to contact fibrations.

We begin with the definition of \(c\).
Choosing a vertical symbol class \([\sigma]\in \mathrm{K^0}(\mathcal{V}^*)\) we want to obtain a well defined class in \(\mathrm{KK}^{geo}_0(X,B)\). 
The clutching manifold construction generalized to the family setting by defining the \(\mathrm{Spin}^c\)-manifold \(\Sigma(X/B):=\mathbb{S}(\mathcal{V}^*\oplus\underline{\mathbb{R}})\). This manifold is a sphere bundle over \(X\) with projection \(f:\Sigma(X/B)\to X\). Composing with the bundle projection \(g:=\pi\circ f:\Sigma(X/B)\to B\), we obtain a bundle over \(B\), where each fiber \(\Sigma(X/B)_b:=g^{-1}(\{b\})\) is isomorphic to the clutching manifold \(\Sigma(X_b)\), defined in \cite[Paragraph 10]{BaumDouglas1980}. We can identify the upper and lower hemispheres of \(\Sigma(X/B)\) with the interior of the ball bundle \(\mathbb{B}(\mathcal{V}^*)\) and glue them along the equator \(\mathbb{S}(\mathcal{V^*})\), so we obtain a decomposition: \[\Sigma(X/B)=\mathbb{B}(\mathcal{V}^*)\cup_{\mathbb{S}(\mathcal{V^*})}\mathbb{B}(\mathcal{V}^*)\]

Representing the class \([\sigma]\) by the triple \((\sigma,\pi^*E,\pi^*F)\) we can define the clutching bundle \(E_\sigma\to \Sigma(X/B)\) of \((\sigma,\pi^*E,\pi^*F)\) by restricting \(\pi^*E\) and \(\pi^*F\) to \(\mathbb{B}(\mathcal{V}^*)\) and clutch them along the common boundary by \(\sigma\), i.e. \[E_\sigma:=\pi^*E\cup_\sigma\pi^*F\]

The clutching manifold \(\Sigma(X/B)\) with \(f\) and \(g\) defines a correspondence:
\[\begin{tikzcd}
	& {\Sigma(X/B)} \\
	X && B
	\arrow["f"', from=1-2, to=2-1]
	\arrow["g", from=1-2, to=2-3]
\end{tikzcd}\]

The map \(f\) is K-oriented (see e.g. \cite[Paragraph 10]{BaumDouglas1980}) and so is the map \(X \to B\) because the fibers are contact manifolds. Their composition \(g\) is thus K-oriented. The projection together with the clutching bundle \(E_\sigma\), this defines defines a geometric \(\mathrm{KK}^{geo}_0\) -correspondence \([\Sigma(X/B),E_\sigma,f,g]\in \mathrm{KK}^{geo}_0(X,B)\).

The clutching map \(c\) is defined by \[c:\mathrm{K}^0(\mathcal{V^*})\to \mathrm{KK}_0^{geo}(X,B)\quad [\sigma,\pi^*E,\pi^*F]\mapsto [\Sigma(X/B),E_\sigma,f,g]\]

For the construction of \(b:\mathrm{K}_0(C^*(\mathcal{V_H}))\to \mathrm{KK}^{geo}_0(X,B)\) we will need the two canonical vertical \(\mathrm{Spin}^c\)-fundamental classes of \(\mathcal{V}\), corresponding to  the two co-orientations. We define the map \(b\) as the the following composition after a lift:
\[\begin{tikzcd}
	{\mathrm{K}^0(\mathcal{I_H})\cong\mathrm{K}^1(X)\oplus\mathrm{K^1}(X)} & {\mathrm{KK}^{geo}_0(X,B)\oplus\mathrm{KK}^{geo}_0(X,B)} \\
	{\mathrm{K}_0(C^*(\mathcal{V_H}))} & {\mathrm{KK}^{geo}_0(X,B)}
	\arrow["\Phi", from=1-1, to=1-2]
	\arrow["{+}", from=1-2, to=2-2]
	\arrow[dotted, from=2-1, to=1-1]
\end{tikzcd}\]
Where the map \(\Phi\) is given by \(\Phi=([(X/B)^+]\otimes)\oplus ([(X/B)^-]\otimes)\). Let us give a concrete description of the isomorphism \(\mathrm{K}^0(\mathcal{I_H})\cong\mathrm{K}^1(X)\oplus\mathrm{K^1}(X)\) used in the diagram. Recall that by Morita equivalence, we have an isomorphism
\[\K_0(I_{\mathcal{H}}) \cong \K_0(X \times \mathbb{R}^*_+) \oplus \K_0(X\times \mathbb{R}^*_-).\]

Using the diffeomorphism \(\log:(0;\infty)\to \mathbb{R}\), we obtain the isomorphism in K-theory:
\[\mathrm{K}^0(I_{\mathcal{H}})\cong\mathrm{K}^0(X\times \mathbb{R}^\times)\cong \mathrm{K}^0(X\times \mathbb{R})\oplus \mathrm{K}^0(X\times \mathbb{R})\cong\mathrm{K^1}(X)\oplus\mathrm{K^1}(X).\]

Let \(\tau^+,\tau^-\in \K^0(X)\) be the Thom classes of the complex bundles \(\mathcal{H}^{1,0}\) and \(\mathcal{H}^{0,1}\) respectively. We then get an isomorphism
\[\K^0(X) \overset{\cup \tau^+}{\rightarrow} \K^0(\mathcal{H}^{1,0}) \cong \mathrm{K}^0(\mathcal{H}^*),\]
and likewise for \(\tau_-\).Adding the two classes obtained from one in \(\K_0(I_{\mathcal{H}})\) defines an isomorphism:
\[\K_0(I_{\mathcal{H}}) \to \K^0(\mathcal{H}^*\times \mathbb{R}^{\times}) \cong \K^0(\mathcal{V}^* \setminus\mathcal{H}^*).\]
This map is the inverse of the isomorphism in \(\K\)-theory induced by \(\Psi\). We use these Thom classes to relate the clutching construction \(b\) (which we still have to prove is well defined) and the clutching \(c\). If we consider the cup product with the Thom classes on \(X \times \mathbb{R}\) instead of \(X\), we obtain a map
\[\cup \tau^{\pm} \colon \K^0(X \times \mathbb{R}) \to \K^0(\mathcal{H}^*\times\mathbb{R} ) \cong \K^0(\mathcal{V}^*),\]
using the identification \(\mathcal{H}^*\times\mathbb{R}\cong\mathcal{V}^*\).

 \begin{lem}\label{Lemma: Thom Clutching}
    
     Let \(\tau^+\) denote the Thom class of \(\mathcal{H}^{1,0}\). Then the following diagram commutes:
\[\begin{tikzcd}
	{\mathrm{K}^0(X\times\mathbb{R})} & {\mathrm{K}^0(\mathcal{V^*})} \\
	{\mathrm{K}^1(X)} & {\mathrm{KK}^{geo}_0(X,B)}
	\arrow["{\cup\tau^+}", from=1-1, to=1-2]
	\arrow["\cong"', from=1-1, to=2-1]
	\arrow["c", from=1-2, to=2-2]
	\arrow["{\cap[(X/B)^+]}", from=2-1, to=2-2]
\end{tikzcd}\]
 \end{lem}
\begin{proof}
    In order to prove the commutativity of the diagram, we use \(\tau_\mathcal{V}\) the Thom class of \(\mathcal{V}^* \)with respect to the positive \(\mathrm{Spin}^c\)-structure. It induces a Thom isomorphism for the bundle \(\mathcal{V}^*\to X\).
\[\begin{tikzcd}
	{\mathrm{K}^0(X\times\mathbb{R})} & {\mathrm{K}^0(\mathcal{V^*})} \\
	{\mathrm{K}^1(X)} & {\mathrm{KK}^{geo}_0(X,B)}
	\arrow["{\cup\tau^+}", from=1-1, to=1-2]
	\arrow["\cong"', from=1-1, to=2-1]
	\arrow["c", from=1-2, to=2-2]
	\arrow["{\tau_\mathcal{V}}"{description}, from=2-1, to=1-2]
	\arrow["{\cap[(X/B)^+]}"', from=2-1, to=2-2]
\end{tikzcd}\]
    The upper triangle commutes by naturality of the Thom classes. To show the commutativity of the lower triangle we use the Dirac-dual Dirac method. Denote by \(\beta_{\mathcal{V}^*}\in \mathrm{KK}_1(X,\mathcal{V^*})\) the KK-cycle defined by the Thom element and by \(\alpha_{\mathcal{V}^*}\in \mathrm{KK}_1(\mathcal{V}^*,X)\) the  KK-cycle given by the vertical Dirac operator for the fibers of \(\mathcal{V}^*\). The Kasparov product of these elements yields: \[\beta_{\mathcal{V}^*}\otimes_\mathcal{V^*}\alpha_{\mathcal{V^*}}=\mathrm{id}_X\in \mathrm{KK_0(X,X)}\]
    The total space of \(\mathcal{V}^*\) admits itself a vertical \(\mathrm{Spin^c}\) structure \([\mathcal{V}^*]\), which factorizes. \[[\mathcal{V}^*]={\alpha_{\mathcal{{V}}^*}}\otimes_{\mathcal{{V}^*}}{[(X/B)^+]}\in \mathrm{KK}_0(\mathcal{V}^*,B)\]

    The Kasparov-product with \(\beta_{\mathcal{V}^*}\) therefore yields:

    \begin{align*}
        \beta_{\mathcal{V}^*}\otimes[\mathcal{V}^*]&={\beta_{\mathcal{V}^*}}\otimes_{\mathcal{{V}^*}}{\alpha_{\mathcal{{V}}^*}}\otimes_{X}{[(X/B)^+]} \\
        &=[(X/B)^+]\in \mathrm{KK}_1(X,B)
    \end{align*}
    This equality under the isomorphism \(\mathrm{KK}_*^{geo}(X,B)\cong\mathrm{KK}_*(X,B)\) proves the commutativity of the lower triangle.

\end{proof}

\begin{lem}
    The map \(b \colon \K_0(C^*(\mathcal{V}_{\mathcal{H}})) \to \KK_0^{geo}(X,B)\) is well defined: the element \(b(\xi)\) does not depend on a lift of \(\xi\) to \(\K_0(I_{\mathcal{H}})\).
\end{lem}
\begin{proof}
    Using the exact sequence in K-theory, the kernel of the map \(\K_0(I_{\mathcal{H}}) \to \K_0(C^*(\mathcal{V}_{\mathcal{H}}))\) is given by the image of \(\K^1(\mathcal{H})\). In the corresponding diagram:
    \[\xymatrix{\K_0(I_{\mathcal{H}}) \ar[r]^(.4){\sim} & \K^1(X) \oplus \K^1(X) \\
                \K^1(\mathcal{H}^*) \ar[u] \ar[r]^{\sim} & \K^1(X) \ar[u]}  \]
    where the bottom horizontal arrow is the Thom isomorphism of the bundle \(\mathcal{H}\) (with the positive complex structure, i.e. \(\tau^+\)), the vertical arrow is the embedding given by the matrix \(\begin{pmatrix}
        1 \\ (-1)^n
    \end{pmatrix}\). Indeed the inclusion in the first coordinate is made through \(\tau^+\) and the second one through \(\tau
    ^-\) (see the previous lemma). Since \(\mathcal{H}\) has real rank \(2n\) then \(\tau^+ = (-1)^n \tau^-\). Finally since \([(X/B)^+] = (-1)^{n+1}[(X/B)^-]\) we get that the image of \(\xi \in \K^1(\mathcal{H}^*)\) in \(\KK^{geo}_0(X,B)\) has the following form (writing \(\widetilde{\xi}\) for the corresponding element in \(K^1(X)\)):
    \[[(X/B)^+] \cap \widetilde{\xi} + (-1)^n. (-1)^{n+1} [(X/B)^+] \cap \widetilde{\xi} = 0.\]

    Therefore the map \(\K_0(I_{\mathcal{H}}) \to \KK^{geo}_0(X,B)\) factors through \(\K_0(C^*(\mathcal{V}_{\mathcal{H}}))\) and \(b\) is well defined.
\end{proof}

\begin{thm}\label{Theorem: clutching maps commute}
    The following diagram commutes:

\[\begin{tikzcd}
	& {\mathrm{K}_0(\text{C}^*(\mathcal{V}_\mathcal{H}))} \\
	{\mathrm{KK}^{geo}_0(X,B)} \\
	& {\mathrm{K}^{0}(\mathcal{V^*})}
	\arrow["{b}"', from=1-2, to=2-1]
	\arrow["{\Psi_{(X,B)}}"', from=3-2, to=1-2]
	\arrow["{c}"', from=3-2, to=2-1]
\end{tikzcd}\]
\end{thm}

\begin{proof}

We prove the commutativity of the diagram in terms of diagram chasing.

\[\begin{tikzcd}
	& {\mathrm{K}^0(X\times \mathbb{R}^\times)} & {\mathrm{K}^1(X)\oplus\mathrm{K}^1(X)} \\
	& {\mathrm{K}^0(\mathcal{V}^*\setminus \mathcal{H}^*)} & {\mathrm{KK}_0^{geo}(X,B)\oplus\mathrm{KK}_0^{geo}(X,B)} \\
	& {\mathrm{K}^0(\mathcal{V}^*)} & {\mathrm{KK}_0^{geo}(X,B)} \\
	{\mathrm{K}_0(C^*(\mathcal{V_H}))}
	\arrow["\cong", from=1-2, to=1-3]
	\arrow["{\cup\tau^\pm}"', from=1-2, to=2-2]
	\arrow[shift right=3, curve={height=18pt}, from=1-2, to=4-1]
	\arrow["{\cap[(X/B)^\pm]}", from=1-3, to=2-3]
	\arrow["{i_*}"', from=2-2, to=3-2]
	\arrow["{+}", from=2-3, to=3-3]
	\arrow["c", from=3-2, to=3-3]
	\arrow["\Psi"', from=3-2, to=4-1]
	\arrow["{\mathrm{lift}}", shift left=5, curve={height=-18pt}, dotted, from=4-1, to=1-2]
	\arrow["b"{description}, curve={height=30pt}, from=4-1, to=3-3]
\end{tikzcd}\]

The commutativity of the most left triangle was proven in \ref{Proposition: Lift symbol class}. From \ref{Lemma: Thom Clutching} we have the commutativity of the square for both \(\pm\) signs. The sign of the left vertical isomorphism depends on the orientation of \(X\times \mathbb{R}\), hence the first diagram commutes with the product with reversed orientation fundamental class \([(X/B)^-]\).


\[\begin{tikzcd}
	{\mathrm{K}^0(X\times(0,\infty))} & {\mathrm{K}^0(\mathcal{V^*})} \\
	{\mathrm{K}^1(X)} & {\mathrm{KK}^{geo}_0(X,B)}
	\arrow["{\cup\tau^+}", from=1-1, to=1-2]
	\arrow["\cong"', from=1-1, to=2-1]
	\arrow["c", from=1-2, to=2-2]
	\arrow["{\cap[(X/B)^+]}", from=2-1, to=2-2]
\end{tikzcd}\]
\[\begin{tikzcd}
	{\mathrm{K}^0(X\times(-\infty,0))} & {\mathrm{K}^0(\mathcal{V^*})} \\
	{\mathrm{K}^1(X)} & {\mathrm{KK}^{geo}_0(X,B)}
	\arrow["{\cup\tau^-}", from=1-1, to=1-2]
	\arrow["\cong"', from=1-1, to=2-1]
	\arrow["c", from=1-2, to=2-2]
	\arrow["{\cap[(X/B)^-]}", from=2-1, to=2-2]
\end{tikzcd}\]

Combining the diagrams for each connected component of \(\mathbb{R}^\times\) we obtain the commutativity of large square.
By the definition of \(b\), as composition of the maps above, also the lower triangle commutes.
This finishes the proof.

\end{proof}


\subsection{Family index theorem}\label{Subsection: Index theorem}

The family index theorem for contact fibrations reduces to the usual one using what we have proved thus far.

\begin{thm}[Atiyah-Singer family index]\label{Theorem:Atiyah Singer families}
    Let \(X \to B\) be a compact fibration, let \(\mathcal{V} \to X\) be the vertical bundle. The following diagram commutes:
    \[\begin{tikzcd}
                                     & \K^0(\mathcal{V}^*) \arrow[ld, "c"'] \arrow[rd, "\Op"] &            \\
        {\KK^{geo}_0(X,B)} \arrow[rr, "\mu"] &                                                      & {\KK_0(X,B)}
    \end{tikzcd}\]
\end{thm}   
\begin{proof}
    This is a reformulation of the Atiyah Singer family index \cite{AtiyahSinger1968} in \(\KK\)-theory. The proof is identical as in \cite{BaumvanErp2016Elliptic}, the only difference is in their Thom isomorphism argument with \(X\), \(T^*X\), and \(T^*(T^*X)\). We need to replace the latter two by \(\mathcal{V}^*\) and the dual of the vertical bundle for the fibration \(\mathcal{V}^*\to X \to B\).
\end{proof}

\begin{thm}
    Let \(X \to B\) be a compact contact fibration. Let \(\mathcal{V}_{\mathcal{H}} \to X\) be the Heisenberg vertical bundle. The following diagram commutes:
    \[\begin{tikzcd}
                                     & \K_0(C^*(\mathcal{V}_{\mathcal{H}})) \arrow[ld, "b"'] \arrow[rd, "\Op_H"] &            \\
        {\KK^{geo}_0(X,B)} \arrow[rr, "\mu"] &                                                      & {\KK_0(X,B)}
    \end{tikzcd}\]
    Consequently, if \(P\) is a family of Rockland operators, we have the equality in the group \(\KK_0(X,B)\):
    \[[P] = [\sigma_0^{-1}\circ\sigma_+(P)]\cap[(X/B)^+] + [\sigma_0^{-1}\circ\sigma_+(P)]\cap[(X/B)^+].\]
\end{thm}

\begin{proof}
    The following diagram commutes:
    \[\begin{tikzcd}
	& {\mathrm{K}_*(\text{C}^*(\mathcal{V}_\mathcal{H}))} \\
	\\
	& {\mathrm{K}^{*}(\mathcal{V^*})} \\
	{\mathrm{KK}^{geo}_*(X,B)} && {\mathrm{KK}_*(C(X),C(B))}
	\arrow["b"', curve={height=12pt}, from=1-2, to=4-1]
	\arrow["{\mathrm{Op}}", curve={height=-12pt}, from=1-2, to=4-3]
	\arrow["\Psi"', from=3-2, to=1-2]
	\arrow["c"', from=3-2, to=4-1]
	\arrow[from=3-2, to=4-3]
	\arrow["\mu", from=4-1, to=4-3]
    \end{tikzcd}\]
    Indeed the upper left triangle commutes by Theorem \ref{Theorem: clutching maps commute}. The upper right triangle commutes by Theorem \ref{Theorem: Op behaviour and commutation}. The lower triangle is Theorem \ref{Theorem:Atiyah Singer families}.

    The formula is then the equality \(\mu^{-1}([P]) = b([\sigma_H(P)])\). 
\end{proof}

We can now apply this formula for families of Toeplitz operators, generalizing Boutet de Monvel's result.

\begin{thm}
    Let \(X \to B\) be a compact contact fibration, \(E\to X\) a complex vector bundle. Let \(S\) be a family of (formal) Szegö projections acting on \(E\). Let \(f \colon X \to \mathrm{GL}(E)\) be an everywhere invertible section. Let \(T_f = S\mathcal{M}(f)S\) be the Toeplitz operator associated with \(f\). We have the equality in \(\KK_0(X,B)\):
    \[[P] = [f] \cap [(X/B)^+].\]
    This means that \(P\) defines the same index bundle in \(\K^0(B)\) as the family of Dirac operators \(\mathrm{D}_{\mathcal{V}}^{E_f}\) on \(X\times \mathbb{S}^1\to B\) where \(E_f \to X \times \mathbb{S}^1\) is the bundle obtained by clutching construction. 
\end{thm}
\begin{proof}
    We go back to the previous index theorem. Define \(\widetilde{T}_f = T_f + (1-S) \in \Psi^0_H(X/B;E).\) The principal symbol of \(\widetilde{T}_f\) is given by:
    \begin{align*}
        \sigma_+(\widetilde{T}_f) &= s_0 f s_0 + (1-s_0) \\
        \sigma_-(\widetilde{T}_f) &= 1 \\
        \sigma_0(\widetilde{T}_f) &= 1
    \end{align*}
    Therefore \(\widetilde{T}_f\) is a Rockland family. The \(\KK\)-cycles defined by \(T_f\) and \(\widetilde{T}_f\) are operator homotopic (by \(T_f + \varepsilon(1-S)\)) therefore we get the same class in \(\KK_0(X,B)\) and compute it with the previous theorem.
\end{proof}

Given an element \([P] \in \KK_0(X,B)\), write \(\mathrm{ind}([P]) \in \K^0(B)\) for its push-forward via the map from \(X\) to a point. This is a virtual bundle over \(B\) that encodes the formal difference of vector bundles \(\ker(P)\) and \(\ker(P^*)\) (this makes sense after a perturbation of the operator so that they become vector bundles). The usual Chern characters of vector bundles extends as an isomorphism:
\[\mathrm{ch} \colon K^{0/1}(B)\otimes \mathbb{Q} \to \mathrm{H}^{even/odd}(B;\mathbb{Q}).\]
Using the geometric description of the classes in \(\KK_0(X,B)\), we may describe the Chern character of the index bundle of \([P]\), using the family index formula due to Atiyah and Singer \cite{AtiyahSinger1968}.

\begin{thm}[\cite{AtiyahSinger1968}] Let \(X \to B\) be a fibration, \(E \to X\) a complex vector bundle. The Chern character of the index bundle of the twisted vertical Dirac operator is:
\[\mathrm{ch}(\mathrm{ind}([\mathrm{D}_{\mathcal{V}}^E])) = \int_{X/B} \mathrm{Td}(X/B)\mathrm{ch}(E).\]
In this formula, \(\mathrm{Td}(X/B)\) is the Todd class of the vertical bundle and 
\[\int_{X/B} \colon \mathrm{H}^*(X;\mathbb{Q}) \to \mathrm{H}^{*-N}(B;\mathbb{Q})\] 
is the integration along the fibers (of which \(N\) is the dimension).
\end{thm}

The following formulas are then deduced from this theorem using standard manipulations between the characteristic classes (see e.g. \cite{LawsonMichelson1989,BerlineGetzlerVergne1992}).

\begin{cor}
    Let \(X\to B\) be a contact fibration, \(E,F \to X\) complex vector bundles and \(P \colon \Psi^k(X/B;E,F)\) a family of Rockland operators. We have the formula in \(\mathrm{H}^{*}(B;\mathbb{Q)}\):
    \begin{align*}
        \mathrm{ch}(\mathrm{ind}([P])) &= \int_{X/B}\mathrm{ch}([\sigma_0^{-1}\circ\sigma_+(P)])\mathrm{Td}(X/B) \\
        &+ (-1)^{n+1}\int_{X/B}\mathrm{ch}([\sigma_0^{-1}\circ\sigma_-(P)])\mathrm{Td}(X/B)
    \end{align*}
    where \(\mathrm{Td}(X/B)\) is the Todd class of the vertical bundle with the \(\mathrm{Spin}^c\)-structure corresponding to the orientation of the contact forms on the fibers.
\end{cor}

The \((-1)^{n+1}\) comes from the fact that the Todd class of the vertical bundle with the opposite \(\mathrm{Spin}^c\)-structure is equal to the first one times this sign.

\begin{cor}\label{Corollary: Toeplitz index computation}
    Let \(X\to B, E\) and \(f\) be as previously. The Chern character of the index bundle of \(T_f\) is given by:
    \[\mathrm{ch}(\mathrm{ind}([T_f])) = \int_{X/B}\mathrm{ch}([f])\mathrm{Td}(X/B).\]
\end{cor}

Another class of operators, we can compute the family index for are the operators \(P_\gamma\) defined by Baum and van Erp. We will briefly adapt their definition to our setting and then provide the family index theorem.

Let \((X\to B,\mathcal{H})\) be a contact fibration endowed with a one form \(\Theta\) and the associated Reeb vector field \(R\). Let \((E\to X,h,\nabla^E)\) be a hermitian vector bundle and \(\gamma\in \Gamma^\infty(X,\mathrm{End}(E))\).

Let \(X_1,\cdots,X_{2n}\) be an orthonormal frame for \(\mathcal{H}_\mathcal{V}\). We can define up to Heisenberg order one terms the vertical sublaplacian:
\[\Delta_{H}^E:=-\sum_{i=1}^{2n}(\nabla^E_{X_i})^2 +\mathrm{l.o.t}.\]

The Baum-van Erp operator is also defined up to lower order terms by 
\[P_\gamma:=\Delta^E_H+\mathrm{i}\gamma\nabla^E_R+\mathrm{l.o.t}.\]
Both terms appearing in the operator have Heisenberg order two. The positive and negative symbol of \(\Delta^E_H\) are quantum harmonic oscillators, acting on \[\mathcal{F}^+(\mathcal{H}^{1,0})=l^2-\bigoplus_k\mathrm{Sym}^k(\mathcal{H}^{1,0})\quad \mathcal{F}^+(\mathcal{H}^{0,1})=l^2-\bigoplus_k\mathrm{Sym}^k(\mathcal{H}^{0,1}) \]
respectively. When restricted to \(\mathrm{Sym}^k(\mathcal{H}^{1,0})\) the quantum harmonic oscillator acts by multiplication of \(2k+n\). The symbol of the Reeb vector field acts by either \(\mathrm{i}\) for the positive symbol and \(-\mathrm{i}\) for the negative. Therefore, we obtain for the positive and negative symbol 
\begin{align*}
    \sigma_+(P_\gamma)|_{\mathrm{Sym}^k(\mathcal{H}^{1,0})\otimes E}&=(2k+n)\mathrm{id}_E-\gamma\\
    \sigma_-(P_\gamma)|_{\mathrm{Sym}^k(\mathcal{H}^{0,1})\otimes E}&=(2k+n)\mathrm{id}_E+\gamma
\end{align*}

Consequently, \(P_\gamma\) is Rockland if and only if \(\gamma(x)-\lambda\mathrm{id}_E\) is invertible for every \(\lambda\in \{\pm(2k+n), k\in \mathbb{N}\}\).
In that case, we can define two families of bundle automorphisms \[\phi^k_+:=(2k+n)\mathrm{id}_E-\gamma \quad \phi^k_-:=(2k+n)\mathrm{id}_E+\gamma\]
Each of them determines a class in \(\mathrm{K}^1(X)\). We can consider the cup product with each tensor power and sum them up. 
\[\sum_{k\geq0}[\mathrm{Sym}^k(\mathcal{H}^{1,0})]\cup[\phi^k_+]\quad \sum_{k\geq0}\mathrm{Sym}^k(\mathcal{H}^{0,1})]\cup[\phi^k_-] .\]
Indeed, both sums are actually finite in \(\mathrm{K}^1(X)\), since \(\|\gamma\|_\infty<\infty\) for any choice of metric and if \(2k+n\geq\|\gamma\|_\infty\), then \(\Phi_\pm^k(t):=(2k+n)\mathrm{id}_E\mp t\gamma\) is a homotopy from \(\phi^k_\pm\) 
to an invertible scalar automorphism.
We therefore obtain a family generalization  of the theorem of Baum and van Erp.

\begin{thm}
    Assume \(P_\gamma\) is a Rockland operator. The KK-class represented by \(P_\gamma\) is 
    \begin{align*}[P_\gamma]&=\sum_{k\geq0}\left([\mathrm{Sym}^k(\mathcal{H}^{1,0})]\cup[\phi^k_+]\right)\cap[(X/B)^+]\\
    &+ \sum_{k\geq0}\left(\mathrm{Sym}^k(\mathcal{H}^{0,1})]\cup[\phi^k_-]\right)\cap[(X/B)^-].
    \end{align*}
\end{thm}
\appendix

\section{Forms on the fibered pair groupoid}\label{Appendix}

In this short appendix we characterize multiplicative forms on the fibered pairs groupoid \(S\times_BS\rightrightarrows S\). Recall that on a groupoid \(\mathcal{G}\), a differential form \(\omega \in \Omega^k(\mathcal{G})\) is multiplicative if on the space of composable pairs \(\mathcal{G}^{(2)}= \mathcal{G} \ {}_r\!\times_s \mathcal{G}\), we have the relation:
\[m^*\omega = \mathrm{pr}_1^*\omega + \mathrm{pr}_2\omega.\]
In the case of the fibered pair groupoid, this translates to the following. A form \(\omega\in \Omega^k(S\times_BS)\) is multiplicative if for every \(x,y,z \in S\) on the same fiber, we have:
\[\omega_{(x,z)} = \omega_{(x,y)} + \omega_{(y,z)}.\]
\begin{thm}\label{Theorem: Multiplicative forms fibered pairs}
    A form \(\omega\in \Omega^k(S\times_BS)\) is multiplicative, if and only if there exists an \(\eta\in \Omega^k(S)\), such that \(\omega=r^*\eta-s^*\eta\). Moreover, the form \(\eta\) is unique up to a pullback, i.e. if \(\eta'\) is another form, such that \(\omega=r^*\eta'-s^*\eta'\), then \(\eta-\eta'=p^*\beta\) for some \(\beta\in \Omega^k(B)\). Hence, the sequence
\[\begin{tikzcd}
	{0} & {\Omega^k(B)} & {\Omega^k(S)} & {\Omega_{mult}^k(S\times_BS)} & {0}
    \arrow[from=1-1, to=1-2]
	\arrow["{p^*}", from=1-2, to=1-3]
	\arrow["{r^*-s^*}", from=1-3, to=1-4]
    \arrow[from=1-4, to=1-5]
\end{tikzcd}\]
is exact.
\end{thm}
\begin{proof}
    A form of the type \(r^*\eta-s^*\eta\) is clearly multiplicative. To prove the converse, let \(\omega\in \Omega^k(S\times_BS)\). Since \(p\colon S\to B\) is a surjective submersion, around every point \(b\in B\) there is a an open neighborhood \(U\subset B\), such that there are local sections, i.e. \(\sigma\colon U\to S\) with \(p\circ \sigma=\mathrm{id}_U\). We denote \(p^{-1}(U)=S|_U\) and define the map \[\tilde{\sigma}_1\colon S|_U\to S\times_BS \quad x\mapsto (x,\sigma(p(x))).\]
    This map is well-defined, since \(p(\sigma(p(x)))=p(x)\). The composition with the structure maps provides:
    \[r\circ\tilde{\sigma}_1=i_{S|_{U_i}} \quad s\circ \tilde{\sigma}_1=\sigma\circ p|_{S|_{U_i}}.\]
    
    Given a multiplicative form \(\omega\in \Omega^k(S\times_BS)\), we want to show we will first show, that we locally find such an \(\eta\).
    Define the map \[\tilde{\sigma}_2\colon S|_U\times_US|_U\to S\times_BS\times_BS \quad (x,y)\mapsto (x,y,\sigma(p(x))).\]
    This map is also well-defined, since \(p(x)=p(y)=p(\sigma(p(x)))\).
    We have the following compositions of maps \[\mathrm{pr}_1\circ\tilde{\sigma}_2=\mathrm{id} \quad \mathrm{pr}_2\circ\tilde{\sigma}_2=\tilde{\sigma}_1\circ s \quad m\circ\tilde{\sigma}_2=\tilde{\sigma}_1\circ r\]
    Since \(\omega\) was multiplicative, it satisfies \(m^*\omega=\mathrm{pr}^*_1\omega+\mathrm{pr}^*_2\omega\). Pulling this equation back \(\tilde{\sigma}_2\), we obtain:

    \begin{align*}
        0&=\tilde{\sigma}_2^*(m^*\omega-\mathrm{pr}^*_1\omega+\mathrm{pr}^*_2\omega)\\
        &=\tilde{\sigma}_2^*m^*\omega-\tilde{\sigma}_2^*\mathrm{pr}^*_1\omega+\tilde{\sigma}_2^*\mathrm{pr}^*_2\omega\\
        &=\omega|_{S|_U\times_US|_U}-(r^*\tilde{\sigma}_2^*\omega-s^*\tilde{\sigma}_2^*\omega)
    \end{align*}
Hence, locally over \(S|_U\times_US|_U\) we always find a form \(\eta_U:=\tilde{\sigma}^*_2\omega\), which has the desired properties. 
Let \(\{U_i\}_{i\in I}\) be an open cover of \(B\) admitting smooth local sections \(\sigma_i\colon U_i\to S\) and let \(\{\chi_i\}_{i\in I}\) be a partition of unity subordinate to \(\{U_i\}_{i\in I}\). We denote \(S_i:=p^{-1}(U_i)\) and \(\eta_i:=\tilde{\sigma}_{2,i}^*\omega\in \Omega^k(S_i)\). These forms patch together to a globally defined \(k\)-form \(\eta:=\sum_{i\in I}p^*\chi_i\eta_i\in \Omega^k(S)\). Over \(S_i\times_{U_i}S_i\), we have 
\begin{align*}
    r^*(p^*\chi_i\eta_i)-s^*(p^*\chi_i\eta_i)&=(p\circ r)^*\chi_i r^*\eta_i-(p\circ s)^*\chi_i s^*\eta_i\\
    &=q^*\chi_i(r^*\eta_i-s^*\eta_i)\\
    &=q^*\chi_i\omega
\end{align*}
Summing over the enumeration of the cover, we obtain the first statement. We denote \(q:=p\circ r=p\circ s\).

\begin{align*}
    r^*\eta-s^*\eta
    &=\sum_{i\in I}q^*\chi_i(r^*\eta_i-s^*\eta_i)\\
    &=\sum_{i\in I}q^*\chi_i\omega\\
    &=\omega
\end{align*}

Now suppose that there is \(\eta'\in \Omega^k(S)\), such that \(r^*\eta-s^*\eta=r^*\eta'-s^*\eta'\). Defining \(\theta:=\eta-\eta'\), the form satisfies \(r^*\theta=s^*\theta\) in \(\Omega^k(S\times_BS)\).
We want to show, that this implies that \(\theta=p^*\beta\) for \(\beta\in \Omega^k(B)\). 
Given the open cover \(\{U_i\}_{i\in I}\) and the sections \(\tilde{\sigma}_{1,i}\colon S|_{U_i}\to S\times_BS\) we get 
\[\theta|_{S_{U_i}}=(r\circ\tilde\sigma_{1,i})^*\theta=\tilde\sigma_{1,i}^*r^*\theta=\tilde\sigma_{1,i}^*s^*\theta=(s\circ\tilde\sigma_{1,i})^*\theta=(\sigma_i\circ p|_{S_{U_i}})^*\theta=p^*_i\sigma_i^*\theta.\]
We define \(\beta_i:=\sigma^*_i\theta\in \Omega^k(U_i)\). The forms \(\{\beta_i\}_{i\in I}\) patch to a well defined \(k\)-form \(\beta\in \Omega^k(B)\). To see this, define on \(U_{ij}=U_i\cap U_j\) the map \[\tilde\sigma_{ij}\colon U_{ij}\to S\times_BS \quad b\mapsto (\sigma_i(b),\sigma_j(b)).\]
This map is also well defined, since \(p(\sigma_i(b))=b=p(\sigma_j(b))\). The composition with the structure maps gives \[r\circ \tilde\sigma_{ij}=\sigma_i|_{U_{ij}}\quad s\circ \tilde\sigma_{ij}=\sigma_j|_{U_{ij}}.\]
Therefore, pulling back \(r^*\theta=s^*\theta\) by \(\tilde\sigma_{ij}\) leads to \[\beta_i|_{U_{ij}}=\sigma_i^*\theta|_{U_{ij}}=(r\circ\tilde\sigma_{ij})^*\theta=\tilde\sigma_{ij}^*r^*\theta=\tilde\sigma_{ij}^*s^*\theta=(s\circ\tilde\sigma_{ij})^*\theta=\sigma_j^*\theta|_{U_{ij}}=\beta_j|_{U_{ij}}.\]
Therefore, there is a globally defined smooth form \(\beta\in \Omega^k(B)\), such that \(\beta|_{U_i}=\beta_i\) and \(\theta=p^*\beta\). Since \(p^*\) is injective, \(\beta\) is unique.
\end{proof}

\bibliographystyle{plain}
\bibliography{bibliography.bib}

\end{document}